\documentclass[11pt,a4paper,reqno]{amsart} 
\usepackage{anysize}
\marginsize{3.5cm}{3.5cm}{2.2cm}{2.2cm}
\usepackage{lineno}
\usepackage{hyperref} % [hypertex]
\hypersetup{
    colorlinks=true,       
    linkcolor=red,          
    citecolor=blue,        
    filecolor=magenta,      
    urlcolor=cyan           
}
\usepackage{standalone}
\usepackage[english]{babel}
\usepackage{amsmath}
\usepackage{amsfonts,amssymb}
\usepackage{enumerate}
\usepackage{mathrsfs}
\usepackage{amsthm}
\usepackage{tikz}
\usepackage{tikz}

\usetikzlibrary{intersections,decorations.markings}
\usetikzlibrary{calc}

\newfont{\bb}{msbm10 at 12pt}
\newfont{\bbt}{msbm10 at 9pt}
\def\r{\mathbb R}
\def\d{\mathbb D}
\def\n{\mathbb N}
\def\z{\mathbb Z}
\def\t{\mathbb T}
\def\c{\hbox{\bb C}}

\def\b{\hbox{B}}

\def\fr{\mathscr F}
\def\sr{\mathscr S}

\def\w{\mathscr W}
\def\er{\mathscr E}

\DeclareMathOperator{\IM}{Im}
\DeclareMathOperator{\RE}{Re}
\DeclareMathOperator{\supp}{supp}
\DeclareMathOperator{\D}{\langle D \rangle}
\DeclareMathOperator{\DIV}{div}

\DeclareMathOperator{\op}{Op}

\newcommand{\dps}{\displaystyle}
\newcommand{\norm}[1]{\left\Vert #1 \right\Vert}
\newcommand{\abs}[1]{\left\vert #1 \right\vert}
\newcommand{\set}[1]{\left\{#1\right\}}
\newcommand{\eps}{\epsilon}

\usepackage{hyperref}
\hypersetup{
    colorlinks=true,
    linkcolor=red,
    filecolor=red,      
    urlcolor=blue,
    citecolor=red,
    pdftitle={A geometric proof of the Quasi-linearity of the water-waves system},
    bookmarks=true,
    pdfpagemode=FullScreen,
}

\numberwithin{equation} {section}

\theoremstyle{plain}\newtheorem{lemma}{Lemma}[section]
\theoremstyle{plain}\newtheorem{proposition}{Proposition}[section]
\theoremstyle{plain}\newtheorem{theorem}{Theorem}[section]
\theoremstyle{plain}\newtheorem{definition}{Definition}[section]
\theoremstyle{plain}\newtheorem{remark}{Remark}[section]
\theoremstyle{plain}\newtheorem{corollary}{Corollary}[section]
\theoremstyle{plain}
\theoremstyle{plain}\newtheorem{notation}{Notation}[section]
\theoremstyle{plain}\newtheorem{definition-notation}{Definition-Notation}[section]
\theoremstyle{plain}\newtheorem{definition-proposition}{Definition-Proposition}[section]
\numberwithin{equation}{section}

\newcommand{\keyword}[1]{\textbf{\textit{Keywords---}} \textbf{#1}}

\title{A geometric proof of the Quasi-linearity of the water-waves system}%: \[\dps \partial_t u+Re(u)\partial_xu+i\partial_x \abs{D}^{\alpha-1} u=0,\]
\author{Ayman Rimah Said$\mbox{}^\dag$}

\def\notina[#1]#2{\begingroup\def\thefootnote{\fnsymbol{footnote}}\footnote[#1]{#2}\endgroup}

\begin{document}
%\linenumbers
\iffalse
\begin{titlepage}
   \begin{center}
       \vspace*{5cm}

       \textbf{A GEOMETRIC PROOF OF THE QUASI-LINEARITY OF THE
WATER-WAVES SYSTEM}

 \vspace{1.5cm}
	
       \textbf{Ayman Rimah Said}

       \begin{align*}          
       &\text{Institut de Math\'ematique d'Orsay}  &\text{Centre Borelli} \\
       &\text{Paris Sud University}  &\text{ENS Paris Saclay}
         \end{align*}  
             \vspace{1.5cm}
       \subsection*{Acknowledgement}
I would like to express my sincere gratitude to my thesis advisor Thomas Alazard. 
I would also like to thank Claude Zuily for his insightful note on the Burgers equation that helped me understand the problem.   
   \end{center}
\end{titlepage}
\fi

\notina[0]{$\mbox{}^\dag$ PhD student at l'IMO, Paris Sud University and Centre Borelli, ENS Paris Saclay. Email: \url{aymanrimah@gmail.com}.}

\begin{abstract}
In the first part of this paper we prove that the flow associated to the Burgers equation with a non local term
 of the form $\partial_x \abs{D}^{\alpha-1} u$ fails to be uniformly continuous from bounded sets of $H^s(\d)$ to $C^0([0,T],H^s(\d))$ for $T>0$, $s>\frac{1}{2}+2$, $0\leq \alpha <2$, $\d=\r \ \text{or} \ \t $ and $H$ is the Hilbert transform. Furthermore we show that the flow cannot be $C^1$ from bounded sets of $H^s(\d)$ to $C^0([0,T],H^{s-1+(\alpha-1)^+ +\eps}(\d))$ for $\eps>0$.
We generalize this result to a large class of nonlinear transport-dispersive equations in any dimension, that in particular 
   contains the Whitham equation and the paralinearization of the water waves system with and without surface tension. The current result is optimal in the sense that for $\alpha=2$ and $\d=\t$ the flow associated to the Benjamin-Ono equation is Lipschitz on function with $0$ mean value $H^s_0$.
   
    In the second part of this paper we apply this method to deduce the quasi-linearity of the water waves system, which is the main result of this paper.
    
 \keyword{Flow map, Regularity, Quasi-linear, nonlinear Burgers type hyperbolic equations, nonlinear Burgers type dispersive equations, Water Waves system.}
\end{abstract}

\maketitle

%{\color{red}{TODO} : add assumption $s\in 3\xN$ in the statements. }%%%%%%%%%%%%%55%%

\vspace{-5mm}

\tableofcontents

\vspace{-7mm}

\section{Introduction}%%%eq,th2,lem,rem,cor2

A commonly found definition is that a partial differential equation is said to be quasi-linear if it is linear with respect to all the highest order derivatives of the unknown function, for example equations of the form:
\[\partial_t u+\sum A_j(u)\partial_ju=F(u).\]
Which we compare to the definition of semi-linearity as a partial differential equation whose highest order terms are linear, for example equations of the form:
 \[\partial_t u+\sum A_j\partial_ju=F(u).\]
This distinction is supposed to classify the equations in accordance to how one solves their respective Cauchy problems. For example, semi-linear equations are expected to be solved locally by a Picard iteration scheme and thus the associated flow map is expected to depend regularly on the data. On the other hand quasi-linear equations are expected to be solved by a compactness method and no more information than continuity can be recovered on the flow map. The problem of those broad definitions with the count of derivatives is that they fail to classify the equations according to this simple criteria on their Cauchy problem. Indeed by those definitions the \eqref{geometric proof quasi WW_introduction_KPI} and  \eqref{geometric proof quasi WW_introduction_KPII} equations are semi-linear by the count of the derivatives, indeed they are given by:
\begin{align}
(u_t+uu_x+u_{xxx})_x+u_{yy}&=0 \tag{KPI} \label{geometric proof quasi WW_introduction_KPI},\\
(u_t+uu_x+u_{xxx})_x-u_{yy}&=0 \tag{KPII}  \label{geometric proof quasi WW_introduction_KPII}.
\end{align}

Bourgain showed in \cite{Bourgain93} that  \eqref{geometric proof quasi WW_introduction_KPII}  can be solved by an iteration scheme and that the flow map is regular. But Moulinet, Saut and Tzvetkov showed in \cite{Saut02} that the flow map associated to \eqref{geometric proof quasi WW_introduction_KPI} cannot be $C^2$ and that it cannot be solved by a Picard iteration scheme. This was improved upon in \cite{Koch08} where the flow map was shown not to be uniformly continuous on bounded sets of Sobolev spaces. Motivated by \cite{Saut02}, the authors introduced the following definitions to quasi-linearity and semi-linearity that we will use here:
\begin{itemize}
\item A partial differential equation is said to be semi-linear if its flow map is regular (at least $C^1$).
\item A partial differential equation is said to be quasi-linear if its flow map is not $C^1$.
\end{itemize}

It is well known that the flow map associated to the Burgers equation:
\[\partial_t u+u\partial_x u=0 \text{ on } \r, \]
 fails to be uniformly continuous, giving the equation its quasi-linear nature, as for example shown in \cite{Saut13}.
  An important class of equations that arises in the study of asymptotic models of the water waves equations is Burgers type equation with a dispersive term, 
  for example the Benjamin-Ono equation:
  \[\partial_t u+u\partial_x u+H \partial^2_xu=0 \text{ on }  \r, \tag{BO} \label{geometric proof quasi WW_introduction_BO} \]
   and Korteweg-de Vries equation:
   \[\partial_t u+u\partial_x u+ \partial^3_xu=0 \text{ on } \r. \tag{KdV} \label{geometric proof quasi WW_introduction_KdV}\] 
   
  It was also shown in \cite{Koch05}, that the flow map associated to the Benjamin-Ono equation on $H^s(\r), s>\frac{3}{2}$ fails to be uniformly continuous. 
The proof relies heavily on the dimension, the structure of the equation and on some interactions between small and high frequencies thus it does not generalize to the case of $\t$.
More generally in \cite{Molinet01} (see also \cite{Saut13}), it is shown that the flow map fails to be $C^2$ (thus the equations are unsolvable by a Picard fixed point scheme) for equations of the form:
\[\partial_t u+u\partial_x u+\omega(D)\partial_x u=0,  \text{ with }  \abs{\omega(\xi)}\leq \abs{\xi}^\gamma,  \gamma<2.\]
Here the proof relies heavily on the Duhamel formula, on the explicit solvability of the linear part using the Fourier transform 
and again on some interactions between small and high frequencies thus it does not generalize to the case of $\t$.\\ 
 
 In \cite{Saut13}, for the KdV equation, using Strichartz type dispersive estimates the Cauchy problem is solved by a Picard fixed point scheme 
 and thus the flow map is regular, showing a change in nature for the problem. 
 This shows that an interesting phenomena happening where the dispersive term can dominate the nonlinearity. 
On $\r$, the previous examples show that this change of regime happens for a dispersive term of order 3. 
 Thus the result obtained in \cite{Saut13} is optimal in $d=1$.\\

In this paper we improve these results in several directions:
\begin{itemize}
\item we prove the result for a generic dispersive perturbation of order $\alpha<2$,
\item we prove the strongest result possible by proving that the flow is not uniformly continuous,
\item for $\eps>0$ we prove that the flow cannot be $C^1$ from $H^s(\d)$ to \\ $C^0([0,T],H^{s-1+(\alpha-1)^+ +\eps}(\d))),$ where $a^+=\max(a,0)$ for a real number $a$.
\item we prove the result in any dimension for a generalized system of equations covering both the dispersive generalised KdV equation and the system obtained after paralinearisation and symmetrisation of the water waves system.
\end{itemize}
  For the sake of clarity we begin by stating a result in dimension 1.
\begin{theorem}\label{geometric proof quasi WW_introduction_ theorem on dispersive Burgers eq} %%%%%%%%%%%ref%%%%%%%%%%

Consider three real numbers $\alpha\in [0,2[$, $s\in]2+\frac{1}{2},+\infty[$, $r>0$ and $u_0 \in H^{s}(\d)$. Then there exists $T>0$ such that for all $v_0$ in the ball $\b(u_0,r)\subset  H^{s}(\d)$ there exists a unique $v\in C([0,T],H^s(\d))$ solving the Cauchy problem:
\begin{equation}\label{geometric proof quasi WW_introduction_ theorem on dispersive Burgers eq_eq dispersive Burgers}
\begin{cases} 
\partial_t v+v\partial_xv+\partial_x \abs{D}^{\alpha-1} v=0 \\
v(0,\cdot)=v_0(\cdot),
\end{cases}
\end{equation}
where,
\[\abs{D}=\op(\xi).\]

Moreover for all $R>0$, the flow map:
	\begin{align*}
										\b(0,R) \rightarrow &C([0,T],H^s(\d))\\
										v_0 \mapsto &v
										\end{align*}
				is not uniformly continuous.\\
				    
Considering a weaker control norm we get, for all $\eps'>0$ the flow map:
	\begin{align*}
										\b(0,R) \rightarrow &C([0,T],H^{s-1+(\alpha-1)^+ +\eps'}(\d))\\
										v_0 \mapsto &v
										\end{align*}
				is not $C^1$.
\end{theorem}	

The restriction $s>2+\frac{1}{2}$ is essentially a technical one that comes up when handling the dispersive term for $\alpha>1$. Indeed looking closely through the proof, more specifically at estimate \eqref{geometric proof quasi WW_study model pbm_subsec proof key lem point 2_eq2} we see that the exact hypothesis needed for the proof to hold is $s>\max(\frac{3}{2},\alpha+\frac{1}{2})$, which coincides with the threshold of well-posedness of $\frac{3}{2}$ for $\alpha\leq 1$. In the case of $\d=\t$, the non uniform continuity of the flow on Sobolev spaces can be deduced in a more straightforward manner, see the appendix of \cite{Molinet07}, using the following symmetry, if $u$ is a solution to equation then so is $u(t,x+wt)+w$ for $w$ in $\r$. The proof proposed here where we carefully analyse the transport part of the equation can be simply modified to cover the case where the previous symmetry no longer holds. Indeed if we restrict ourselves to initial data to the hyperplane of mean-value zero functions, which is a property propagated by the flow, the proof then holds the same but with the choice of the base function for the ansatz $\omega \in C_0^\infty$ be made in such a way that $\omega$ has mean-value zero.

\begin{remark}
In our work \cite{Ayman19} we improve the previous result by showing that on the Torus, the loss of $2-\alpha$ derivative is sufficient to have Lipschitz control over the flow map, under a zero mean value hypothesis.
\end{remark}

We shall prove a stronger result (see Theorem \ref{geometric proof quasi WW_a tech gen_main theorem}) showing that for a dispersive perturbation of order $\alpha<2$,
 the non-linear transport term dominates the flow's evolution locally and this happens independently of the dimension. 
This limited regularity of the flow implies that the Cauchy problem can not be solved by a Picard fixed point scheme and thus those equations are quasi-linear.

\begin{itemize}
 \item The results in \cite{Saut13} suggest that the result obtained here are sub-optimal because it suggests that the change to the semi-linear type equations happens for $\alpha=3$, and the flow associated to the Benjamin-Ono equation on $\r$ fails to be uniformly continuous as shown in \cite{Koch05}. In \cite{Ayman19} we show that the flow map associated to the following equation
\[\partial_t v+\RE(v)\partial_xv+i\partial_x^2 v=0,\] 
    is Lipschitz from bounded sets of $H^s(\r)$ to $C^0([0,T],H^{s}(\r))$ under the extra hypothesis of $L^1$ control on the data. Showing that the lack of regularity obtained in \cite{Saut13} for $\alpha \geq 2$ is essentially due to the lack of control of the $L^1$ norms in Sobolev spaces on $\r$.
    
   \item This optimality is also confirmed by the results in \cite{Molinet08} where L. Molinet proves that the flow map has Lipschitz regularity for the Benjamin-Ono equation on the torus in $H^s_0(\t)$ for $ s\geq 0$, which are the Sobolev spaces of functions with zero mean value. 
%This same result tells us that for $\eps>0$ the flow cannot be $$C^1(H^s(\t),C^0([0,T],H^{s-1+(\alpha-1)^+ +\eps}_0(\t))),$$ which shows that our result is optimal for $\alpha=2$.

    \item In our work \cite{Ayman19} we generalize the result on the Benjamin-Ono equation and prove that the flow map associated to the Burgers equation with a non local term of the form $D^{\alpha-1} \partial_x  u$, $\alpha \in ]1,+\infty[$ is Lipschitz from bounded sets of $H^s_0(\t;\r)$ to $C^0([0,T],H^{s-(2-\alpha)^+}_0(\t;\r)),s>1+\frac{1}{2}$. Thus proving that the result obtained here is optimal for $\alpha \in ]1,2]$. Moreover we investigate the effect of the low frequency component and show that for $\alpha \in [0,+\infty[$ the flow map is not Lipschitz from bounded sets of $H^s(\t;\r)$ to $C^0([0,T],H^{s-(2-\alpha)^+\eps}(\t;\r)),\eps>0$. 
    
    This agrees with the results from \cite{Saut13} on $\r$ for $\alpha < 3$ but show that the KdV is quasi-linear for initial data in $H^s(\t;\r)$, $s>\frac{3}{2}$ which is due to the lack of dispersive estimates on $\t$. 
    
    It's important to note that those results agree with Bourgain's results on the well posedness for the periodic Kdv equation in \cite{Bourgain93kdv} and Molinet's results in in \cite{Molinet08}. Indeed in \cite{Bourgain93kdv} the contraction method is applied on initial data in $H^s_0$ and then a gauge transform is used to deduce well posedness for general data. It's exactly this gauge transform that we use to prove the lack of regularity of the flow map when passing from $H^s_0$ to $H^s$.
 
    \end{itemize}

    Finally Theorem \ref{geometric proof quasi WW_a tech gen_main theorem} contains applications to different classes of equations:

-Firstly the Whitham equation on $\r$:
\[
\begin{cases}
\partial_t u +u\partial_x u-Lu_x=0,\\
Lf(x)=\int e^{i x\cdot \xi}p(x,\xi)\hat{f}(\xi) d\xi,
\end{cases}
\]
is quasi-linear for $p \in S^\alpha, \alpha<1$ and such that $\IM(p) \in S^0$ (See \eqref{paracomposition_Notions of microlocal analysis_Pseudodifferential Calculus_def symbol} for the definition of the symbol classes).
\\

-The second and main application is the water waves system with and without surface tension. We follow here the presentation in \cite{Alazard11} and \cite{Alazard14}.
\subsection{Assumptions on the domain}
 We consider a domain with free boundary, of the form:
 \[\set{(t,x,y) \in [0,T]\times \r^d \times \r:(x,y)\in \Omega_t},\]
 where $\Omega_t$ is the domain located between a free surface
 \[\Sigma_t=\set{(x,y)\in \r^d\times \r:y=\eta(t,x)}\]
 and a given (general) bottom denoted by $\Gamma = \partial \Omega_t \setminus \Sigma_t$. More precisely we assume that initially $(t=0)$ we have the hypothesis $H_t$ given by:
\begin{itemize}
 \item \label{geometric proof quasi WW_introduction_ subsection assupmtions on the domain_$H_t$}The domain $\Omega_t$ is the intersection of the half space, denoted by $\Omega_{1,t}$, located below the free surface $\Sigma_t$,
 \[\Omega_{1,t}=\set{(x,y)\in \r^d\times \r:y<\eta(t,x)}\]
 and an open set $\Omega_2 \subset \r^{d+1}$ such that $\Omega_2$ contains a fixed strip around $\Sigma_t$, which means that there exists $h>0$ such that,
 \[\set{(x,y)\in \r^d \times \r: \eta(t,x)-h\leq y \leq \eta(t,x)}\subset \Omega_2.\]
 We shall assume that the domain $\Omega_2$ (and hence the domain $\Omega_t=\Omega_{1,t} \cap \Omega_2$) is connected.
 \end{itemize}
 
\subsection{The equations}
 We consider an incompressible inviscid liquid, having unit density. The equations of motion are given by the Euler system on the velocity field $v$: 
 \begin{equation}\label{geometric proof quasi WW_introduction_ subsection the equations_Euler eq}
 \begin{cases}
 \partial_t v+v\cdot \nabla v+\nabla P=-ge_y \\
  \DIV v=0
\end{cases} 
  \text{ in }  \Omega_t,
 \end{equation}
 where $-ge_y$ is the acceleration of gravity $(g>0)$ and where the pressure term $P$ can be recovered from the velocity by solving an elliptic equation. The problem is then coupled with 
 the boundary conditions:
 \begin{align}\label{geometric proof quasi WW_introduction_ subsection the equations_boundary cond 1}
 \begin{cases}
 v\cdot n=0 &  \text{on }  \Gamma, \\
 \partial_t \eta=\sqrt{1+|\nabla \eta|^2}v \cdot \nu &  \text{on }  \Sigma_t,\\
 P=-\kappa H(\eta)  &  \text{on }  \Sigma_t,
 \end{cases}
 \end{align}
 where $n$ and $\nu$ are the exterior normals to the bottom $\Gamma$ and the free surface $\Sigma_t$, $\kappa$ is the surface tension and $H(\eta) $ is the mean curvature of the free surface:
 \[
 H(\eta)=\DIV \bigg(\frac{\nabla \eta}{\sqrt{1+|\nabla \eta|^2}}\bigg)
 .\]

  We take $\kappa=1$ for the case with surface tension and $\kappa=0$ in the case of gravity water waves (without surface tension). The first condition in \eqref{geometric proof quasi WW_introduction_ subsection the equations_boundary cond 1} expresses in fact that the particles in contact with the rigid bottom remain in contact with it. As no hypothesis is made on the regularity of $\Gamma,$ this condition is shown to make sense in a weak variational meaning due to the hypothesis $H_t$, for more details on this we refer to Section 2 in \cite{Alazard11} and Section 3 in \cite{Alazard14}.\\
  
  The fluid motion is supposed to be irrotational and $\Omega_t$ is supposed to be simply connected thus the velocity $v$ field derives
   from some potential $\phi$ i.e $v=\nabla \phi$ and:
\[\begin{cases}
\Delta \phi=0  \text{ in }  \Omega,\\  \partial_n \phi=0  \text{ on } \Gamma.
\end{cases}
\]

The boundary condition on $\phi$ becomes:
 \begin{align}\label{geometric proof quasi WW_introduction_ subsection the equations_boundary cond 2}
 \begin{cases}
 \partial_n \phi =0 &  \text{on } \Gamma, \\
 \partial_t \eta=\partial_y \phi -\nabla \eta \cdot \nabla \phi & \text{on } \Sigma_t,\\
 \partial_t \phi=-g\eta+\kappa H(\eta)-\frac{1}{2} \abs{\nabla_{x,y} \phi}^2 &  \text{on }  \Sigma_t.
 \end{cases}
 \end{align}

Following Zakharov \cite{Zakharov68} and Craig-Sulem \cite{Craig93} we reduce the analysis to a system on the free surface $\Sigma_t$. If $\psi$
is defined by 
\[\psi(t,x)=\phi(t,x,\eta(t,x)),\]
then $\phi$ is the unique variational solution of
\[\Delta \phi =0  \text{ in } \Omega_t,  \ \phi_{|y=\eta}=\psi,  \  \partial_n \phi =0  \text{ on } \Gamma.\]
Define the Dirichlet-Neumann operator by 
\begin{align*}
(G(\eta)\psi)(t,x)&=\sqrt{1+|\nabla \eta|^2}\partial_n \phi_{|y=\eta}\\
&=(\partial_y \phi)(t,x,\eta(t,x))-\nabla \eta(t,x) \cdot (\nabla \phi)(t,x,\eta(t,x)).
\end{align*}
For the case with rough bottom we refer to \cite{Alazard09},  \cite{Alazard11} and \cite{Alazard14} for the well posedness of the variational problem and the Dirichlet-Neumann operator.
Now $(\eta,\psi)$ (see for example \cite{Craig93}) solves:
  \begin{align}\label{geometric proof quasi WW_introduction_ subsection the equations_WW eq}
 \partial_t \eta&=G(\eta)\psi ,\\
 \partial_t \psi&=-g\eta+\kappa H(\eta)+\frac{1}{2} \abs{\nabla \psi}^2  +\frac{1}{2}\frac{\nabla \eta \cdot \nabla \psi+G(\eta)\psi}{1+|\nabla \eta|^2}.\nonumber
 \end{align}
%We see that in the case of $\kappa=0$, the pressure "disappears" from the current formulation and an additional work is needed to redefine it.
\subsection{Gravity water waves: Pressure and Taylor Coefficients} 
Here we give a quick review of the ideas in \cite{Alazard13}. Recall that by definition for gravity water waves we work with $\kappa=0$ and we define
the Taylor coefficient
\[a(t,x)=-(\partial_yP)(t,x,\eta(t,x)).\]

The stability of the waves is dictated by the Taylor sign condition, which is the assumption that there exists a positive constant c such that
\begin{align}\label{geometric proof quasi WW_introduction_ Gravity water waves Pressure and Taylor Coefficients_Taylor condition}
a(t,x)\geq c>0.
\end{align}
In \cite{Alazard14} this condition is needed in the proof of the well posedness of the Cauchy problem and it is shown to be locally propagated by the flow.\\

Now we will show how to define $P$ from the Zakharov formulation. Let R be the variational solution of 
\[\Delta R=0  \text{ in } \Omega_t,\ \ R_{|y=\eta}=\eta g+\frac{1}{2}\abs{\nabla_{x,y}\phi}^2_{|y=\eta}.\]
We define the pressure P in the domain $\Omega$ by 
\[P(x,y)=R(x,y)-gy-\frac{1}{2}\abs{\nabla_{x,y} \phi(x,y)}^2.\]

 In \cite{Alazard13}  Alazard, Burq, and Zuily show that to a solution: $$(\eta,\psi) \in C([0,T],H^{s+\frac{1}{2}(\r^d)}\times H^{s+\frac{1}{2}(\r^d))},$$ for $s >\frac{d}{2}+\frac{1}{2}$ of the Zakharov/Craig-Sulem system \eqref{geometric proof quasi WW_introduction_ subsection the equations_WW eq} corresponds a unique solution $v$ of the Euler system.
 
\subsection{Quasi-linearity of the water Wave system }
 In \cite{Alazard11} and \cite{Alazard14},  Alazard, Burq, and Zuily perform a paralinearization and symmetrization of the the water waves system that take the form:
\[
\partial_tu+ T_V.\nabla u+i T_{\gamma}u=f,
\]
where $\gamma$ is of order $\frac{3}{2}$ in the case with surface tension and $\frac{1}{2}$ in the case without. The terms $V$ and $\gamma$ verify the conditions required by Theorem \ref{geometric proof quasi WW_a tech gen_main theorem} and thus the paralinearization of the water-waves system are quasi-linear in the considered thresholds of regularity. From this we will deduce the following two theorems.\\

First in the case of water waves with surface tension, i.e $\kappa=1$, where the well-posedness of the Cauchy problem is proved in \cite{Alazard11} we complete it by the following.
\begin{theorem}\label{geometric proof quasi WW_introduction_Quasi-linearity of the water Wave system_theorem with surface tension}
Fix the dimension $d \geq 1$ and consider two real numbers $r>0$,\\
 $s \in]2+ \frac{d}{2},+\infty[$ and $(\eta_0,\psi_0)\in H^{s+\frac{1}{2}}(\r^d) \times H^{s}(\r^d) $ such that 
 $$\forall (\eta'_0,\psi'_0) \in \b((\eta_0,\psi_0),r)\subset H^{s+\frac{1}{2}}(\r^d) \times H^{s}(\r^d)$$
  the assumption $H_{t=0}$ is satisfied. Then there exists $T>0$ such that the Cauchy problem \eqref{geometric proof quasi WW_introduction_ subsection the equations_WW eq} with initial data $(\eta'_0,\psi'_0)\in \b((\eta_0,\psi_0),r)$ has a unique solution
\[(\eta',\psi') \in C^0([0,T];H^{s+\frac{1}{2}}(\r^d) \times H^{s}(\r^d) )\]
and such that the assumption $H_t$ is satisfied for $t \in [0,T]$.

Moreover $\forall R>0$ the flow map: 
	\begin{align*}
										\b(0,R)  \rightarrow &C([0,T],H^{s+\frac{1}{2}}(\r^d) \times H^{s}(\r^d))\\
										(\eta'_0,\psi'_0) \mapsto &(\eta',\psi')
										\end{align*}
				is not uniformly continuous.
								
We show that at least a loss of $\frac{1}{2}$ derivative is necessary to have Lipschitz control over the flow map, i.e for all $\eps'>0$ the flow map
	\begin{align*}
										\b(0,R)  \rightarrow &C([0,T],H^{s+\eps'}(\r^d) \times H^{s-\frac{1}{2}+\eps'}(\r^d))\\
										(\eta'_0,\psi'_0) \mapsto &(\eta',\psi')
										\end{align*}

				is not $C^1$.
			
\end{theorem} 

\begin{remark}
In our work \cite{Ayman19} we prove that for the Gravity Capillary equation in dimension one, the loss of $\frac{1}{2}$ is sufficient to have Lipschitz control over the flow map, which is an improvement compared to the non linearity of order $\frac{3}{2}$ in the dispersive term.
\end{remark}

Now we turn to gravity water waves, i.e $\kappa=0$ where the well posedness of the Cauchy problem is proved in \cite{Alazard14}. It is well known that the vertical and horizontal traces of the velocity on the free boundary play an important role in the well posedness of the Cauchy problem and are given by:
\begin{align}\label{geometric proof quasi WW_introduction_Quasi-linearity of the water Wave system_equations on the traces of the velocity}
B&=(\partial_y \phi)_{|y=\eta}=\frac{\nabla \eta \cdot \nabla \psi +G(\eta) \psi}{1+|\nabla \eta|^2},\\
V&=(\nabla_x \phi)_{|y=\eta}=\nabla\psi-B\nabla \eta.\nonumber
\end{align}
\begin{theorem}\label{geometric proof quasi WW_introduction_Quasi-linearity of the water Wave system_theorem without surface tension}
Fix the dimension $d \geq 1$ and consider two real numbers $r>0$,\\
 $s \in]2+ \frac{d}{2},+\infty[$\footnote{Here we are slightly above the threshold of well-posedness of $1+\frac{d}{2}$ proved in \cite{Alazard14}.} and $(\eta_0,\psi_0)\in H^{s+\frac{1}{2}}(\r^d) \times H^{s+\frac{1}{2}}(\r^d) $ and consider $$(\eta'_0,\psi'_0) \in \b((\eta_0,\psi_0),r)\subset H^{s+\frac{1}{2}}(\r^d) \times H^{s+\frac{1}{2}}(\r^d)$$ such that we have:
\begin{enumerate}
\item $V'_0 \in H^s(\r^d), \ \ B'_0 \in H^s(\r^d),$
\item $H_{t=0}$ is satisfied,
\item there exits a positive constant c such that, $\forall x \in \r^d, a'_0(x)\geq c>0$.
\end{enumerate}
Then there exists $T>0$ such that the Cauchy problem \eqref{geometric proof quasi WW_introduction_ subsection the equations_WW eq} with initial data $(\eta'_0,\psi'_0)$ has a unique solution
\[(\eta',\psi') \in C^0([0,T];H^{s+\frac{1}{2}}(\r^d) \times H^{s+\frac{1}{2}}(\r^d) )\]
such that for $t \in [0,T]$ the assumption $H_t$ is satisfied, $\forall x \in \r^d, a'(t,x)\geq \frac{c}{2}$ and 
\[(V',B') \in C^0([0,T];H^{s}(\r^d) \times H^{s}(\r^d) ).\]

Moreover $\forall R>0$, the flow map: 	\begin{align*}
										\b(0,R) %\cap H^{s+\frac{1}{2}}(\r^d) \times H^{s+\frac{1}{2}}(\r^d)
										 \rightarrow &C([0,T],H^{s+\frac{1}{2}}(\r^d) \times H^{s+\frac{1}{2}}(\r^d))\\
										(\eta'_0,\psi'_0) \mapsto &(\eta',\psi')
										\end{align*}
				is not uniformly continuous.\\
Considering a weaker control norm we get: For all $\eps'>0$, the flow map:
	\begin{align*}
										\b(0,R) %\cap H^{s+\frac{1}{2}}(\r^d) \times H^{s+\frac{1}{2}}(\r^d)
										 \rightarrow &C([0,T],H^{s-\frac{1}{2}+\eps'}(\r^d) \times H^{s-\frac{1}{2}+\eps'}(\r^d))\\
										(\eta'_0,\psi'_0) \mapsto &(\eta',\psi')
										\end{align*}

				is not $C^1$.	
\end{theorem}

\begin{remark}
It is worth noticing that a previous result was obtained on the regularity of the flow map for the two dimensional gravity-capillary water waves (i.e with surface tension) in \cite{Marzuola13} where they have proved that the flow is not $C^3$ with respect to initial data $(\eta_0,\psi_0)\in H^{s+\frac{1}{2}}(\r^2) \times H^{s}(\r^2) $ for $s<3$. 

This result is in contrast with our result which holds for $s>3$ and this can indeed be seen in the fact that in \cite{Marzuola13} the lack of regularity of the flow is shown to be primarily due to the influence of surface tension. Though in our work the lack of regularity of the flow is shown to be due to the hydrodynamic term (the non-linear transport term).
\end{remark}
\begin{remark}
As the Cauchy problem for the water waves system on $\t^d$ is solved by the same particularization and symmetrization (see \cite{Alazard16}) and our technical generalization in Section \ref{geometric proof quasi WW_a tech gen} is proved on $\d^d$ the previous results for the water waves on $\r^d$ extend tautologically to $\t^d$.
\end{remark}
\subsection{Strategy of the proof}
We explain the key ideas at the level of the equation \eqref{geometric proof quasi WW_introduction_ theorem on dispersive Burgers eq_eq dispersive Burgers}, 
\[\partial_t v+v\partial_xv+\partial_x \abs{D}^{\alpha-1} v=0.\]
 
 The point of start is to adapt the classic proof of the quasi-linearity of the Burgers equation, presented to me in a personal note of C. Zuily \cite{Zuily}, that we will recall here.
\subsubsection{Quasi-linearity of the Burgers equation}
 The result of quasi-linearity of the Burgers equation is that the flow map taken point-wise in time fails to be uniformly continuous. Such a result is obtained by constructing two families of solutions $u $ and $v $ from some initial data $u^0 $and $v^0 $ depending on parameters $\lambda$ and $\epsilon$ such that 
 $$ \lim_{\substack{\lambda \to +\infty \\ \eps \to 0 }} \norm{u^0 -v^0 }_{H^s}=0 \text{ and } \norm{(u -v )(t,\cdot)}_{H^s}\geq c>0, \text{ with $t >0$.}$$ 
 
 To show how to construct such families we start by recalling the usual geometric construction of the graph of a function $u(t,\cdot)$ solution to the Burgers equation with initial data $u^0$. Put $$\chi(t,x)=x+tu^0(x)$$ the characteristic flow associated to the problem, which is a diffeomorphism in the $x$ variable. Then,  $$ u(t,\cdot)=u^0\circ \chi(t,x)^{-1}.$$
 The action of $\chi^{-1}$ on the graph of $u^0$ is given by the following Figure \ref{geometric proof quasi WW_introduction_sketch of the proof_figure 1} that also shows the shock formation phenomena.
 %%%%%%%%Figure%%%%%%%%
 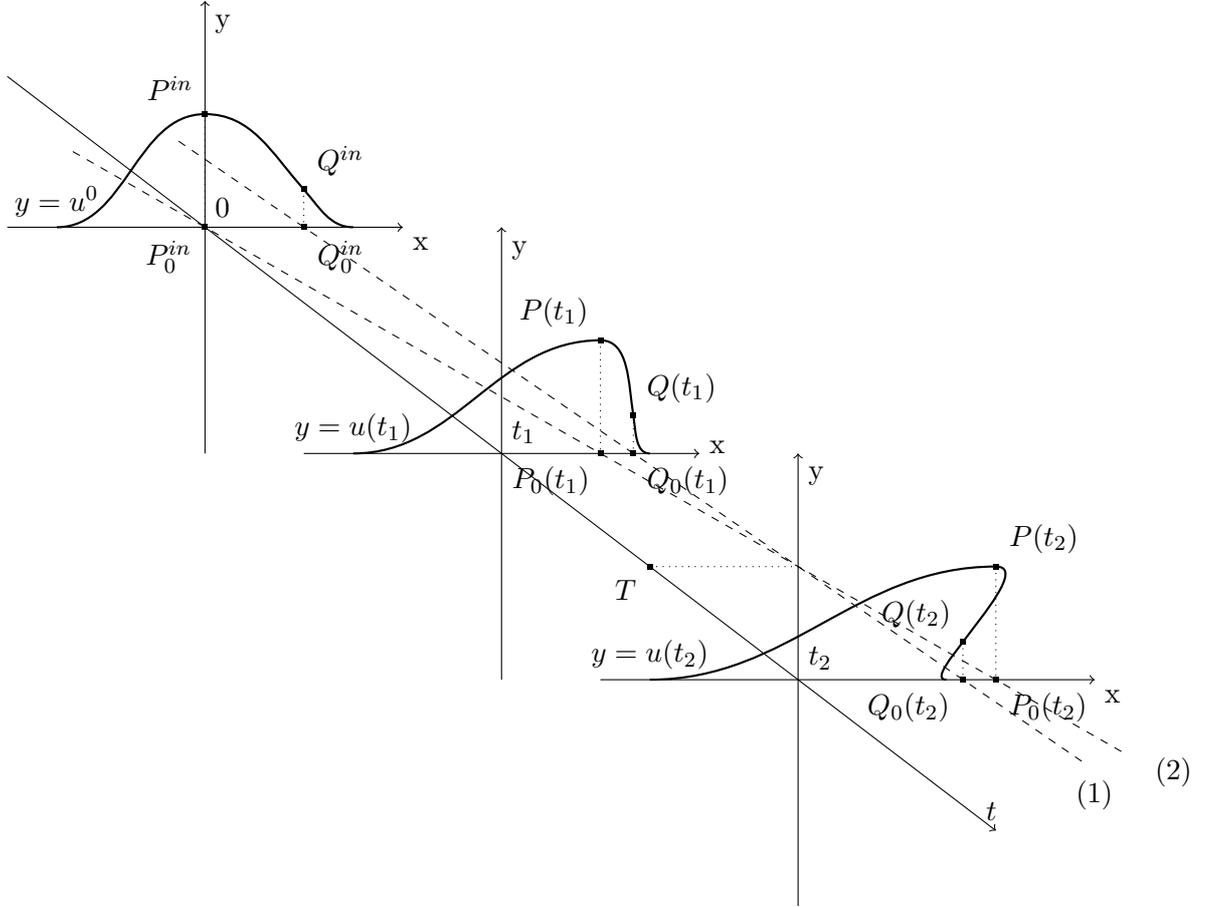
\begin{figure}[h!]
\begin{tikzpicture}[xscale=1.3,yscale=1.5,dot/.style={circle,inner sep=1pt,fill,label={#1},name=#1},
 extended line/.style={shorten >=-#1,shorten <=-#1},
 extended line/.default=1cm]
%\draw[help lines] (-5,0) grid (5,9);
%%%Grid 1
\draw [->] (-3,5) -- (-3,9);
\draw [->] (-5,7) -- (-1,7);
\node[right] at (-3,8.8) {y};
\node[below right] at (-1,7) {x};
\path [name path=C0] (-4.5,7) to [out=0,in=180] (-3,8) to [out=0,in=135] (-2,7+1/3)
        to [out=-45,in=180] (-1.5,7) ;
\path [name path=D0] (-2,2)--(-2,9);
\coordinate (Q0) at (-2,7);
\coordinate (P0) at (-3,7);
\coordinate (Pin) at (-3,8);
\node [fill=black,inner sep=1pt,label=225:$P^{in}_0$] at (P0) {};
\node [fill=black,inner sep=1pt,label=-45:$Q^{in}_0$] at (Q0) {};
\node [fill=black,inner sep=1pt,label=120:$P^{in}$] at (Pin) {};
\path [name intersections={of=C0 and D0,by=Qin}];
\node [fill=black,inner sep=1pt,label=45:$Q^{in}$] at (Qin) {};
\draw [dotted] (Qin) -- (Q0);
\draw [dotted] (Pin) -- (P0);
\draw [thick]  (-4.5,7) node[above] {$y=u^0$}  to [out=0,in=180] (-3,8) to  [out=0,in=135] (-2,7+1/3)
        to [out=-45,in=180] (-1.5,7) ;
%%%Grid 2
\draw [->] (0,3) -- (0,7);
\draw [->] (-2,5) -- (2,5);
\node[right] at (0,6.8) {y};
\node[below right] at (2,5.2) {x};
\coordinate (Q1) at (1+1/3,5);
\coordinate (Qin1) at (1+1/3,5+1/3);
\coordinate (P1) at (1,5);
\coordinate (Pin1) at (1,6);
\node [fill=black,inner sep=1pt,label=260:$P_0(t_1)$] at (P1) {};
\node [fill=black,inner sep=1pt,label=-45:$Q_0(t_1)$] at (Q1) {};
\node [fill=black,inner sep=1pt,label=120:$P(t_1)$] at (Pin1) {};
\node [fill=black,inner sep=1pt,label=45:$Q(t_1)$] at (Qin1) {};
\draw [dotted] (Qin1) -- (Q1) ;
\draw [dotted] (Pin1) -- (P1);
\draw [thick] (-1.5,5) node[above] {$y=u(t_1)$}  to [out=0,in=180] (Pin1) to  [out=0,in=100] (Qin1)
        to [out=-80,in=180] (1.5,5) ;
%%%Grid 3
\draw [->] (3,1) -- (3,5);
\draw [->] (1,3) -- (6,3);
\node[right] at (3,4.8) {y};
\node[below right] at (6,3) {x};
\coordinate (Q2) at (4+2/3,3);
\coordinate (Qin2) at (4+2/3,3+1/3);
\coordinate (P2) at (5,3);
\coordinate (Pin2) at (5,4);
\node [fill=black,inner sep=1pt,label=-45:$P_0(t_2)$] at (P2) {};
\node [fill=black,inner sep=1pt,label=225:$Q_0(t_2)$] at (Q2) {};
\node [fill=black,inner sep=1pt,label=45:$P(t_2)$] at (Pin2) {};
\node [fill=black,inner sep=1pt,label=120:$Q(t_2)$] at (Qin2) {};
\draw [dotted] (Qin2) -- (Q2);
\draw [dotted] (Pin2) -- (P2);
\draw [thick] (1.5,3) node[above] {$y=u(t_2)$}  to [out=0,in=180] (Pin2) to  [out=0,in=50] (Qin2)
        to [out=-130,in=180] (4.5,3) ;
%%%Time-lines
\draw[->,thin](-5,7+4/3)--(5,3-4/3);
\coordinate (t_1) at (0,5);
\coordinate (t_2) at (3,3);
\coordinate (0) at (-3,7);
\node[above right] at (0) {$0$};
\node[above right] at (t_1) {$t_1$};
\node[above right] at (t_2) {$t_2$};
\node[above right] at (4.8,3-4/3) {$t$};
\draw [extended line=2cm, dashed] (Q0) -- (Q2);
\node[below] at (6,2.2) {(1)};
\node[below] at (6.8,2.4) {(2)};
\draw [extended line=2cm, dashed]  (P0) -- (P2);
\path [name path=Q0Q2](Q0) -- (Q2);
\path [name path=P0P2](P0) -- (P2);
\draw[dotted] (1.5,4) -- (3,4);
\node [fill=black,inner sep=1pt,label=225:$T$] at (1.5,4) {};
\path [name intersections={of=C0 and D0,by=Qin}];
\end{tikzpicture}
\caption{The lines (1) and (2) are the characteristic curves from $Q^{in}_0$ and $P^{in}_0$. T is the time of formation of the shock wave.}
\label{geometric proof quasi WW_introduction_sketch of the proof_figure 1}
 \end{figure}
  %%%%%%%%Figure%%%%%%%%
\\ 
Then $u^0 $ and $v^0$ are chosen as a high frequency compactly supported ansatz depending on $(\lambda,\eps) $:
$$ u^0(x)=\lambda^{\frac{1}{2}-s}\omega(\lambda x),  \ \ v^0(x) =u^0(x) +\eps \omega(x), \ \ \text{with} \ \ \omega \in C^\infty_0,$$
where $\eps $ represents a change in the initial speed of transport,
 and $(\eps ,\lambda )$ verify:
 \begin{itemize}
 \item $\eps\rightarrow 0$ insuring that the difference in the $H^s$ norm of the sequences of initial data goes to $0$.
 \item $\lambda \rightarrow +\infty$ is the usual ansatz parameter hypothesis.
 \item $\eps \lambda \rightarrow +\infty$ insuring that the change of transport speed is enough to have disjoint supports at positive time.
 \end{itemize} 
 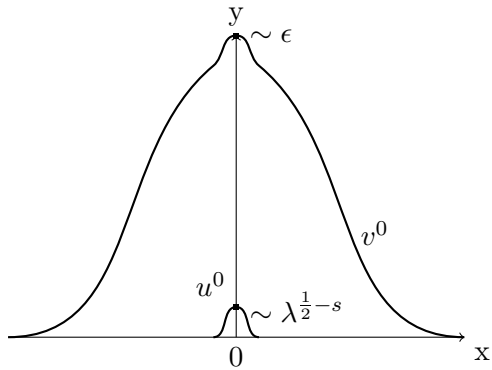
\begin{figure}[h!]
\begin{tikzpicture}[xscale=1,yscale=1,dot/.style={circle,inner sep=1pt,fill,label={#1},name=#1},
 extended line/.style={shorten >=-#1,shorten <=-#1},
 extended line/.default=1cm]
%\draw[help lines] (-3,0) grid (4,4);
 \draw [->] (-3,0) -- (3,0);
\draw [->] (0,0) -- (0,4);
\node[above] at (0,4) {y};
\node[below right] at (3,0) {x};
\path [name path=C0] (-3,0) to [out=0,in=220] (-0.3,3.6)
 	 to [out=40,in=180] (0,4) 
        to [out=0,in=180]  (0.3,3.6)
        to [out=0,in=180] (3,0) ;
 \draw [thick] (-3,0) to [out=0,in=220] (-0.3,3.6)
 	 to [out=40,in=180] (0,4) 
        to [out=0,in=140]  (0.3,3.6)
        to [out=-40,in=180]  node[right] {$v^0$}  (3,0) ;
 \path [name path=C1] (-0.3,0) to [out=0,in=180] (0,0.4) 
        to [out=0,in=180]  (0.3,0) ;
 \draw [thick] (-0.3,0) to [out=0,in=180] (0,0.4) node[above left] {$u^0$}
        to [out=0,in=180]  (0.3,0) ;
   \node [fill=black,inner sep=1pt,label=0:$\sim \lambda^{\frac{1}{2}-s}$] at (0,0.4) {};
      \node [fill=black,inner sep=1pt,label=0:$\sim \eps$] at (0,4) {};
 \coordinate (0) at (0,0);
  \node[below ] at (0) {$0$};
 \end{tikzpicture}
\caption{Graph of the ansatz.}
\label{geometric proof quasi WW_introduction_sketch of the proof_figure 2}
 \end{figure}
Now if we put $\chi $ and $\tilde{\chi }$ to be the characteristic flows associated to the solutions $u^0 $ and $v^0$ then:
  \begin{align*}
 (u -v )(t,x)&=u^0 (\chi (t,x)^{-1})-v^0 (\tilde{\chi} (t,x)^{-1})\\
 &=u^0 (\chi (t,x)^{-1})-u^0 (\tilde{\chi} (t,x)^{-1})+O_{H^s}(\eps ).
  \end{align*}
 
 Then using the compactly supported property of $u^0 $ and the change of speed we prove that $u^0 (\chi (t,x)^{-1})$ and $u^0 (\tilde{\chi} (t,x)^{-1})$ have disjoint supports which is illustrated by Figure \ref{geometric proof quasi WW_introduction_sketch of the proof_figure 3}.
  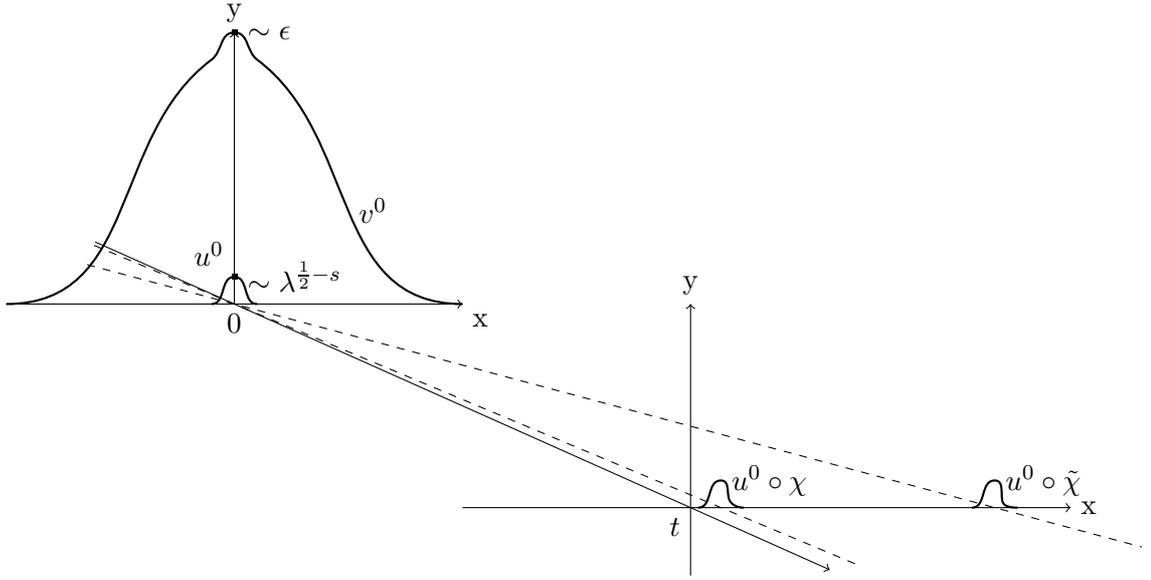
\begin{figure}[h]
\begin{tikzpicture}[xscale=1,yscale=0.9,dot/.style={circle,inner sep=1pt,fill,label={#1},name=#1},
 extended line/.style={shorten >=-#1,shorten <=-#1},
 extended line/.default=1cm]
%\draw[help lines] (-8,-4) grid (6,4);
%%%%%%%%%%Grid1%%%%%%%%%%%%
 \draw [->] (-8,0) -- (-2,0);
\draw [->] (-5,0) -- (-5,4);
\node[above] at (-5,4) {y};
\node[below right] at (-2,0) {x};
\path [name path=C0] (-8,0) to [out=0,in=220] (-5.3,3.6)
 	 to [out=40,in=180] (-5,4) 
        to [out=0,in=140]  (-4.7,3.6)
        to [out=-40,in=180]  (-2,0) ;
 \draw [thick] (-8,0) to [out=0,in=220] (-5.3,3.6)
 	 to [out=40,in=180] (-5,4) 
        to [out=0,in=140]  (-4.7,3.6)
        to [out=-40,in=180]  node[right] {$v^0$}  (-2,0) ;
 \path [name path=C1] (-5.3,0) to [out=0,in=180] (-5,0.4) 
        to [out=0,in=180]  (-4.7,0) ;
 \draw [thick] (-5.3,0) to [out=0,in=180] (-5,0.4) node[above left] {$u^0$}
        to [out=0,in=180]  (-4.7,0) ;
   \node [fill=black,inner sep=1pt,label=0:$\sim \lambda^{\frac{1}{2}-s}$] at (-5,0.4) {};
      \node [fill=black,inner sep=1pt,label=0:$\sim \eps$] at (-5,4) {};
 \coordinate (0) at (-5,0);
  \node[below ] at (0) {$0$};
  %%%%%%%%%%%%%Grid2%%%%%%%%%
   \draw [->] (-2,-3) -- (6,-3);
\draw [->] (1,-4) -- (1,0);
\node[above] at (1,0) {y};
\node[right] at (6,-3) {x};
 \coordinate (t) at (1,-3);
  \node[below left] at (t) {$t$};
 \path [name path=C2] (1.1,-3) to [out=0,in=180] (1.4,-2.6) 
        to [out=0,in=180]  (1.7,-3) ;
 \draw [thick] (1.1,-3) to [out=0,in=180] (1.4,-2.6) node[right] {$u^0\circ \chi$} to  [out=0,in=100] (1.5,-2.85)
        to [out=-80,in=180]  (1.7,-3) ;
\path [name path=C3] (3.7,-3) to [out=0,in=180] (4,-2.6) 
        to [out=0,in=180]  (4.3,-3) ;
 \draw [thick] (4.7,-3) to [out=0,in=180] (5,-2.6) node[right] {$u^0\circ \tilde{\chi}$} to  [out=0,in=100] (5.1,-2.9)
        to [out=-80,in=180]  (5.3,-3) ;
  %%%Time-lines%%%%%%%%
  \draw [extended line=2cm, thin,->]  (0) -- (t);
  \draw [extended line=2cm, dashed]  (0) -- (1.4,-3);
  \draw [extended line=2cm, dashed]  (0) -- (5,-3);
 \end{tikzpicture}
\caption{Transport of the ansatz.}
\label{geometric proof quasi WW_introduction_sketch of the proof_figure 3}
 \end{figure}
 We then prove that $\norm{u^0 (\chi (t,x)^{-1})}_{H^s}\geq c>0$ which finishes the proof of the non uniform continuity of the flow map.
 For the control in a weaker norm, that is the flow map cannot be $C^1(H^s(\d),C^0([0,T],H^{s-1+\eps}(\d)))$, we get it from the estimate $\norm{u^0 (\chi (t,x)^{-1})}_{H^{s-\mu}}\geq c\lambda^{-\mu}$.
 
\subsubsection{Quasi-linearity of problem \eqref{geometric proof quasi WW_study model pbm_subsec prereq cauchy pbm_eq1}}
Now if we adapt the proof to our current problem \eqref{geometric proof quasi WW_introduction_ theorem on dispersive Burgers eq_eq dispersive Burgers} we get: 
 \begin{align*}
 (u -v )(t,x)&=f (t,\chi (t,x)^{-1})-g(t, \tilde{\chi} (t,x)^{-1})\\
 &=f (t,\chi (t,x)^{-1})-f (t,\tilde{\chi} (t,x)^{-1})+O_{H^s}\big(\eps+t^2\eps\lambda^\alpha\big),
 \end{align*}
 where $f $ and $g $ are solutions to
  \begin{align} 
 \partial_t f+(\partial_x \abs{D}^{\alpha-1})^*  f=0 \label{geometric proof quasi WW_introduction_sketch of the proof_equation on f}\\
 \partial_t g+\widetilde{(\partial_x \abs{D}^{\alpha-1})}^* g=0 \label{geometric proof quasi WW_introduction_sketch of the proof_equation on g} 
\end{align}
and $(\cdot)^*  $ and $\widetilde{(\cdot)}^* $ are the change of variables by the characteristic flows defined for a symbol $a$ by
\[
\op(a)^*(u\circ \chi)=(\op(a)u)\circ \chi \text{ i.e } \op(a)^*(u)=(\op(a)[u\circ \chi^{-1}])\circ \chi,
\]
and analogously for $\widetilde{(\cdot)}^* $, which we prove that they are well posed in Appendix \ref{geometric proof quasi WW_Appendix energy est of pulled back eq}. 
  
  The first immediate problem we face is the extra term $t^2 \eps\lambda ^{\alpha} $ which diverges, to remedy this we give up control of the flow map punctually in time and use a conveniently chosen sequence of small time $(\tau)$ to control $\tau^2 \eps \lambda ^{\alpha}$:
  $$\tau \rightarrow 0,\text{ but still insure } \lambda \eps \tau \rightarrow +\infty.$$
  
  The second, deeper problem we face is that we lose control over the support of the solution. 
  Indeed  \eqref{geometric proof quasi WW_introduction_sketch of the proof_equation on f} and \eqref{geometric proof quasi WW_introduction_sketch of the proof_equation on g} are obtained by pull-back of the linear equation
\begin{equation}\label{geometric proof quasi WW_introduction_sketch of the proof_linear equation}
  \partial w +\partial_x \abs{D}^{\alpha-1} w=0
  \end{equation}
  which is a non local-dispersive equation that is expected to disperse the support of the solution and the $L^\infty$ norm. This phenomena is thus expected to oppose the phenomena illustrated by the previous Figures \eqref{geometric proof quasi WW_introduction_sketch of the proof_figure 1} and \eqref{geometric proof quasi WW_introduction_sketch of the proof_figure 2} and indeed does so for the KdV equation on $\r$.\\
  
  To remedy this, the idea is not to use $u^0 $ and $v^0 $ as initial data but by profiting of the time reversibility \footnote{This idea fundamentally depends on the local reversibility in time of the linearised equations and thus fails for the fractional Burgers equation.} of the equations use the backward in time solutions $u^1 $ and $v^1 $ defined by:
  \begin{align*}
  \begin{cases}
  \omega \text{ solution of \eqref{geometric proof quasi WW_introduction_sketch of the proof_linear equation}},\\
  \omega(\tau,\cdot)=u^0,\\
  \omega(0,\cdot)=u^1,
  \end{cases}
    &\begin{cases}
  \omega' \text{ solution of \eqref{geometric proof quasi WW_introduction_sketch of the proof_linear equation}},\\
  \omega'(\tau,\cdot)=v^0,\\
  \omega'(0,\cdot)=v^1.
  \end{cases}
  \end{align*}
  
   This gives us:
 \[
 (u -v )(\tau,x)=u^0 (\chi (0,\tau,x))-u^0 (\tilde{\chi} (0,\tau,x))+O_{H^s}\big(\eps +\tau^2\eps \lambda^{\alpha}+\tau^2 \lambda^{\alpha-1}\big).
 \]

We then prove that this gives the desired result, in the threshold $\alpha \in [0,2[$, by proving analogously to the Burgers equation: $\norm{u^0 (\chi^{-1} (t,x))}_{H^s}\geq c>0$ and then using the compactly supported property of $u^0 $ and the change of speed we prove that $u^0 (\chi (0,\tau,x))$ and $u^0 (\tilde{\chi} (0,\tau,x))$ have disjoint supports.
\subsection{Acknowledgement}
I would like to express my sincere gratitude to my thesis advisor Thomas Alazard. 
I would also like to thank Claude Zuily for his insightful note on the Burgers equation that helped me understand the problem.

\section{Study of the model equation}%%%eq8,th,lem1,rem,cor
In this section we give a full proof of Theorem \ref{geometric proof quasi WW_introduction_ theorem on dispersive Burgers eq}.
\subsubsection{Prerequisites on the Cauchy Problems}\label{geometric proof quasi WW_study model pbm_subsec prereq cauchy pbm} %%%eq10,th4,lem,rem2,cor
For a real number $\alpha \in [0,2[$, we consider the Cauchy problem\footnote{Recall that $D=\op(\abs{\xi})$.}:
\begin{equation} \label{geometric proof quasi WW_study model pbm_subsec prereq cauchy pbm_eq1} %%%%%%%%%ref%%%%%%%%%%
\begin{cases} 
\partial_t u+u\partial_xu+\partial_x \abs{D}^{\alpha-1} u=0 \\
u(0,\cdot)=u_0(\cdot) \in H^s(\d), \ s>\frac{3}{2},
\end{cases}
\end{equation}
It is well known that the problem is well posed in Sobolev spaces, this can be summarized in the following Theorem:
\begin{theorem} \label{geometric proof quasi WW_study model pbm_subsec prereq cauchy pbm_thm cauchy pbm}%%%%%%%%%ref%%%%%%%%%%
Consider two real numbers, $ s\in ]\frac{3}{2},+\infty[$ and $r>0$. Fix $u_0\in H^{s}(\d)$. 
Then there exists $T>0$, such that for all $v_0 \in \b(u_0,r) \subset H^{s}(\d)$, the problem \eqref{geometric proof quasi WW_study model pbm_subsec prereq cauchy pbm_eq1} with initial data $v_0$ has a unique solution $v \in C^0([0,T],H^{s}(\d))$, the map $v_0 \mapsto v$ is continuous from $\b(u_0,r)$ to $C^0([0,T],H^{s}(\d))$ and maps real functions into real functions.
Moreover we have the estimates:
\begin{equation} \label{geometric proof quasi WW_study model pbm_subsec prereq cauchy pbm_thm cauchy pbm_eq1} %%%%%%%%%ref%%%%%%%%%%
\forall 0 \leq \mu \leq s, \norm{v(t)}_{H^{\mu}(\d)} \leq C_\mu  \norm{v_0}_{H^{\mu}(\d)}. 
%\forall 0 \leq k \leq s-\frac{1}{2}, \norm{u(t)}_{W^{k,\infty}(\r)} \leq C_\mu  \norm{u_0}_{W^{k,\infty}(\r)}.
\end{equation}
Taking two different solutions $u,v$, assuming moreover $u_0 \in H^{s+1}(\d)$ then we have:
\begin{align} \label{geometric proof quasi WW_study model pbm_subsec prereq cauchy pbm_thm cauchy pbm_eq2} %%%%%%%%%ref%%%%%%%%%%
\norm{(u-v)(t)}_{H^{s}(\d)} &\leq C_s(\norm{(u,v)}_{H^s(\d)} ; t\norm{u}_{H^{s+1}(\d)})  \norm{u_0-v_0}_{H^{s}(\d)}
%\forall 1 \leq k \leq s-\frac{1}{2},\norm{(u-v)(t)}_{W^{k,\infty}(\r)} &\leq  \norm{u_0-v_0}_{W^{k,\infty}(\r)}e^{C_k \int_0^t \norm{u(s)}_{W^{k+1,\infty}(\r)}ds}.
\end{align}
We will also need to remark that fixing the initial data at 0 is an arbitrary choice, that is all of the previous conclusions hold for the Cauchy problem defined for $t_0\leq T$:
\begin{equation} 
\begin{cases} 
\partial_t v+v\partial_xv+\partial_x \abs{D}^{\alpha-1} v=0 \\
v(t_0,\cdot)=v_0(\cdot) \in H^s(\d), \ s>\frac{3}{2}.
\end{cases}
\end{equation}
\end{theorem}

\begin{remark}\label{geometric proof quasi WW_study model pbm_subsec prereq cauchy pbm_remark on link to Burgers eq}
Note that the previous Theorem holds for the Cauchy problem associated to the Burgers equation:
\begin{equation} \label{geometric proof quasi WW_study model pbm_subsec prereq cauchy pbm_remark on link to Burgers eq_eq1} %%%%%%%%%ref%%%%%%%%%%
\begin{cases} 
\partial_t u+u\partial_xu=0 \\
u(0,\cdot)=u_0(\cdot) \in H^s(\d), \ s>\frac{3}{2},
\end{cases}
\end{equation}
 Though we have some extra estimates in H\"older type spaces:
 \begin{align}\label{geometric proof quasi WW_study model pbm_subsec prereq cauchy pbm_remark on link to Burgers eq_eq2} %%%%%%%%%ref%%%%%%%%%%
 \forall 0 \leq k < s-\frac{1}{2}, \norm{u(t)}_{W^{k,\infty}(\d)} &\leq C_k  \norm{u_0}_{W^{k,\infty}(\d)},
  \intertext{Taking two different solution $u,v$, assuming moreover $u_0 \in H^{s+1}(\d)$ then we have:}
 \forall 1 \leq k < s-\frac{1}{2},\norm{(u-v)(t)}_{W^{k,\infty}(\d)} &\leq  \norm{u_0-v_0}_{W^{k,\infty}(\d)}e^{C_k \int_0^t \norm{u(s)}_{W^{k+1,\infty}(\d)}ds}.\nonumber
 \end{align}
\end{remark}

\begin{remark}\label{geometric proof quasi WW_study model pbm_subsec prereq cauchy pbm_remark on the scaling}
We compute the change of scale for the evolution PDE \eqref{geometric proof quasi WW_study model pbm_subsec prereq cauchy pbm_eq1}:
 \[
 u_0 \mapsto \lambda^{\alpha-1}u_0(\lambda x)\]
  gives the solution
   \[ \lambda^{\alpha-1}u(\lambda^\alpha t,\lambda x).\]
    Thus giving the critical scaling in Sobolev spaces: $s_c=1+\frac{1}{2}-\alpha$, thus we prove quasi-linearity in the subcritical regime of the problem.
\end{remark}
\begin{notation} In order not to be confused with the pull-back symbol, henceforth the conjugate of a symbol $a$ will be written as $a^\top$.
\end{notation}
As the linearized equation is a hyperbolic pseudo-differential equation we recall the result on the Cauchy problem associated to this type of equations:
\begin{theorem}\label{geometric proof quasi WW_study model pbm_subsec prereq cauchy pbm_thm cauchy pbm linear pseudo}
Consider $(a_t)_{t\in \r}$ a family of symbols in $S^\beta(\d^d)$,$\beta \in \r$, such that\\
 $t \mapsto a_t$ is continuous and bounded from $\r$ to $S^\beta(\d^d)$
  and such that $ \RE(a_t)=\frac{a_t+a_t^\top}{2}$ is bounded in $S^0(\d^d)$, and take $T>0$.
   Then for all $s \in \r$, $u_0 \in H^s(\d^d)$ and $f \in  C^0([0,T];H^s(\d^d))$ the Cauchy problem:
 \begin{equation}
\begin{cases}
 \partial_t u+\op(a)u=f\\
 \forall x \in \d^d, u(0,x)=u_0(x)
\end{cases}
\end{equation}
has a unique solution $u \in C^0([0,T];H^s(\d^d))\cap C^1([0,T];H^{s-\beta}(\d^d))$ which verifies the estimates:
\[\norm{u(t)}_{H^s(\d^d)}\leq e^{Ct}\norm{u_0}_{H^s(\d^d)}+2\int_0^te^{C(t-t')}\norm{f(t')}_{H^s(\d^d)}dt',\]
where C depends on a finite symbol semi-norm of $ \RE(a_t) $.
We will also need to remark that fixing the initial data at 0 is an arbitrary choice.
More precisely, $\forall 0 \leq t_0\leq T$ and all data $u_0 \in H^s(\d^d)$  the Cauchy problem:
 \begin{equation} 
\begin{cases}
 \partial_t u+\op(a)u=f\\
 \forall x \in \d^d, u(t_0,x)=u_0(x)
\end{cases}
\end{equation}
has a unique solution $u \in C^0([0,T];H^s(\d^d))\cap C^1([0,T];H^{s-\beta}(\d^d))$ which verifies the estimate:
\[\norm{u(t)}_{H^s(\d^d)}\leq e^{C\abs{t-t_0}}\norm{u_0}_{H^s(\d^d)}+2\abs{\int_{t_0}^te^{C(t-t')}\norm{f(t')}_{H^s(\d^d)}dt'}.\]
\end{theorem}

\subsection{Proof of Theorem \ref{geometric proof quasi WW_introduction_ theorem on dispersive Burgers eq}}%%%eq8,th,lem1,rem,cor
To prove the theorem we will show that there exists a positive constant $C$ and two sequences $(u^\lambda_{\eps,\tau})$ and $(v^\lambda_{\eps,\tau})$ solutions of \ref{geometric proof quasi WW_introduction_ theorem on dispersive Burgers eq_eq dispersive Burgers} on $[0,1]$ such that for every $t\in [0,1]$,
\[
\sup_{\lambda,\eps,\tau} \norm{u^\lambda_{\eps,\tau}}_{L^\infty([0,1],H^s(\d))}+
\norm{v^\lambda_{\eps,\tau}}_{L^\infty([0,1],H^s(\d))}\leq C,
\]
$(u^\lambda_{\eps,\tau})$ and $(v^\lambda_{\eps,\tau})$ satisfy initially
\[
\lim_{\substack{\lambda \rightarrow +\infty \\ \eps,\tau\rightarrow 0}} \norm{u^\lambda_{\eps,\tau}(0,\cdot)-v^\lambda_{\eps,\tau}(0,\cdot)}_{H^s(\d)}=0,
\]
but,
\[
\liminf_{\substack{\lambda \rightarrow +\infty \\ \eps,\tau\rightarrow 0}} \norm{u^\lambda_{\eps,\tau}-v^\lambda_{\eps,\tau}}_{L^\infty([0,1],H^s(\d))}\geq c>0.
\]

 Considering a weaker control norm we want to get, for all $\delta>0$,
 \[
\liminf_{\substack{\lambda \rightarrow +\infty \\ \eps,\tau\rightarrow 0}} \frac{\norm{u^\lambda_{\eps,\tau}-v^\lambda_{\eps,\tau}}_{L^\infty([0,1],H^{s-1+(\alpha-1)^+ +\delta}(\d))}}{\norm{u^\lambda_{\eps,\tau}(0,\cdot)-v^\lambda_{\eps,\tau}(0,\cdot)}_{H^s(\d)}}=+\infty.
\]
\subsubsection{Definition of the Ansatz}
\begin{itemize}
\item For $\d=\r$, take $\omega \in C_0^{\infty}(\r), \omega(x)=1  \text{ if } \abs{x}\leq \frac{1}{2}, \omega(x)=0 \text{ if } \abs{x} \geq 1$.
\item For $\d=\t$, we see functions on $\t=\r/2\pi \z$ as $2\pi$ periodic function on $\r$ and we take $\omega\in C_0^{\infty}(\t)$ as the periodic extension of the function defined above.
\end{itemize}

 Let $(\lambda ,\eps )$ be two positive real sequences such that
\begin{equation} \label{geometric proof quasi WW_study model pbm_subsec def ansatz_proof eq 1} %%%%%%%%%ref%%%%%%%%%%
  \lambda  \rightarrow + \infty, \  \eps  \rightarrow  0, \   \lambda \eps  \rightarrow  + \infty.
\end{equation}
Put 
\begin{itemize}
\item for $\d=\r$, \[u ^0(x)=\lambda ^{\frac{1}{2}-s}\omega(\lambda x), \ \ v ^0(x)=u ^0(x)+ \eps  \omega(x), \]
\item for $\d=\t$, $u^0$ and $v^0$ as the periodic extensions of the functions defined above.
\end{itemize}
Take $t_0>0$ smaller than a harmless constant which will be fixed later, and $(\tau),
 0 <\tau \leq t_0$ and $\tau \rightarrow 0$. \\
Now let $l, l'  $ be the solutions to the Cauchy problem on $[0,t_0]$:
\begin{align*}
  \begin{cases}
  \partial_t l +\partial_x \abs{D}^{\alpha-1} l=0,\\
  l(\tau,\cdot)=u^0,
  \end{cases}
    &\begin{cases}
  \partial_t l' +\partial_x \abs{D}^{\alpha-1} l'=0,\\
 l'(\tau,\cdot)=v^0.
  \end{cases}
  \end{align*}
Put $u ^1(x)=l(0,x)$ and define analogously $v ^1(x)=l'(0,x)$.\\

Define $u $ and $v $ as the solution given by Theorem \ref{geometric proof quasi WW_study model pbm_subsec prereq cauchy pbm_thm cauchy pbm}with initial data $u ^1$ and $v ^1$ on the intervals $[0,T]$ and $[0,T']$. 
Taking $0<\delta<s-\frac{3}{2}$, $u ^0$ and $v ^0$ are uniformly bounded in $H^{\frac{3}{2}+\delta}(\d) $
 when $\lambda  \rightarrow + \infty$ and thus by Theorem \ref{geometric proof quasi WW_study model pbm_subsec prereq cauchy pbm_thm cauchy pbm linear pseudo}, $u ^1$ and $v ^1$ are also
uniformly bounded in $H^{\frac{3}{2}+\delta}(\d) $
 and thus by the Sobolev injection Theorem they are bounded in $\dot{W}^{1,\infty}(\d) $. 
Thus we can take a uniform $0<T$ on which all the solutions are well defined and we take $0<t_0\leq T$
 \footnote{Heuristically, if the existence time of the solution with initial data $\omega$ is $[0,T]$ 
 then the existence time of the solution with initial data $u_0$ is $\sim T \lambda ^{s-\frac{3}{2}}$ which tends to infinity with $\lambda$, 
 thus we are "dilating" the time scale of the problem with initial data $\omega$ and  "zooming" for short time and in the $\dot{H}^s(\d) $ norm. 
 In this part of the evolution, we prove that
the Burgers transport term is more important and gives this quasi-linear character to the PDE.}.

\subsubsection{Change of variables by transport}

Put 
\[
\begin{cases}
	\frac{d}{dt}\chi (t,s,x)=u(t,\chi (t,s,x)) \\
	\chi (s,s,x)=x
\end{cases},
\]
and define analogously $\tilde{\chi} $ from $v $.% $\kappa$ from $w$ and $\kappa' $ from $w'$. 

We recall that from the Cauchy-Lipschitz Theorem we have  as $u ^0$ and $v ^0$ are $H^{+\infty}(\d) $ functions, 
then $u ^1$, $v ^1$ are  $H^{+\infty}(\d) $ and $u $ and $v $ are  $H^{+\infty}(\d) $ with respect to the $x$ variable thus $\chi ,\tilde{\chi} \in C^1([0,T]^2,C^{\infty} )$.
And they are both diffeomorphisms in the $x$ variable.\\

By the estimate \eqref{geometric proof quasi WW_study model pbm_subsec prereq cauchy pbm_thm cauchy pbm_eq1} $u$ and $v $ are uniformly bounded in $\dot{W}^{1,\infty}(\d) $ because
their Sobolev norms are dominated by those of $u ^1$ and $v ^1$ thus by those of $u ^0$ and $v ^0$ by Theorem \ref{geometric proof quasi WW_study model pbm_subsec prereq cauchy pbm_thm cauchy pbm linear pseudo}.  
 By classic manipulations of ODEs we get the estimates:
\begin{equation}\label{geometric proof quasi WW_study model pbm_subsec change of var by trans_proof eq 1}%%%%%%%%%ref%%%%%%%%%%
\begin{cases}
\exists C>0, \forall t',t \leq t_0, \forall x,  C^{-1}\leq \abs{\partial_x \chi (t,t',x)}\leq C \\
 \forall 2 \leq k< \lfloor s-\frac{1}{2}\rfloor , \norm{\partial^k_x \chi (t,t',x)}_{L^\infty } \leq C \norm{u}_{W^{k,\infty} }. 
 \end{cases}
\end{equation}
Analogous estimates hold for $\tilde{\chi} $ using $v $.\\
The classic transport computation reads:
\[
\begin{cases}
\partial_t(u (t, \chi (t,0,x)))=(\partial_t u)(t, \chi (t,0,x))+\partial_t( \chi (t,0,x) )(\partial_x u)(t, \chi (t,0,x)) \\
		\hspace{3cm}			=-(\partial_x \abs{D}^{\alpha-1} u)(t, \chi (t,0,x))\\
	\hspace{3cm}		=-(\partial_x \abs{D}^{\alpha-1})^*  (u(t, \chi (t,0,x))),\\
 u (0, \chi (0,0,x))=u (0,x)=u^1 (x).
 \end{cases}					
\]
where $(\cdot)^*  $ is the change of variables by $\chi (t,0,x)$ as presented in Theorem \ref{paracomposition_section Pull-back of pseudo and para- differential operators_theorem change of variable pseudo}.\\
Thus if we put $f $ the solution to the following Cauchy problem, which is well posed by Appendix \ref{geometric proof quasi WW_Appendix energy est of pulled back eq}:
 \begin{equation} \label{geometric proof quasi WW_study model pbm_subsec change of var by trans_proof eq 1} %%%%%%%%%ref%%%%%%%%%%
\begin{cases}
 \partial_t f+(\partial_x \abs{D}^{\alpha-1})^*  f=0\\
 \forall x \in \d, f(0,x)=u ^1(x)
\end{cases}
\end{equation}
 we get:
\begin{align} \label{geometric proof quasi WW_study model pbm_subsec change of var by trans_proof eq 2} %%%%%%%%%ref%%%%%%%%%%
u (t,\chi (t,0,x))=f (t,x)  \Leftrightarrow  u (t,x)=f (t, \chi (0,t,x)).
\end{align}
Analogously, if we put $g $ the solution to the well posed Cauchy problem,
\begin{equation} \label{geometric proof quasi WW_study model pbm_subsec change of var by trans_proof eq 3} %%%%%%%%%ref%%%%%%%%%%
\begin{cases}
 \partial_t g+\widetilde{(\partial_x \abs{D}^{\alpha-1})}^* g=0\\
 \forall x \in \d, g(0,x)=v ^1(x)
\end{cases}
\end{equation}
where $\widetilde{(\cdot)}^* $ is the change of variables by $\tilde{\chi} (t,0,x)$, we get
\begin{align} \label{geometric proof quasi WW_study model pbm_subsec change of var by trans_proof eq 4} %%%%%%%%%ref%%%%%%%%%%
v (t,x)=g (t, \tilde{\chi} (0,t,x))  \Leftrightarrow  v (t,\tilde{\chi} (t,0,x))=g (t, x).
\end{align}

Returning to the ODEs defining $\chi $ and $\tilde{\chi} $, for a generic initial time $0\leq t' \leq t_0$ we get:
\begin{align} \label{geometric proof quasi WW_study model pbm_subsec change of var by trans_proof eq 5} %%%%%%%%%ref%%%%%%%%%%
\begin{cases}
\chi (t,t',x)=x+\int_{t'}^t f (s,x) ds, \\
\tilde{\chi} (t,t',x)=x+\int_{t'}^t g (s,x) ds,\\
\end{cases}
\end{align}

\begin{proposition}\label{geometric proof quasi WW_study model pbm_subsec change of var by trans_proof prop 1}
There exists $C>0$ independent of $(\tau,\eps,\lambda)$ such that:\\
$\forall h \in H^s(\d),\forall (t,t')\leq t_0,$
\[ C^{-1}\norm{h}_{H^s}\leq\norm{h\circ \chi (t,t',x) }_{H^s}\leq C\norm{h}_{H^s}, \]
\[C^{-1}\norm{h}_{H^s}\leq\norm{h\circ \tilde{\chi} (t,t',x) }_{H^s}\leq C\norm{h}_{H^s}.\]
\end{proposition}

\begin{proof}
We will start by proving the upper bound for the estimate on the composition with $\chi$.
As $u$ is bounded in $(\tau,\eps,\lambda)$ on $C([0,T],H^s(\d))$ then there exists a unique solution $H\in C([0,T],H^s(\d))$ to
\[
\begin{cases}
\partial_t H+u\partial_x H=0,\\
H(t,x)=h(x),
\end{cases}
\]
and $H$ is bounded in $(\tau,\eps,\lambda)$ on $C([0,T],H^s(\d))$. The desired bound come from the fact that we have the explicit formula for $H$:
\[H(t',x)=h\circ \chi (t,t',x).\]
Now to get the lower bound it suffices to write by the upper bound computations:
\begin{align*}
\norm{h}_{H^s }&=\norm{h\circ \chi (t,t',x) \circ \chi (t',t,x)}_{H^s}\\
			&\leq C \norm{h\circ \chi (t,t',x)}_{H^s}.\\			
\end{align*}
We get analogously the estimates on the composition with $\tilde{\chi}$.
\end{proof}

\subsubsection{Key Lemma and proof of the Theorem}%%%eq2,th,lem1,rem,cor

\begin{lemma} \label{geometric proof quasi WW_study model pbm_subsec key lem_lem 1}%%%%%%%%%ref%%%%%%%%%%
Take $\eps'>0$ sufficiently small, as $0\leq\alpha< 2$ we can find a sequence $(\tau,\eps,\lambda)$ such that:
\begin{align}\label{geometric proof quasi WW_study model pbm_subsec key lem_lem 1_eq 1}
\begin{cases}
 \tau\rightarrow 0,\\
  \eps \rightarrow 0,\\
\lambda \rightarrow +\infty,
\end{cases}
&
\begin{cases}
 \tau \lambda ^{(\alpha-1)^+}\rightarrow0,\\
\eps^{-1}\lambda^{-1+(\alpha-1)^++\eps'}\rightarrow +\infty,\\
\lambda \eps \tau \rightarrow +\infty,\\
\lambda^{\alpha} \eps \tau^2 \rightarrow 0.
\end{cases}
\end{align}
Then there exists $c>0$ such that:
\begin{enumerate} 
\item For $\nu \geq 0$ and $\forall (\tau,\eps,\lambda), \norm{u ^0\circ \chi (0,\tau,x)-u ^0\circ \tilde{\chi} (0,\tau,x)}_{H^{s-\nu}}>c\lambda^{-\nu}. $
\item For $\nu \geq 0$
 \begin{multline}
  u (\tau,x)-v (\tau,x)=u ^0\circ \chi (0,\tau,x)-u ^0\circ \tilde{\chi} (0,\tau,x)\\+ O_{H^{s-\nu}}\big(\eps+(\tau\lambda^{(\alpha-1)^+}+\tau\lambda^{\alpha-(s-\nu)})\lambda^{-\nu}+\tau^2\eps\lambda^{\alpha-\nu} \big) .
  \end{multline}
\end{enumerate}
\end{lemma}
We will now show that this Lemma implies the Theorem \ref{geometric proof quasi WW_introduction_ theorem on dispersive Burgers eq}.
We have by combining the estimates $(1)$ and $(2)$ for $\nu=s$:
\[
\forall (\tau,\eps,\lambda), \norm{ u (\tau,x)-v (\tau,x)}_{H^s }>\frac{c}{2}>0 \text{ thus } \sup_{\tau,\eps,\lambda} \norm{ u (\tau,x)-v (\tau,x)}_{H^s }>\frac{c}{2}>0.
\]
Also by Theorem \ref{geometric proof quasi WW_study model pbm_subsec prereq cauchy pbm_thm cauchy pbm linear pseudo}: 
\[
\exists C>0,\norm{ u ^1(x)-v ^1(x)}_{H^s }\leq C\eps, \text{ thus } \norm{ u ^1(x)-v ^1(x)}_{H^s } \rightarrow 0,
\] 
which gives the non uniform continuity in the desired norms.\\
Now for the control in a weaker norm we write:
\[ \frac{\norm{ u (\tau,x)-v (\tau,x)}_{H^{s -1+(\alpha-1)^++\eps'}}}{\norm{ u ^1(x)-v ^1(x)}_{H^s }}\geq c\eps^{-1}\lambda^{-1+(\alpha-1)^++\eps'}\rightarrow +\infty, \]
which gives the desired result.
\subsubsection{Proof of point 1 of Lemma \ref{geometric proof quasi WW_study model pbm_subsec key lem_lem 1}}%%%eq5,th,lem,rem,cor
 We first prove that there exists $c>0$ such that
 $\norm{u ^0\circ \chi (0,\tau,x)}_{H^{s-\nu} }>c\lambda^{-\nu}$, indeed by Proposition \ref{geometric proof quasi WW_study model pbm_subsec change of var by trans_proof prop 1} and change of variable:
 \begin{equation}\label{geometric proof quasi WW_study model pbm_subsec proof key lem point 1_eq 1}
 \norm{u ^0\circ \chi (0,\tau,x)}_{H^{s-\nu} }\geq C^{-1}\norm{u ^0}_{H^{s-\nu} } \geq C^{-1}\lambda^{-\nu}\norm{\omega}_{H^{s-\nu} }.
 \end{equation}
Now we will show that $u ^0\circ \chi (0,\tau,x)$ and $u ^0\circ \tilde{\chi} (0,\tau,x)$ have disjoint supports
 which will suffice to conclude given \eqref{geometric proof quasi WW_study model pbm_subsec proof key lem point 1_eq 1}.
 Put $y=\chi (0,\tau,x)$, thus $x=\chi (\tau,0,y)$. On the support of $u ^0\circ \chi (0,\tau,x)$ we have:
 \begin{itemize}
\item If $\d=\r$,  $\lambda \abs{y}\leq 1$. 
  \item If $\d=\t$,  $\forall k\in \n,2\pi k-1\leq \lambda  \abs{y}\leq 2\pi k+1$. 
  \end{itemize}
 We compute by the Taylor formula, since $x=\tilde{\chi} (\tau,0,y)$:
\begin{align}\label{geometric proof quasi WW_study model pbm_subsec proof key lem point 1_eq 2}
\tilde{\chi} (0,\tau,x)&=\tilde{\chi} (0,\tau,\tilde{\chi} (\tau,0,y))\\
&+(\chi (\tau,0,y)-\tilde{\chi} (\tau,0,y)) \int_0^1 \partial_y\tilde{\chi} (0,\tau,r \chi (\tau,0,y)+(1-r)\tilde{\chi} (\tau,0,y)) dr \nonumber \\
				&=y+(\chi (\tau,0,y)-\tilde{\chi} (\tau,0,y)) \int_0^1 \partial_y\tilde{\chi} (0,\tau,r \chi (\tau,0,y)+(1-r)\tilde{\chi} (\tau,0,y)) dr. \nonumber
\end{align}
First by \eqref{geometric proof quasi WW_study model pbm_subsec change of var by trans_proof eq 5},
\begin{multline*}
\partial_y\tilde{\chi} (0,\tau,r \chi (\tau,0,y)+(1-r)\tilde{\chi} (\tau,0,y))\\
=1+\int_0^{\tau} \partial_y [g (t,r\chi (\tau,0,y)+(1-r)\tilde{\chi} (\tau,0,y))] dt.
\end{multline*}
Thus by estimates of Theorem \ref{geometric proof quasi WW_Appendix energy est of pulled back eq_theorem on para eq}, taking $0<\delta<s-\frac{3}{2}$\footnote{Recall the notation $O_{\norm{ \ }}$ in \ref{paracomposition_section Notations and functional analysis}.}:
\begin{align*}
\partial_y\tilde{\chi} (0,\tau,r \chi (\tau,0,y)+(1-r)\tilde{\chi} (\tau,0,y))&=1+O_{L^\infty}\big(\tau[1+\norm{v^1 }_{H^{\frac{3}{2}+\delta} }+\norm{u^1 }_{H^{\frac{3}{2}+\delta} }]\big)\\
&=1+O_{L^\infty}(\tau),
\end{align*}
Which gives
\begin{equation}\label{geometric proof quasi WW_study model pbm_subsec proof key lem point 1_eq 3}
 \int_0^1 \partial_y\tilde{\chi} (0,\tau,r \chi (\tau,0,y)+(1-r)\tilde{\chi} (\tau,0,y)) dr=1+O_{L^\infty}\big(\tau\big).
\end{equation}
Now we estimate $\chi (\tau,0,y)-\tilde{\chi} (\tau,0,y)$, by \eqref{geometric proof quasi WW_study model pbm_subsec change of var by trans_proof eq 5} :
\begin{equation}\label{geometric proof quasi WW_study model pbm_subsec proof key lem point 1_eq 4}
\chi (\tau,0,y)-\tilde{\chi} (\tau,0,y)=\int_0^{\tau} f (t,y)-g (t,y)dt.\\
\end{equation}
Taking $0<\delta<s-\frac{3}{2}$, by estimates of Theorem \ref{geometric proof quasi WW_Appendix energy est of pulled back eq_theorem on para eq}:
\begin{align*}
f (t,y)&=f (0,y)+\int_0^t\partial_tf (r,y)dr=u^1 (y)+tO_{L^\infty}(\norm{u^1 }_{H^{\frac{1}{2}+\alpha+\delta} })\\&=u^1 (y)+O_{L^\infty}\big(t\eps\big).
\end{align*}
Analogously we get:
\[g (t,y)=v^1 (y)+O_{L^\infty}\big(t\eps\big).\]

Consider $\mu$ the solution to the Cauchy problem:
\begin{equation} \label{geometric proof quasi WW_study model pbm_subsec proof key lem point 1_eq 5} %%%%%%%%%ref%%%%%%%%%%
\begin{cases}
 \partial_t \mu +\partial_x \abs{D}^{\alpha-1} \mu =0\\
 \forall y \in \d, \mu(\tau,y)=\omega(y).
\end{cases}
\end{equation}
By definition:
\begin{align*}
u^1 (y)-v^1 (y)&=-\eps  \mu(0,y)=-\eps \omega(y)+\eps \int_{\tau}^0\partial_t \mu(t,y)dt\\&=-\eps \omega(y)+O_{L^\infty}\big(\eps \tau\big).
\end{align*}
Thus,
\[\chi (\tau,0,y)-\tilde{\chi} (\tau,0,y)=- \eps \tau\omega(y)+O_{L^{\infty}}\big(\tau^2\eps\big),\]
and finally we get in \eqref{geometric proof quasi WW_study model pbm_subsec proof key lem point 1_eq 2},
\begin{align*}			
	\tilde{\chi} (\tau,0,x)-y			
				&=- \eps \tau\omega(y)+O_{L^{\infty}}\big(\tau^2\eps\big).
\end{align*}
We get for $x \in \supp\ u ^0\circ \chi (0,\tau,\cdot) $:
\begin{itemize}
\item For $\d=\r$:
\[\lambda \abs{\tilde{\chi} (0,\tau,x)}\geq \tau\eps  \lambda -1+o_{L^\infty}\big(\tau\eps  \lambda\big) \geq 2 ,\]
 by hypothesis $  \tau\eps  \lambda \rightarrow+ \infty$, which gives the desired result.
\item For $\d=\t$ given an adequate choice of  $\tau, \eps $ and $\lambda $:
\[2 n\pi +1\leq \lambda \abs{\tilde{\chi} (0,\tau,x)} \leq 2(n+1)\pi -1,\]
Which again gives the desired result.
\end{itemize}

\subsubsection{Proof of point 2 of Lemma \ref{geometric proof quasi WW_study model pbm_subsec key lem_lem 1}}%%%eq6,th,lem,rem,cor
We start by writing:
\begin{align*}
u (t,x)-v (t,x)&=f (t, \chi (0,t,x)) -g (t, \tilde{\chi} (0,t,x))\\
			&=\underbrace{f (t, \chi (0,t,x))-f (t, \tilde{\chi} (0,t,x))}_{(1)}+(f -g )(t, \tilde{\chi} (0,t,x)).
\end{align*}
Term $(1)$ resembles the main term in the usual transport estimates we used in point 1 of the Lemma
 \footnote{Like the ones used in proving the quasi-linearity of the Burgers equation.} 
 but with a main difference of $f $ being some dispersed data and not compactly supported. 
The main trick here was to construct from $u^0 , v ^0$ the defocused data in the past $u^1 ,v^1 $ and use this as the initial data for $f $ and $g $.
\begin{align*}
u (\tau,x)-v (\tau,x)
			&=u ^0\circ \chi (0,\tau,x)-u ^0\circ \tilde{\chi} (0,\tau,x)\\
			&+(f -u^0 )(\tau, \chi (0,\tau,x))-(f -u^0 )(\tau, \tilde{\chi} (0,t,x))+(f -g )(\tau, \tilde{\chi} (0,\tau,x)).
\end{align*}
The idea is then to see that by definition of $l$:
$l (\tau,x)=u ^0(x)$ and we get:
\begin{align*}
u (\tau,x)-v (\tau,x)
			&=u ^0\circ \chi (0,\tau,x)-u ^0\circ \tilde{\chi} (0,\tau,x)\\
			&+\underbrace{(f -l )(\tau, \chi (0,\tau,x))-(f -l )(\tau, \tilde{\chi} (0,t,x))}_{(1)}+\underbrace{(f -g )(\tau, \tilde{\chi} (0,\tau,x))}_{2}.
\end{align*}
We start by estimating $(1)$, by Proposition \ref{geometric proof quasi WW_study model pbm_subsec change of var by trans_proof prop 1}:
\begin{align*}
 \norm{(f -l )(\tau, \chi (0,\tau,x))}_{H^s} \leq C  \norm{(f -l )(\tau, \cdot)}_{H^s}.			
\end{align*}
Now $f -l $ solve: 
\begin{equation}\label{geometric proof quasi WW_study model pbm_subsec proof key lem point 2_eq1}%%%%%%%%%ref%%%%%%%%%%
\begin{cases}
 \partial_t (f -l )+\partial_x \abs{D}^{\alpha-1} (f -l )=(\partial_x \abs{D}^{\alpha-1}-{\partial_x \abs{D}^{\alpha-1}}^* )f \\
 \forall x \in \d, (f -l )(0,x)=0.
\end{cases}
\end{equation}
Thus we have the estimates:
\begin{align} \label{geometric proof quasi WW_study model pbm_subsec proof key lem point 2_eq2} %%%%%%%%%ref%%%%%%%%%%
\norm{f -l (\tau,\cdot)}_{H^{\nu} }&\leq C \norm{(\partial_x \abs{D}^{\alpha-1}-{\partial_x \abs{D}^{\alpha-1}}^* )f }_{L^1([0,\tau],H^{\nu}) } \nonumber\\
					& \leq C\tau\norm{(\partial_x \abs{D}^{\alpha-1}-{\partial_x \abs{D}^{\alpha-1}}^* )f }_{L^\infty([0,\tau],H^{\nu}) }\nonumber\\
\intertext{By Theorem \ref{paracomposition_section Pull-back of pseudo and para- differential operators_theorem change of variable pseudo} and the Kato-Ponce commutator estimates \ref{paracomposition_Notations and functional analysis_KatoPonce commutator estimate},}
	\norm{f -l (\tau,\cdot)}_{H^{\nu} }& \leq C\tau\norm{(Id-D\chi (0,t,\chi (t,0,x)))}_{L^\infty}\norm{f}_{L^\infty([0,\tau],H^{\nu+\alpha})}\nonumber\\
					&+ C\tau\norm{Id-D\chi (0,t,\chi (t,0,x))}_{{L^\infty([0,\tau],W^{\nu,\infty}) }}\norm{ f }_{L^\infty([0,\tau],H^\alpha) }\nonumber\\
					&+C\tau \norm{Id-D\chi (0,t,\chi (t,0,x))}_{{L^\infty([0,\tau],W^{1,\infty}) }}\norm{ f }_{L^\infty([0,\tau],H^{\nu+\alpha-1}) },\nonumber\\
\intertext{Using Theorem \ref{geometric proof quasi WW_Appendix energy est of pulled back eq_theorem on regularization of elliptic evolution pde} and applying the Sobolev embedding Theorem with $\delta>0$ and $\delta<s-\alpha- \frac{1}{2}$, we get:} 
\norm{f -l (\tau,\cdot)}_{H^{\nu} }&\leq C (\tau\lambda^{(\alpha-1)^+}+\tau\lambda^{\alpha-\nu})\lambda^{\nu-s}.%\label{geometric proof quasi WW_study model pbm_subsec proof key lem point 2_eq exp 2+1/2}
\end{align}
Thus we get
\begin{equation*}
 \norm{(f -l )(\tau, \chi (0,\tau,x))}_{H^\nu }\leq C(\tau\lambda^{(\alpha-1)^+}+\tau\lambda^{\alpha-\nu})\lambda^{\nu-s}.
 \end{equation*}
 Analogously we get 
 \begin{equation*}
 \norm{(f -l )(\tau, \tilde{\chi} (0,\tau,x))}_{H^\nu }\leq C(\tau\lambda^{(\alpha-1)^+}+\tau\lambda^{\alpha-\nu})\lambda^{\nu-s},
 \end{equation*}
 which gives
  \begin{equation}\label{geometric proof quasi WW_study model pbm_subsec proof key lem point 2_eq3} %%%%%%%%%ref%%%%%%%%%%
 \norm{(1)}_{H^\nu }\leq C(\tau\lambda^{(\alpha-1)^+}+\tau\lambda^{\alpha-\nu})\lambda^{\nu-s}.
 \end{equation}
Now we estimate $(2)$ in the same manner, by Proposition \ref{geometric proof quasi WW_study model pbm_subsec change of var by trans_proof prop 1}:
\begin{align*}
 \norm{(f -g )(\tau, \tilde{\chi} (0,\tau,x))}_{H^\nu } \leq  \norm{(f -g )(\tau, \cdot)}_{H^\nu }
\end{align*}
$f -g $ solve: 
\begin{equation}\label{geometric proof quasi WW_study model pbm_subsec proof key lem point 2_eq4}%%%%%%%%%ref%%%%%%%%%%
\begin{cases}
 \partial_t (f -g )+H{\partial_x \abs{D}^{\alpha-1}}^*  (f -g )+({\partial_x \abs{D}^{\alpha-1}}^* -\widetilde{\partial_x \abs{D}^{\alpha-1}}^* )g=0 \\
 \forall x \in \d, (f -g )(0,x)=(u^1 -v^1 )(x).
\end{cases}
\end{equation}
By Theorem \ref{paracomposition_section Pull-back of pseudo and para- differential operators_theorem change of variable pseudo} and the Kato-Ponce commutator estimates \ref{paracomposition_Notations and functional analysis_KatoPonce commutator estimate},
\begin{align*}
&\norm{f -g (\tau,\cdot)}_{H^{\nu} }\\ 
&\leq C \norm{u_1-v_1}_{H^{\nu}}\\
&+ C\tau\norm{(D\tilde{\chi} (0,t,\tilde{\chi} (t,0,x))-D\chi (0,t,\chi (t,0,x)))}_{L^\infty}\norm{g}_{L^\infty([0,\tau],H^{\nu+\alpha})}\nonumber\\
					&+ C\tau\norm{D\tilde{\chi} (0,t,\tilde{\chi} (t,0,x))-D\chi (0,t,\chi (t,0,x))}_{{L^\infty([0,\tau],W^{\nu,\infty}) }}\norm{ g }_{L^\infty([0,\tau],H^\alpha) }\nonumber\\
					&+C\tau \norm{D\tilde{\chi} (0,t,\tilde{\chi} (t,0,x))-D\chi (0,t,\chi (t,0,x))}_{{L^\infty([0,\tau],W^{1,\infty}) }}\norm{ g }_{L^\infty([0,\tau],H^{\nu+\alpha-1}) }.
\end{align*}
Using Theorem \ref{geometric proof quasi WW_Appendix energy est of pulled back eq_theorem on regularization of elliptic evolution pde} and taking $0<\delta<s-\alpha-\frac{1}{2}$:
\begin{align*}
\norm{f -g (\tau,\cdot)}_{H^{\nu} }&\leq C(\eps+\tau^2 \eps \lambda^{\alpha}\lambda^{\nu-s}+\eps\tau+\tau\lambda^{(\alpha-1)^+}\lambda^{\nu-s}),
\end{align*}
 which gives
  \begin{equation}\label{geometric proof quasi WW_study model pbm_subsec proof key lem point 2_eq5} %%%%%%%%%ref%%%%%%%%%%
 \norm{(2)}_{H^\nu }\leq   C(\eps+\tau^2 \eps \lambda^{\alpha}\lambda^{\nu-s}+\eps\tau+\tau\lambda^{(\alpha-1)^+}\lambda^{\nu-s}),
 \end{equation}
  finishing the proof of Lemma \ref{geometric proof quasi WW_study model pbm_subsec key lem_lem 1} and Theorem  \ref{geometric proof quasi WW_introduction_ theorem on dispersive Burgers eq}.

\section{A technical generalization} \label{geometric proof quasi WW_a tech gen}%%%eq6,th,lem,rem,cor
The techniques used in the previous proof will be generalized but with some care in the estimates due to the non linearity we add to the dispersive term. This extra "complication" is crucial for our application to the Water Waves system.
\begin{theorem}\label{geometric proof quasi WW_a tech gen_main theorem}%%%%%%%%%%%ref%%%%%%%%%%
Consider five real numbers $ \alpha\in[0, 2[$, $s \in ]2+\frac{d}{2},+\infty[$ and $T>0$ and $(\beta,k) \in \r^+$ verifying: 
\[
\begin{cases}
k\geq 1, \ \ \ \ \beta \leq \alpha,\\
\beta<(k+1)\alpha-2k+1.
\end{cases}
\]
Consider a elliptic skew symmetric\footnote{Recall the notation $a^\top$ for the adjoin of an operator $a$.} $C^1$ symbol $a:[0,T]\times H^s(\d^d) \rightarrow \Gamma^{\alpha}_1(\d^d)$,
\[\text{i.e such that }\RE(a_t)=\frac{a_t+a_t^\top}{2} \text{ is bounded in } \Gamma^0_1(\d^d),\]
\[\exists C>0, \forall (t,u,x)\in [0,T]\times H^s(\d^d)\times\d^d,\forall \xi, \abs{\xi}\geq \frac{1}{2}, \abs{a(t,u,x,\xi)}\geq C\abs{\xi}^{\alpha}.\]

Moreover we suppose the following bounds on the nonlinearity in $a$:
\begin{equation}\label{geometric proof quasi WW_a tech gen_main theorem_est Dg 1}
\forall \mu \in  \r, \forall g \in H^\mu(\d^d), 
\norm{T_{D_u a}g+ \overline{T_{D_u a}g}}_{H^\mu}\leq M_0^{\beta}(D_{u} a+\overline{D_{u} a}) \norm{g}_{H^{\mu+\beta}}.
\end{equation}
\begin{equation}\label{geometric proof quasi WW_a tech gen_main theorem_est Dg 2}
\forall (t,u), M_1^\alpha(a)\leq C(1+\norm{u}_{W^{1,\infty}}) \ , \ 
M_0^{\beta}(D_{u} a+\overline{D_{u} a}(u,\cdot))\leq C\norm{u}_{L^\infty}^{k-1}.
\end{equation}
Consider a $C^1$ function $V(t,x,u):[0,T]\times\d^d\times H^s(\d^d) \rightarrow H^s(\d^d;\r^d)$ and a function $F \in L^\infty\big([0,T], W^{1,\infty}\big(H^s(\d^d), H^s(\d^d)\big)\big)$. 
 
Suppose that the following hypothesis \ref{geometric proof quasi WW_a tech gen_main theorem_H1} is verified,
there exists $ \omega \in C^\infty_c(\d^d)$ supported in $\b(0,1)$ such that
\[ \forall (t,x) \in [0,T]\times \supp \omega,C_x^{-1}t \leq  \abs{ \int_0^t D_uV(s,x,0)[\omega(x)]ds} \leq C_xt.
\tag{\textbf{H1}}\label{geometric proof quasi WW_a tech gen_main theorem_H1}
\]
for a constant $C_x>0$ when $\omega(x)\neq 0$.

Fix $u_0 \in H^{s}(\d^d)$ and take $r>0$, then there exists $T>0$ such that for all $v_0$ in the ball $ \b(u_0,r) \subset  H^{s}(\d^d)$ the Cauchy problem:
 \begin{equation} \label{geometric proof quasi WW_a tech gen_main theorem_equation 1}
\begin{cases} 
\partial_t v+T_{V(t,x,v)}\cdot \nabla v+T_{a(t,v)} v=F(t,v) \\
v(0,\cdot)=v_0(\cdot),
\end{cases}
 \end{equation}
has a unique solution $v \in C([0,T], H^{s}(\d^d))$. Moreover $\forall R>0$, the flow map: 	\begin{align*}
										\b(0,R) \rightarrow &C([0,T],H^s(\d^d))\\
										v_0 \mapsto &v
										\end{align*}
				is not uniformly continuous.		
					
Considering a weaker control norm we get, for all $\eps>0$ the flow map:
	\begin{align*}
										\b(0,R) \rightarrow &C([0,T],H^{s-1+(\alpha-1)^+ +\eps}(\d^d))\\
										v_0 \mapsto &v
										\end{align*}
				is not $C^1$.	
\end{theorem}

In the proof of quasi-linearity of the water waves systems we will need the following slight generalization to systems given by the following corollary.

\begin{corollary}\label{geometric proof quasi WW_a tech gen_cor theorem sys}%%%%%%%%%%%ref%%%%%%%%%%
Consider a positive integer $n \geq 1$ and five real numbers $ \alpha\in[0, 2[$, $s \in ]2+\frac{d}{2},+\infty[$, $T>0$ and $(\beta,k) \in \r^+$ verifying: 
\[
\begin{cases}
k\geq 1, \ \ \ \ \beta \leq \alpha,\\
\beta<(k+1)\alpha-2k+1.
\end{cases}
\]
Consider a $C^1$ skew symmetric elliptic symbol $a:[0,T]\times H^s(\d^d;\r^n) \rightarrow \Gamma^{\alpha}_1(\d^d;M_n(\r))$.

Moreover we suppose the following bounds on the nonlinearity in $a$:
\begin{equation}\label{geometric proof quasi WW_a tech gen_main theorem_est Dg 1}
\forall \mu \in  \r, \forall g \in H^\mu(\d^d), 
\norm{T_{D_u a}g+ \overline{T_{D_u a}g}}_{H^\mu}\leq M_0^{\beta}(D_{u} a+\overline{D_{u} a}) \norm{g}_{H^{\mu+\beta}}.
\end{equation}
\begin{equation}\label{geometric proof quasi WW_a tech gen_main theorem_est Dg 2}
\forall (t,u), M_1^\alpha(a)\leq C(1+\norm{u}_{W^{1,\infty}}) \ , \ 
M_0^{\beta}(D_{u} a+\overline{D_{u} a}(u,\cdot))\leq C\norm{u}_{L^\infty}^{k-1}.
\end{equation}

Consider a $C^1$ function
  $V(t,x,u):[0,T]\times\d^d\times H^s(\d^d;\c^n) \rightarrow H^s(\d^d;\r^n)$
 and
a function $F \in L^\infty\big([0,T], W^{1,\infty}\big(H^s(\d^d;\r^n), H^s(\d^d;\r^n)\big)\big)$.  
  
Suppose that the following hypothesis \ref{geometric proof quasi WW_a tech gen_cor theorem sys_H1} is verified,
there exists $ \omega \in C^\infty_c(\d^d;\r^n)$ supported in $\b(0,1)$ such that
\[ \forall (t,x) \in [0,T]\times \supp \omega,C_x^{-1}t \leq  \abs{ \int_0^t D_uV(s,x,0)[\omega(x)]ds} \leq C_xt,
\tag{\textbf{H1}}\label{geometric proof quasi WW_a tech gen_cor theorem sys_H1}
\] 
for a constant $C_x>0$ when $\omega(x)\neq 0$.

Fix $u_0 \in H^s(\d^d;\r^n) $ and take $r>0$, then there exists $T'>0$ such that for all $v_0$ in the ball $ \b(u_0,r) \subset  H^s(\d^d;\r^n) $ the Cauchy problem:
\[
\begin{cases}
\forall i \in [1,..,n], \ \partial_t v_i+V(t,x,v)\cdot \nabla v_i+(T_{a(t,v)} v)_i=F_i(t,v), \\
v(0,\cdot)=v_0(\cdot),
\end{cases}
\]
has a unique solution $v \in C([0,T'], H^s(\d^d;\r^n))$. Moreover $\forall R>0$, the flow map: 	\begin{align*}
										\b(0,R) \rightarrow &C([0,T'],H^s(\d^d;\r^n))\\
										v_0 \mapsto &v
										\end{align*}
				is not uniformly continuous.		
						
Considering a weaker control norm we get, for all $\eps'>0$ the flow map:
	\begin{align*}
										\b(0,R) \rightarrow &C([0,T'],H^{s-1+(\alpha-1)^+ +\eps'}(\d^d;\r^n))\\
										v_0 \mapsto &v
										\end{align*}
				is not $C^1$.
\end{corollary}

\subsection{Prerequisites on the Cauchy problem}
  We consider the Cauchy problem associated to Theorem \ref{geometric proof quasi WW_a tech gen_main theorem}:
\begin{equation}
\begin{cases} 
\partial_t u+T_{V(t,x,u)}\cdot \nabla u+T_a u=F(t,u)\\
u(0,\cdot)=u_0(\cdot) \in H^s(\d^d), \ s>1+\frac{d}{2},
\end{cases}
 \label{geometric proof quasi WW_a tech gen_subsec prereq CP eq1}%%%%%%%%%ref%%%%%%%%%%
\end{equation}
\begin{theorem} \label{geometric proof quasi WW_a tech gen_subsec prereq CP_th1} %%%%%%%%%ref%%%%%%%%%%
Consider $0 \leq \alpha< 2$, $T>0$, an elliptic $C^1$ symbol $a:[0,T]\times H^s(\d^d) \rightarrow \Gamma^{\alpha}_1(\d^d)$ skew symmetric
\[\text{i.e such that} \ \ \RE(a_t)=\frac{a_t+a_t^\top}{2} \text{ is bounded in }  \Gamma^0_1(\d^d).\]

Moreover we suppose the following bounds on the nonlinearity in $a$:
\begin{equation}\label{geometric proof quasi WW_a tech gen_main theorem_est Dg 1}
\forall \mu \in  \r, \forall g \in H^\mu(\d^d), 
\norm{T_{D_u a}g+ \overline{T_{D_u a}g}}_{H^\mu}\leq M_0^{\beta}(D_{u} a+\overline{D_{u} a}) \norm{g}_{H^{\mu+\beta}}.
\end{equation}

Consider a $C^1$ function $V(t,x,u):[0,T]\times\d^d\times H^s(\d^d) \rightarrow H^s(\d^d;\r^d)$ and a function $F \in L^\infty\big([0,T], W^{1,\infty}\big(H^s(\d^d), H^s(\d^d)\big)\big)$.

Consider $\dps s>1+\frac{d}{2}$, $r>0$ and $u_0\in H^{s}(\d^d)$ such that :
\[\forall  v_0 \in \b(u_0,r), T\norm{\nabla_x v_0}_{L^\infty}<1.\]
 Then the problem \eqref{geometric proof quasi WW_a tech gen_subsec prereq CP eq1} with initial data $v_0$ has a unique solution $v \in C^0([0,T],H^{s}(\d^d))$ and the map $v_0 \mapsto v$ is continuous from $\b(u_0,r)$ to $C^0([0,T],H^{s}(\d^d))$.
Moreover we have the estimates:
\begin{equation} \label{geometric proof quasi WW_a tech gen_subsec prereq CP_th1_eq1} %%%%%%%%%ref%%%%%%%%%%
\forall 0 \leq \mu \leq s, \norm{v(t)}_{H^{\mu}(\d^d)} \leq C_\mu  \norm{v_0}_{H^{\mu}(\d^d)}. 
%\forall 0 \leq k \leq s-\frac{1}{2}, \norm{u(t)}_{W^{k,\infty}(\r)} \leq C_\mu  \norm{u_0}_{W^{k,\infty}(\r)}.
\end{equation}
Taking Two different solution $v,v'$, assuming moreover $v_0 \in H^{s+\alpha}(\d^d)$ then we have:
\begin{align} \label{geometric proof quasi WW_a tech gen_subsec prereq CP_th1_eq2}%%%%%%%%%ref%%%%%%%%%%
\norm{(v-v')(t)}_{H^{s}(\d^d)} &\leq  \norm{v_0-v'_0}_{H^{s}(\d^d)}e^{C_s \int_0^t \norm{v(s)}_{H^{s+\beta}(\d^d)}}ds.
\end{align}
\end{theorem}

We will also work with hyperbolic paradifferential  equations and we summarize the properties needed in the following Theorem:
\begin{theorem} \label{geometric proof quasi WW_a tech gen_subsec prereq CP_th2} %%%%%%%%%ref%%%%%%%%%%
Consider$(a_t)_{t\in \r}$ a family of symbols in $\Gamma^\beta_1(\d^d)$ with $\beta \in \r$, such that
 $t \mapsto a_t$ is continuous and bounded from $\r$ to $\Gamma^\beta_1(\d^d)$
  and such that $ \RE(a_t)=\frac{a_t+a_t^\top}{2}$ is bounded in $\Gamma^0_1(\d^d)$, and take $T>0$.
   Then for all initial data $u_0 \in H^s(\d^d)$, and $f \in  C^0([0,T];H^s(\d^d))$ the Cauchy problem:
 \begin{equation} \label{geometric proof quasi WW_a tech gen_subsec prereq CP_th2_eq1}%%%%%%%%%ref%%%%%%%%%%
\begin{cases}
 \partial_t u+T_au=f\\
 \forall x \in \d^d, u(0,x)=u_0(x)
\end{cases}
\end{equation}
has a unique solution $u \in C^0([0,T];H^s(\d^d))\cap C^1([0,T];H^{s-\beta}(\d^d))$ which verifies the estimates:
\[\norm{u(t)}_{H^s(\d^d)}\leq e^{Ct}\norm{u_0}_{H^s(\d^d)}+2\int_0^te^{C(t-t')}\norm{f(t')}_{H^s(\d^d)}dt',\]
where C depends on a finite symbol semi-norm $M_{1}^0(\RE(a_t))$.
\begin{remark}
We will also need to remark that fixing the initial data at 0 is an arbitrary choice.
More precisely, $\forall 0 \leq t_0\leq T$ and all data $u_0 \in H^s(\d^d)$  the Cauchy problem:
 \begin{equation} \label{geometric proof quasi WW_Appendix energy est of pulled back eq_theorem on para eq_eq 4} %%%%%%%%%ref%%%%%%%%%%
\begin{cases}
 \partial_t u+T_au=f\\
 \forall x \in \d^d, u(t_0,x)=u_0(x)
\end{cases}
\end{equation}
has a unique solution $u \in C^0([0,T];H^s(\d^d))\cap C^1([0,T];H^{s-\beta}(\d^d))$ which verifies the estimate:
\[\norm{u(t)}_{H^s(\d^d)}\leq e^{C\abs{t-t_0}}\norm{u_0}_{H^s(\d^d)}+2\abs{\int_{t_0}^te^{C(t-t')}\norm{f(t')}_{H^s(\d^d)}dt'}.\]
\end{remark}
\end{theorem} 

\subsection{Proof of Theorem \ref{geometric proof quasi WW_a tech gen_main theorem}} \label{geometric proof quasi WW_a tech gen}%%%eq6,th,lem,rem,cor
As for Theorem \ref{geometric proof quasi WW_introduction_ theorem on dispersive Burgers eq}, for the proof we will show that there exists a positive constant $C$ and two sequences $(u^\lambda_{\eps,\tau})$ and $(v^\lambda_{\eps,\tau})$ solutions of \eqref{geometric proof quasi WW_a tech gen_main theorem_equation 1} on $[0,1]$ such that for every $t\in [0,1]$,
\[
\sup_{\lambda,\eps,\tau} \norm{u^\lambda_{\eps,\tau}}_{L^\infty([0,1],H^s(\d^d))}+
\norm{v^\lambda_{\eps,\tau}}_{L^\infty([0,1],H^s(\d^d))}\leq C,
\]
$(u^\lambda_{\eps,\tau})$ and $(v^\lambda_{\eps,\tau})$ satisfy initially
\[
\lim_{\substack{\lambda \rightarrow +\infty \\ \eps,\tau\rightarrow 0}} \norm{u^\lambda_{\eps,\tau}(0,\cdot)-v^\lambda_{\eps,\tau}(0,\cdot)}_{H^s(\d^d)}=0,
\]
but,
\[
\liminf_{\substack{\lambda \rightarrow +\infty \\ \eps,\tau\rightarrow 0}} \norm{u^\lambda_{\eps,\tau}-v^\lambda_{\eps,\tau}}_{L^\infty([0,1],H^s(\d^d))}\geq c>0.
\]

 Considering a weaker control norm we want to get, for all $\eps'>0$,
 \[
\liminf_{\substack{\lambda \rightarrow +\infty \\ \eps,\tau\rightarrow 0}} \frac{\norm{u^\lambda_{\eps,\tau}-v^\lambda_{\eps,\tau}}_{L^\infty([0,1],H^{s-1+(\alpha-1)^+ +\eps'}(\d^d))}}{\norm{u^\lambda_{\eps,\tau}(0,\cdot)-v^\lambda_{\eps,\tau}(0,\cdot)}_{H^s(\d^d)}}=+\infty.
\]

\subsubsection{Definition of the Ansatz}

 Let $(\lambda ,\eps )$ be two positive real sequences such that
\begin{equation} \label{geometric proof quasi WW_a tech gen_subsec def ansatz_eq1} %%%%%%%%%ref%%%%%%%%%%
 \lambda  \rightarrow + \infty, \  \eps  \rightarrow 0, \ \lambda \eps   \rightarrow + \infty,
\end{equation}
and put 
\begin{itemize}
\item on $\r^d$,
\[u ^0(x)=\lambda ^{\frac{d}{2}-s}\omega(\lambda x), \ v ^0(x)=u ^0(x)+ \eps  \omega(x), \]
\item on $\t^d$, $u ^0$ and $v^0$ as the periodic extensions of the functions defined above.
\end{itemize}
Take $t_0>0$ smaller than a harmless constant which will be fixed later, and $(\tau),
\\ 0 <\tau \leq t_0$. \\
Now let $l  $ be the solutions to the Cauchy problem on $[0,t_0]$:
\begin{equation} \label{geometric proof quasi WW_a tech gen_subsec def ansatz_eq2}%%%%%%%%%ref%%%%%%%%%%
\begin{cases}
 \partial_t l +T_{a(t,l)} l =F(t,l)\\
 \forall x \in \d^d, l(\tau,x)=u ^0(x).
\end{cases}
\end{equation}
Put $u ^1(x)=l (0,x)$ and define $l' $
to be the solutions to the Cauchy problem on $[0,t_0]$:
\begin{equation} \label{geometric proof quasi WW_a tech gen_subsec def ansatz_def l'} %%%%%%%%%ref%%%%%%%%%%
\begin{cases}
 \partial_t l' +T_{a(t,l)} l' =F(t,l')\\
 \forall x \in \d^d, l(\tau,x)=v ^0(x).
\end{cases}
\end{equation}
 and put $v ^1=l(0,x)$.
\begin{remark}\label{geometric proof quasi WW_a tech gen_subsec def ansatz_semi lin intial datum}
It's important to notice that we use the same term $T_{a(t,l)}$ in \eqref{geometric proof quasi WW_a tech gen_subsec def ansatz_eq2}and \eqref{geometric proof quasi WW_a tech gen_subsec def ansatz_def l'} and thus $(l,l')$ have Lipschitz dependence on the data $(u^0,v^0)$.
\end{remark}
Define $u $ and $v $ as the solution given by Theorem \ref{geometric proof quasi WW_a tech gen_subsec prereq CP_th1} with initial data $u ^1$ and $v ^1$ on the intervals $[0,T]$, $[0,T']$. 
Taking $0<\delta<s-1-\frac{d}{2}$, $u ^0$ and $v ^0$ are uniformly bounded in $H^{1+\frac{d}{2}+\delta}(\d^d) $ and thus by Theorem \ref{geometric proof quasi WW_a tech gen_subsec prereq CP_th2}, $u ^1$ and $v ^1$ are also
uniformly bounded in $H^{1+\frac{d}{2}+\delta}(\d^d) $ and thus by the Sobolev injection Theorems they are bounded in $\dot{W}^{1,\infty}(\d^d) $. 
Thus we can take a uniform $0<T$ on which all the solutions are well defined and we take $0<t_0\leq T$.

\subsubsection{Change of variables by transport}

Put 
\[
\begin{cases}
	\frac{d}{dt}\chi (t,s,x)=V\big(t,\chi (t,s,x),u(t,\chi (t,s,x))\big) , \\
	\chi (s,s,x)=x,
\end{cases}
\]
and define analogously $\tilde{\chi} $ from $v $. 
We recall that from the Cauchy-Lipschitz Theorem as $u ^0$ and $v ^0$ are $H^{+\infty}(\d^d) $ functions, 
then $u ^1$, $v ^1$ are  $H^{+\infty} $ and $u $ and $v $ are  $H^{+\infty} (\d^d)$ with respect to the $x$ variable thus $\chi ,\tilde{\chi} \in C^1([0,T]^2,C^{\infty}(\d^d) )$. 
And they are both diffeomorphisms in the $x$ variable.\\

By the estimate \eqref{geometric proof quasi WW_study model pbm_subsec prereq cauchy pbm_thm cauchy pbm_eq1} $u $ and $v $ are uniformly bounded in $\dot{W}^{1,\infty}(\d^d) $ say by $M>0$
 and their Sobolev norms are dominated by those of $u ^1$ and $v ^1$ thus by those of $u ^0$ and $v ^0$ by Theorem \ref{geometric proof quasi WW_a tech gen_subsec prereq CP_th2}.  
 By classic manipulations of ODEs we get the estimates:
\begin{equation}\label{geometric proof quasi WW_a tech gen_subsec change of var by trans_eq1}%%%%%%%%%ref%%%%%%%%%%
\begin{cases}
\exists C>0, \forall (t',t)\in [0, t_0], \forall x \in \d^d,  C^{-1}\leq \abs{D  \chi (t,t',x)}\leq C \\
 \forall 2 \leq k< \lfloor s -\frac{d}{2}\rfloor , \norm{D^k  \chi (t,t',x)}_{L^\infty } \leq C \norm{ u }_{W^{k,\infty} } 
 \end{cases}
\end{equation}
Analogous estimates hold for $\tilde{\chi} $ using $v $.\\
Now we compute the analogue of the classic transport computation but with the paracomposition operator which reads:
\begin{align*}
\partial_t(\chi(t,0,x)^* u(t,x))&=\chi(t,0,x)^*\partial_tu+T_{\partial_t \chi(t,0,x)}\cdot\chi(t,0,x)^* \nabla u(t,x)+R(t,u) \\
	&=-\chi(t,0,x)^*(T_{a(t,u)}u)(t, x)+\chi(t,0,x)^*F(t,u)+R(t,u)\\
 &=-T_{a(t,u)^*}  \chi(t,0,x)^*u(t, x)+\chi(t,0,x)^*F(t,u)+R(t,u)+R'(t,u),\\
 \chi (0,0,x)^*u (0, x)&=u (0,x)=u^1 (x).
\end{align*}
where $(\cdot)^*  $ is the change of variables by $\chi (t,0,x)$ as presented in Theorem \ref{paracomposition_section Paracomposition_subsec paracomp on the euclidean space_theorem paralinearisation of composition}. We can assemble the terms $R,R'$ and $F$ in a new term $F'$ verifying the same hypothesis as F, thus without loss of generality henceforth we will keep the generic notation F for all the terms verifying the same hypothesis. \\
Thus if we put $f $ the solution to the Cauchy problem, which is well posed by Appendix \ref{geometric proof quasi WW_Appendix energy est of pulled back eq}:
 \begin{equation} \label{geometric proof quasi WW_a tech gen_subsec change of var by trans_eq2} %%%%%%%%%ref%%%%%%%%%%
\begin{cases}
 \partial_t f+ T_{a(t,u)^*}   f=\chi (t,0,x)^*F(t,u)\\
 \forall x \in \d^d, f(0,x)=u ^1(x)
\end{cases}
\end{equation}
 we get:
\begin{align} \label{geometric proof quasi WW_a tech gen_subsec change of var by trans_eq3}%%%%%%%%%ref%%%%%%%%%%
\chi (t,0,x)^*u (t,x)=f (t,x).
\end{align}
Analogously, if we put $g $ the solution to the well posed Cauchy problem,
\begin{equation} \label{geometric proof quasi WW_a tech gen_subsec change of var by trans_eq4}%%%%%%%%%ref%%%%%%%%%%
\begin{cases}
 \partial_t g+ T_{\widetilde{a(t,v)}^*}  g=\tilde{\chi} (t,0,x)^*F(t,v)\\
 \forall x \in \d^d, g(0,x)=v ^1(x)
\end{cases}
\end{equation}
where $\widetilde{(\cdot)}^* $ is the change of variables by $\tilde{\chi} (t,0,x)$, we get
\begin{align} \label{geometric proof quasi WW_a tech gen_subsec change of var by trans_eq5}%%%%%%%%%ref%%%%%%%%%%
 \tilde{\chi} (t,0,x)^*v (t,x)=g (t, x).
\end{align}
Returning to the ODEs defining $\chi $ and $\tilde{\chi} $ we get:
\begin{equation} \label{geometric proof quasi WW_a tech gen_subsec change of var by trans_eq6} %%%%%%%%%ref%%%%%%%%%%
\begin{cases}
\chi (t,t',x)=x+\int_{t'}^t V(s,\chi (s,t',x),f (s,x)))  ds, \\
\tilde{\chi} (t,t',x)=x+\int_{t'}^t V(s,\tilde{\chi} (s,t',x),g (s,x))  ds.
\end{cases}
\end{equation}
\begin{proposition}\label{geometric proof quasi WW_a tech gen_subsec change of var by trans_ prop 1}
There exists $C>0$ independent of $(\tau,\eps,\lambda)$ such that:\\
$\forall h \in H^s(\d^d),\forall (t,t')\leq t_0,$
\[ C^{-1}\norm{h}_{H^s}\leq\norm{h\circ \chi (t,t',x) }_{H^s}\leq C\norm{h}_{H^s}, \]
\[C^{-1}\norm{h}_{H^s}\leq\norm{h\circ \tilde{\chi} (t,t',x) }_{H^s}\leq C\norm{h}_{H^s}.\]
\end{proposition}

\begin{proof}
We will start by proving the upper bound for the estimate on the composition with $\chi$.
As $u$ is bounded in $(\tau,\eps,\lambda)$ on $C([0,T],H^s(\d^d))$ then there exists a unique solution $H\in C([0,T],H^s(\d))$ to
\[
\begin{cases}
\partial_s H(s,x)+V(s,x,u)\cdot \nabla H(s,x)=0\\
H(t,x)=h(x)
\end{cases}
\]
and $H$ is bounded in $(\tau,\eps,\lambda)$ on $C([0,T],H^s(\d^d))$. The desired bounds come from the fact that we have the explicit formula for $H$:
\[H(t',x)=h\circ \chi (t,t',x).\]
Now to get the lower bound it suffices to write by the upper bound computations:
\begin{align*}
\norm{h}_{H^s }&=\norm{h\circ \chi (t,t',x) \circ \chi (t',t,x)}_{H^s}\\
			&\leq C \norm{h\circ \chi (t,t',x)}_{H^s}.\\			
\end{align*}
We get analogously the estimates on the composition with $\tilde{\chi}$.
\end{proof}
%\[\norm{\int_0^\tau \int_0^1 D_uV(t,\chi(t,0,y),r f(t,y)+(1-r)g(t,y))drdt[\omega(y)]}^{-1}\]
%%%%%%%%%%%%%%%%%%%%%%%%%
\subsubsection{Key Lemma and proof of the Theorem}%%%eq2,th,lem1,rem,cor

\begin{lemma} \label{geometric proof quasi WW_a tech gen_subsec key lem}%%%%%%%%%ref%%%%%%%%%%
 As $0\leq\alpha< 2$ we can find a sequence $(\tau,\eps,\lambda)$ such that for all $\eps'>0$ sufficiently small:
\begin{align}\label{geometric proof quasi WW_study model pbm_subsec key lem_lem 1_eq 1}
\begin{cases}
\tau\rightarrow 0, \\
\tau= \lambda ^{1-\alpha}, \text{ for }\alpha>1,\\
\tau\eps^k\lambda^{\beta}\rightarrow 0,
\end{cases}
&
\begin{cases}
\tau^2\eps \lambda^\alpha\rightarrow 0\\
 \eps^{-1}\lambda^{-1+(\alpha-1)^+-\eps'}\rightarrow +\infty,\\
 \tau \lambda \eps \rightarrow +\infty.
\end{cases}
\end{align}

Then there exists $c>0$ such that:
\begin{enumerate} 
\item $\forall (\tau,\eps,\lambda,\nu), \norm{u ^0\circ \chi (0,\tau,x)-u ^0\circ \tilde{\chi} (0,\tau,x)}_{H^{s-\nu}}>c\lambda^{-\nu}. $
\item For $\delta$ such that $0<\delta<s-1-\frac{d}{2}$:
\[ u (\tau,x)-v (\tau,x)=u ^0\circ \chi (0,\tau,x)-u ^0\circ \tilde{\chi} (0,\tau,x)+ O_{H^{s-\nu}}\big(\eps+[\tau^2\lambda^{\alpha-1}+\tau^2\eps \lambda^\alpha+\tau \eps^k \lambda^\beta ]\lambda^{-\nu}\big) .
\]
\end{enumerate}
\end{lemma}
We will now show that this Lemma implies the Theorem.
We have by combining the estimates $(1)$ and $(2)$ for $\nu=s$:
\[
\forall (\tau,\eps,\lambda), \norm{ u (\tau,x)-v (\tau,x)}_{H^s }>\frac{c}{2}>0  \text{ thus } \sup_{\tau,\eps,\lambda} \norm{ u (\tau,x)-v (\tau,x)}_{H^s }>\frac{c}{2}>0
\]
Also by Theorem \ref{geometric proof quasi WW_a tech gen_subsec prereq CP_th2} and Remark \ref{geometric proof quasi WW_a tech gen_subsec def ansatz_semi lin intial datum}: 
\[
\exists C>0,\norm{ u ^1(x)-v ^1(x)}_{H^s }\leq C\eps \text{ thus } \norm{ u ^1(x)-v ^1(x)}_{H^s } \rightarrow 0,
\] 
which gives the non uniform continuity in the desired norms.
Now for the control in a weaker norm we write:
\[ \frac{\norm{ u (\tau,x)-v (\tau,x)}_{H^{s -1+(\alpha-1)^+-\eps'}}}{\norm{ u ^1(x)-v ^1(x)}_{H^s }}\geq c\eps^{-1}\lambda^{-1+(\alpha-1)^+-\eps'}\rightarrow +\infty, \]
which gives the desired result.

\subsubsection{Proof of point 1 of Lemma \ref{geometric proof quasi WW_a tech gen_subsec key lem}}%%%eq5,th,lem,rem,cor

 We first prove that $\exists c>0$ such that
 $$\norm{u ^0\circ \chi (0,\tau,x)}_{H^s }>c\lambda^{-\nu},$$ indeed by Proposition \ref{geometric proof quasi WW_a tech gen_subsec change of var by trans_ prop 1}  and change of variable:
 \begin{equation}\label{geometric proof quasi WW_a tech gen_subsec proof point 1 key lem_eq1}
 \norm{u ^0\circ \chi (0,\tau,x)}_{H^{s-\nu} }\geq C^{-1}\norm{u ^0}_{H^{s-\nu} } \geq C^{-1}\lambda^{-\nu}\norm{\omega}_{H^{s-\nu} }.
 \end{equation}
Now we will show that $u ^0\circ \chi (0,\tau,x)$ and $u ^0\circ \tilde{\chi} (0,\tau,x)$ have disjoint supports
 which will suffice to conclude given \eqref{geometric proof quasi WW_a tech gen_subsec proof point 1 key lem_eq1}.
 Put $y=\chi (0,\tau,x)$, thus $x=\chi (\tau,0,y)$. On the support of $u ^0\circ \chi (0,\tau,x)$ we have:
 \begin{itemize}
\item If $\d^d=\r^d$,  $\lambda \abs{y}\leq 1$. 
  \item If $\d^d=\t^d$,  $\forall k\in \n,2\pi k-1\leq \lambda  \abs{y}\leq 2\pi k+1$. 
  \end{itemize}
 We then compute by the Taylor formula:
\begin{align}\label{geometric proof quasi WW_a tech gen_subsec proof point 1 key lem_eq2}%%%%%%%%%ref%%%%%%%%%%
\tilde{\chi} (0,\tau,x)&=\tilde{\chi} (0,\tau,\tilde{\chi} (\tau,0,y)) \nonumber\\
&+ \int_0^1 D_y\tilde{\chi} (0,\tau,r \chi (\tau,0,y)+(1-r)\tilde{\chi} (\tau,0,y)) dr[\chi (\tau,0,y)-\tilde{\chi} (\tau,0,y)] \\
				&=y+ \int_0^1 D_y\tilde{\chi} (0,\tau,r \chi (\tau,0,y)+(1-r)\tilde{\chi} (\tau,0,y)) dr[\chi (\tau,0,y)-\tilde{\chi} (\tau,0,y)]. \nonumber
\end{align}
First,
\begin{multline*}
D_y\tilde{\chi} (0,\tau,r \chi (\tau,0,y)+(1-r)\tilde{\chi} (\tau,0,y))\\
=Id+\int_0^{\tau} D_y [V(t,\tilde{\chi} (0,\tau,r \chi (\tau,0,y)+(1-r)\tilde{\chi} (\tau,0,y)),g (t,r\chi (\tau,0,y)+(1-r)\tilde{\chi} (\tau,0,y)))] dt.
\end{multline*}
Thus by estimates of Theorem \ref{geometric proof quasi WW_Appendix energy est of pulled back eq_theorem on para eq}taking $0<\delta<s-\frac{d}{2}-1$:
\begin{align*}
D_y\tilde{\chi} (0,\tau,r \chi (\tau,0,y)+(1-r)\tilde{\chi} (\tau,0,y))&=Id+O_{L^\infty}(\tau(\norm{v^1 }_{H^{1+\frac{d}{2}+\delta} }
+\norm{u^1 }_{H^{1+\frac{d}{2}+\delta} }))\\
&=Id+O_{L^\infty}(\tau),
\end{align*}
Which gives
\begin{equation}\label{geometric proof quasi WW_a tech gen_subsec proof point 1 key lem_eq3}%%%%%%%%%ref%%%%%%%%%%
 \int_0^1 D_y\tilde{\chi} (0,\tau,r \chi (\tau,0,y)+(1-r)\tilde{\chi} (\tau,0,y)) dr=Id+O_{L^\infty}(\tau).
\end{equation}
Now we estimate $\chi (\tau,0,y)-\tilde{\chi} (\tau,0,y)$, by \eqref{geometric proof quasi WW_a tech gen_subsec change of var by trans_eq6} :
\begin{align}\label{geometric proof quasi WW_a tech gen_subsec proof point 1 key lem_eq4}%%%%%%%%%ref%%%%%%%%%%
&\chi (\tau,0,y)-\tilde{\chi} (\tau,0,y)=\int_0^{\tau} V(t,\chi (t,0,y),f (t,y))-V(t,\tilde{\chi} (t,0,y),g (t,y))dt\\
&=\int_0^{\tau} \int_0^1 D_uV(t,\chi (t,0,y),r f (t,y)+(1-r)g (t,y))[f (t,y)-g (t,y)]dtdr\nonumber\\
&+\int_0^{\tau} \int_0^1 D_x V(t,r\chi (t,0,y)+(1-r)\tilde{\chi} (t,0,y),g (t,y))[\chi (t,0,y)-\tilde{\chi} (t,0,y)]dtdr. \nonumber
\end{align}
Taking $0<\delta<s-\alpha-\frac{d}{2}$,by estimates of Theorem \ref{geometric proof quasi WW_Appendix energy est of pulled back eq_theorem on para eq}:
\begin{multline*}
f (t,y)=f (0,y)+\int_0^t\partial_tf (r,y)dr\\=u^1 (y)+O_{L^\infty}(t(\norm{u^1 }_{H^{\frac{d}{2}+\alpha+\delta} }))=u^1 (y)+O_{L^\infty}(t\eps).
\end{multline*}
Analogously we get:
\[g (t,y)=v^1 (y)+O_{L^\infty}(t\eps).\]

Now $(u^1-v^1)(y)= (l-l')(0,y)$ is the evaluation of the solution of the following Cauchy problem at $t=0$:
\begin{equation} \label{geometric proof quasi WW_a tech gen_subsec proof point 1 key lem_eq5} %%%%%%%%%ref%%%%%%%%%%
\begin{cases}
 \partial_t (l-l') +T_{a(t,l)} (l-l') =F(t,l)-F(t,l')\\
 \forall y \in \d^d, (l-l')(\tau,y)=-\eps\omega(y).
\end{cases}
\end{equation}
thus by estimates of Theorem \ref{geometric proof quasi WW_Appendix energy est of pulled back eq_theorem on para eq}:
\begin{align*}
u^1 (y)-v^1 (y)&=(l-l')(0,y)=-\eps \omega(y)+ \int_{\tau}^0\partial_t (l-l')(t,y)dt\\
&=-\eps \omega(y)+O_{L^\infty}( \tau(\norm{v^1 }_{H^{\alpha+\frac{d}{2}+\delta}}+\norm{u^1 }_{H^{\alpha+\frac{d}{2}+\delta}}))\\
&=-\eps \omega(y)+O_{L^\infty}(\tau\eps).
\end{align*}
Thus,
\begin{align*}
&\chi (\tau,0,y)-\tilde{\chi} (\tau,0,y)\\
&=-\eps [\int_0^{\tau} \int_0^1 D_uV(t,\chi (t,0,y),r f (t,y)+(1-r)g (t,y))dtdr] [\omega(y)]+O_{L^{\infty}}(\tau^2\eps)\\
&+\int_0^{\tau} \int_0^1 D_x V(t,r\chi (t,0,y)+(1-r)\tilde{\chi} (t,0,y),g (t,y))dr[\underbrace{\chi (t,0,y)-\tilde{\chi} (t,0,y)}_{*}]dt.
\end{align*}
Iterating the computation in (*):
\begin{align*}
&\chi (\tau,0,y)-\tilde{\chi} (\tau,0,y)\\
&=-\eps [\int_0^{\tau} \int_0^1 D_uV(t,\chi (t,0,y),r f (t,y)+(1-r)g (t,y))dtdr] [\omega(y)]
+O_{L^{\infty}}(\tau^2\eps)\\
&=-\eps [\int_0^{\tau} D_uV(t,y,0)dt] [\omega(y)]
+O_{L^{\infty}}(\tau^2\eps+\eps^2\tau).
\end{align*}
and finally we get in \eqref{geometric proof quasi WW_a tech gen_subsec proof point 1 key lem_eq2},
\begin{align*}			
	&\tilde{\chi} (\tau,0,x)-y\\			
				&=-\eps [\int_0^{\tau} D_uV(t,y,0)dt] [\omega(y)]
				+O_{L^{\infty}}(\tau^2\eps+\eps^2\tau).	
\end{align*}
We get for $x \in \ \supp\ u ^0\circ \chi (0,\tau,\cdot) $:
\begin{itemize}
\item For $\d^d=\r^d$:
\begin{align*}
&\lambda \abs{\tilde{\chi} (0,\tau,x)}\\
& \geq \eps  \lambda  \abs{\int_0^\tau \int_0^1 D_uV(t,y,0)drdt[\omega(y)]}-1+o_{L^\infty}(\tau \lambda \eps) 
\\&
\geq 2 ,
\end{align*}
 which gives the desired result.
\item For $\d^d=\t^d$ given an adequate choice of  $\tau, \eps $ and $\lambda $ we get:
\[2 n\pi +1\leq \lambda \abs{\tilde{\chi} (0,\tau,x)} \leq 2(n+1)\pi -1,\]
Which again gives the desired result.
\end{itemize}

\subsubsection{Proof of point 2 of Lemma \ref{geometric proof quasi WW_a tech gen_subsec key lem}}\label{geometric proof quasi WW_a tech gen_subsec proof point 2 key lem}
We start by writing:
\begin{align*}
u (t,x)-v (t,x)&=\chi (0,t,x)^*f (t, x) -\tilde{\chi} (0,t,x)^*g (t, x)+R(f)-R(g)
\intertext{where $R$ is a regularizing operator of order 2, }
u (t,x)-v (t,x)&=\underbrace{\chi (0,t,x)^*f (t, x)-\tilde{\chi} (0,t,x)^*f (t, x)}_{(1)}\\
&+\tilde{\chi} (0,t,x)^*(f -g )(t, x)+R(f)-R(g).
\end{align*}
Term $(1)$ resembles the main term in the usual transport estimates we used in point 1 of the Lemma
 \footnote{Like the ones used in proving the quasi-linearity of the Burgers equation.} 
 but with a main difference is $f $ being some dispersed data and not compactly supported and the use of the paracomposition operator. 
Again, the main trick here was to construct from $u^0 , v ^0$ the defocused data in the past $u^1 ,v^1 $ and use this as the initial data for $f $ and $g $. 
\begin{align*}
u (\tau,x)-v (\tau,x)
			&=u ^0\circ \chi (0,\tau,x)-u ^0\circ \tilde{\chi} (0,\tau,x)\\
			&+T_{(u ^0)'\circ \chi (0,\tau,x)}\chi (0,\tau,x)-T_{(u ^0)'\circ \tilde{\chi} (0,\tau,x)}\tilde{\chi} (0,\tau,x)\\
			&+\chi (0,\tau,x)^*(f -u^0 )(\tau, x)-\tilde{\chi} (0,t,x)^*(f -u^0 )(\tau, x)\\
			&+\tilde{\chi} (0,\tau,x)^*(f -g )(\tau, x)+R(f)-R(g).\\
			&=u ^0\circ \chi (0,\tau,x)-u ^0\circ \tilde{\chi} (0,\tau,x)\\
			&+\underbrace{\chi (0,\tau,x)^*(f -l )(\tau, x)-\tilde{\chi} (0,t,x)^*(f -l )(\tau, x)}_{(1)}\\
			&+\underbrace{\tilde{\chi} (0,\tau,x)^*(f -g )(\tau, x)}_{(2)}+R(f)-R(g)\\
			&+\underbrace{T_{(u ^0)'\circ \chi (0,\tau,x)}\chi (0,\tau,x)-T_{(u ^0)'\circ \tilde{\chi} (0,\tau,x)}\tilde{\chi} (0,\tau,x)}_{(3)}.
\end{align*}
Where $l$ is defined by \eqref{geometric proof quasi WW_a tech gen_subsec def ansatz_eq2}and $R$ was modified to contain other regularizing operators of order 2 that appear by symbolic calculus rules.
The easiest part to estimate is the remainder one because of the gain of derivatives and Theorem \eqref{Pre eq8}:
\[\norm{R(f)-R(g)}_{H^s}\leq C \eps. \]
We turn to estimating $(1)$, by Theorem \ref{paracomposition_section Paracomposition_subsec paracomp on the euclidean space_theorem paralinearisation of composition}:
\[ \norm{\chi (0,\tau,x)^*(f -l )(\tau, x)}_{H^s } \leq C  \norm{(f -l )(\tau, \cdot)}_{H^s}.\]				
Now $f -l $ solve: 
\begin{equation}\label{geometric proof quasi WW_a tech gen_subsec proof point 2 key lem_eq1}%%%%%%%%%ref%%%%%%%%%%
\begin{cases}
 \partial_t (f -l )+T_{a(t,l)} (f -l )=(T_{a(t,l)}-T_{a(t,u)^*} )f-F(t,l)+\chi (t,0,x)^*F(t,u)F(t,f) \\
 \forall x \in \d^d, (f -l )(0,x)=0.
\end{cases}
\end{equation}

Writing 
\[\chi (t,0,x)^*F(t,u)-F(t,l)=G_1(f-l)\]
where $G_1$ is a continuous linear operator on $H^s$ we get the estimates by Theorem \ref{geometric proof quasi WW_Appendix energy est of pulled back eq_theorem on para eq}:
\begin{align*} \label{geometric proof quasi WW_a tech gen_subsec proof point 2 key lem_eq2}%%%%%%%%%ref%%%%%%%%%% \lesssim
&\norm{f -l (\tau,\cdot)}_{H^{\nu} }\\
&\leq C[ \norm{(T_{a(\tau,l)}-T_{a(t,l)^*} )f}_{L^1([0,\tau],H^{\nu}) }+\norm{(T_{a(\tau,l)^*}-T_{a(\tau,u)^*} )f}_{L^1([0,\tau],H^{\nu}) }]\\
&\leq C[ \tau^2\norm{Id-D\chi^{-1}}_{L^\infty}\norm{f}_{H^{\nu+\alpha}}+\tau\norm{u-l}_{L^\infty}\norm{f}_{H^{\nu+\alpha}}].
\end{align*}
As $s>2+\frac{d}{2}$,
\[
\norm{f -l (\tau,\cdot)}_{H^{\nu} } \leq C[\tau^2\lambda^{\alpha-1}+\tau^2\lambda^{\alpha-1}]\lambda^{\nu-s},
\]		 
which gives
\begin{equation*}
 \norm{\chi (0,\tau,x)^*(f -l )(\tau, x)}_{H^\nu }
  \leq C\tau\lambda^{\alpha-1}\lambda^{\nu-s},
 \end{equation*}
and
 \begin{equation*}
 \norm{\tilde{\chi} (0,\tau,x)^*(f -l )(\tau, x)}_{H^\nu }\leq C\tau\lambda^{\alpha-1}\lambda^{\nu-s}.
 \end{equation*}
Thus we finally get
  \begin{equation}\label{geometric proof quasi WW_a tech gen_subsec proof point 2 key lem_eq3} %%%%%%%%%ref%%%%%%%%%%
 \norm{(1)}_{H^\nu }\leq  C\tau\lambda^{\alpha-1}\lambda^{\nu-s}.
 \end{equation}
Now we estimate $(2)$ and $(3)$ in the same manner, by Theorem \ref{paracomposition_section Paracomposition_subsec paracomp on the euclidean space_theorem paralinearisation of composition}:
\begin{align*}
 \norm{\tilde{\chi} (0,\tau,x)^*(f -g )(\tau, x)}_{H^\nu } \leq C  \norm{(f -g )(\tau, \cdot)}_{H^\nu },
 \end{align*}
 And as $s>2+\frac{d}{2}$ and by \eqref{geometric proof quasi WW_a tech gen_subsec change of var by trans_eq6}:
 \begin{align*}
 \norm{T_{(u ^0)'\circ \chi (0,\tau,x)}\chi (0,\tau,x)-T_{(u ^0)'\circ \tilde{\chi} (0,\tau,x)}\tilde{\chi} (0,\tau,x) } \leq C  \norm{(f -g )(\tau, \cdot)}_{H^\nu }.
 \end{align*}
Now $f -g $ solve: 
\begin{align}\label{geometric proof quasi WW_a tech gen_subsec proof point 2 key lem_eq4}%%%%%%%%%ref%%%%%%%%%%
 \begin{cases}
 &\partial_t (f -g )+ T_{a(t,u)^*}(f -g )-(T_{a(t,v)^*}-T_{\widetilde{a(t,v)}^*} )g\\
 &-\chi (t,0,x)^*F(t,u)+\tilde{\chi} (t,0,x)^*F(t,v)=(T_{a(t,u)^*}-T_{a(t,v)^*} )g \\
& \forall x \in \d^d, (f -g )(0,x)=(u^1 -v^1 )(x).
\end{cases}
\end{align}
Here will need to be more careful as the nonlinearity in the dispersive term can be more "harmful" than the transport term when $\alpha \geq 1$, which was not there in the treatment of the model problem.
 More precisely we write:
\[\chi (t,0,x)^*F(t,u)-F(t,l)=G_2(f-g),\]
where $G_2$ is a continuous linear operator on $H^s$ and we get by Theorem \ref{geometric proof quasi WW_a tech gen_subsec prereq CP_th1}:
\begin{align*}
&\norm{f -g (\tau,\cdot)}_{H^{\nu} }\\
&\leq C\bigg[ \norm{(T_{a(t,v)^*}-T_{\widetilde{a(t,v)}^*} )g}_{L^1([0,\tau],H^{\nu}) }+\norm{(T_{a(t,u)^*}-T_{a(t,v)^*} )g}_{L^1([0,\tau],H^{\nu}) }+\eps\bigg]\\
&\leq C[ \tau \norm{D\chi^{-1}-D\tilde{\chi}^{-1}}_{L^\infty}\norm{g}_{H^{\nu+\alpha}}+\tau \norm{f-g}_{L^\infty}\norm{(f,g)}_{L^\infty}^k\norm{g}_{H^{\nu+\beta}}+\eps]\\
&\leq C[ \tau^2\eps \lambda^\alpha \lambda^{\nu-s}+\tau \eps^k\lambda^{\beta}\lambda^{\nu-s}+\eps],
\end{align*}
 which gives
  \begin{equation}\label{geometric proof quasi WW_a tech gen_subsec proof point 2 key lem_eq5}%%%%%%%%%ref%%%%%%%%%%
 \norm{(2)}_{H^\nu }\leq  C(\tau^2\eps \lambda^\alpha \lambda^{\nu-s}+\tau \eps^k\lambda^{\beta}\lambda^{\nu-s}+\eps),
 \end{equation}
 and
 \begin{equation}
 \norm{(3)}_{H^\nu }\leq  C(\tau^2\eps \lambda^\alpha \lambda^{\nu-s}+\tau \eps^k\lambda^{\beta}\lambda^{\nu-s}+\eps),
 \end{equation}
  finishing the proof of Lemma \ref{geometric proof quasi WW_a tech gen_subsec key lem}and Theorem  \ref{geometric proof quasi WW_a tech gen_main theorem}.
\begin{remark}\label{geometric proof quasi WW_a tech gen_subsec proof point 2 key lem_rem water waves remark cancellation}
For the application to the water waves system we need to remark that the restriction on $\beta$ comes from the $T_{a(t,u)^*}-T_{a(t,v)^*}g$. Naively estimating this term we see that in the case of water waves with surface tension we are working in the limit case $\beta=\alpha=\frac{3}{2}$ and $k=2$ which is barely missed by Theorem \ref{geometric proof quasi WW_a tech gen_main theorem}. We will will show that this can be avoided in the special case of water waves system by carefully choosing the ansatz. For this we notice that one slightly modify the last two estimates and we write:
  \begin{equation}\label{geometric proof quasi WW_a tech gen_subsec proof point 2 key lem_rem water waves remark cancellation_modif WW eq 1}%%%%%%%%%ref%%%%%%%%%%
 \norm{(2)}_{H^\nu }\leq  C(\tau^2\eps \lambda^\alpha \lambda^{\nu-s}+\tau^2 \eps^k\lambda^{\beta}\lambda^{\nu-s}+\eps+\norm{(T_{a(t,u_0)^*}-T_{a(t,v_0)^*} )g}_{L^1([0,\tau],H^{\nu}) }),
 \end{equation}
 and
 \begin{equation}
 \norm{(3)}_{H^\nu }\leq  C(\tau^2\eps \lambda^\alpha \lambda^{\nu-s}+\tau^2 \eps^k\lambda^{\beta}\lambda^{\nu-s}+\eps+\norm{(T_{a(t,u_0)^*}-T_{a(t,v_0)^*} )g}_{L^1([0,\tau],H^{\nu}) }).
 \end{equation}
\end{remark}
\section{Quasi-linearity of the Water-Waves system with surface tension}
In this section we always have $\kappa=1$.
\subsection{Prerequisites from the Cauchy problem}
We start by recalling the apriori estimates given by Proposition $5.2$ of \cite{Alazard11}. We keep the notations of Theorem \ref{geometric proof quasi WW_introduction_Quasi-linearity of the water Wave system_theorem with surface tension}.
\begin{proposition} \label{geometric proof quasi WW_quasi WW ST CP} (From \cite{Alazard11}) Let $d\geq 1$ be the dimension and consider a real number $s> 2+\frac{d}{2}$. Then there exists a non decreasing function C such that, for all $T \in ]0,1]$ and all solution $(\eta,\psi)$ of \eqref{geometric proof quasi WW_introduction_ subsection the equations_WW eq} such that
\[(\eta,\psi) \in C^0([0,T];H^{s+\frac{1}{2}}(\r^d)\times H^{s}(\r^d) ) \text{ and $H_t$ is verified for $t\in [0,T]$,}\]  
we have \[\norm{(\eta,\psi)}_{L^\infty(0,T;H^{s+\frac{1}{2}}\times H^{s})} \leq C({(\eta_0,\psi_0)}_{H^{s+\frac{1}{2}}\times H^{s}})+TC(\norm{(\eta,\psi)}_{L^\infty(0,T;H^{s+\frac{1}{2}}\times H^{s})}). \]
\end{proposition}
The proof will rely on the para-linearised and symmetrized version of \eqref{geometric proof quasi WW_introduction_ subsection the equations_WW eq} given by Proposition $4.8$ and corollary $4.9$ of \cite{Alazard11}. Before we recall this, for clarity as in \cite{Alazard11} we introduce a special class of operators $\Sigma^m \subset \Gamma^m_0$ given by:
\begin{definition}(From \cite{Alazard11})
Given $m \in \r$, $\Sigma^m$ denotes the class of symbols $a$ of the form
\[a=a^{(m)}+a^{(m-1)}\]
with 
\[a^{(m)}=F(\nabla \eta (t,x),\xi)\]
\[a^{(m-1)}=\sum_{\abs{k}=2}G_\alpha(\nabla \eta (t,x),\xi)\partial^k_x \eta(t,x),\]
such that 
\begin{enumerate}
\item $T_a$ maps real valued functions to real-valued functions;
\item F is of class $C^\infty$ real valued function of $(\zeta,\xi) \in \r^d \times (\r^d \setminus 0),$ homogeneous of order m in $\xi$; and such that there exists a continuous function $K=K(\zeta)>0$ such that
\[F(\zeta,\xi)\geq K(\zeta)\abs{\xi}^m,\]   
for all $(\zeta,\xi)\in \r^d\times (\r^d \setminus 0)$; 
\item $G_\alpha$ is a $C^\infty$ complex valued function of $(\zeta,\xi)\in \r^d\times (\r^d \setminus 0)$, homogeneous of order $m-1$ in $\xi$.
\end{enumerate}
\end{definition}
$\Sigma^m$ enjoys all the usual symbolic calculus properties modulo acceptable reminders that we define by the following:
\begin{definition-notation}(From \cite{Alazard11})
Let $m \in \r$ and consider two families of operators of order m,
\[\set{A(t): t \in [0,T]}, \ \ \ \set{B(t):t \in [0,T]}.\]
We shall say that $A \sim B$ if $A-B$ is of order $m -\frac{3}{2}$ and satisfies the following estimate: for all $\mu \in \r$, there exists a continuous function C such that for all $t \in [0,T]$,
\[\norm{A(t)-B(t)}_{H^{\mu}\rightarrow H^{\mu -m+\frac{3}{2}}} \leq C(\norm{\eta(t)}_{H^{s+\frac{1}{2}}}).\]
\end{definition-notation}
In the next Proposition we recall the different symbols that appear in the para-linearisation and symmetrisation of the equations.
\begin{proposition}\label{geometric proof quasi WW_quasi WW ST para symbols}(From \cite{Alazard11})%%%%%%%%%%%%ref%%%%%%%%%
We work under the hypothesis of Proposition \ref{geometric proof quasi WW_quasi WW ST CP}.
Put 
\[\lambda=\lambda^{(1)}+\lambda^{(0)}, \ l=l^{(2)}+l^{(1)} \text{ with,}\]
\begin{align}\label{geometric proof quasi WW_quasi WW ST para symbols_eq1}%%%%%%%%%%%%ref%%%%%%%%%
&\begin{cases}
\lambda^{(1)}=\sqrt{(1+|\nabla \eta|^2)\abs{\xi}^2-(\nabla \eta \cdot \xi)^2},\\
\lambda^{(0)}=\frac{1+|\nabla \eta|^2}{2\lambda^{(1)}}\set{div \bigg(\alpha^{(1)}\nabla \eta\bigg)+i\partial_\xi \lambda^{(1)}\cdot \nabla \alpha^{(1)} },\\
\alpha^{(1)}=\frac{1}{\sqrt{1+|\nabla \eta|^2}}\bigg( \lambda^{(1)}+i\nabla \eta \cdot \xi \bigg).
\end{cases}
\\
&\begin{cases}
l^{(2)}=(1+|\nabla \eta|^2)^{-\frac{1}{2}}\bigg(\abs{\xi}^2-\frac{(\nabla \eta \cdot \xi)^2}{1+|\nabla \eta|^2} \bigg),\\
l^{(1)}=-\frac{i}{2}(\partial_x \cdot \partial_\xi)l^{(2)}.\\
\end{cases}
\end{align}
Now let $q\in \Sigma^0, p \in \Sigma^{\frac{1}{2}}, \gamma \in \Sigma^{\frac{3}{2}}$ be defined by
\begin{align*}
q&=(1+|\nabla \eta|^2)^{-\frac{1}{2}},\\
p&=(1+|\nabla \eta|^2)^{-\frac{5}{4}}\sqrt{\lambda^{(1)}}+p^{(-\frac{1}{2})},\\
\gamma&=\sqrt{l^{(2)}\lambda^{(1)}}+\sqrt{\frac{l^{(2)}}{\lambda^{(1)}}}\frac{Re \lambda^{(0)}}{2}-\frac{i}{2}(\partial_\xi \cdot \partial_x)\sqrt{l^{(2)}\lambda^{(1)}},\\
p^{(-\frac{1}{2})}&=\frac{1}{\gamma^{(\frac{3}{2})}}\set{ql^{(1)}-\gamma^{(\frac{1}{2})}p^{(\frac{1}{2})}+i\partial_\xi \gamma^{(\frac{3}{2})}\cdot \partial_x p^{(\frac{1}{2})}}.
\end{align*}
Then
\[T_qT_\lambda \sim T_\gamma T_q, \  T_q T_l \sim T_\gamma T_p, \ T_\gamma \sim (T_\gamma)^\top. \]
\end{proposition}
Now we can write the para-linearization and symmetrization of the equations \eqref{geometric proof quasi WW_introduction_ subsection the equations_WW eq} after a change of variable:
\begin{corollary}\label{geometric proof quasi WW_quasi WW ST_cor paralin sym system}(From \cite{Alazard11})%%%%%%%%%%%%ref%%%%%%%%%
Under the hypothesis of Proposition \ref{geometric proof quasi WW_quasi WW ST CP}, introduce the unknowns
\[U=\psi-T_B \eta \footnote{U is commonly called the "good" unknown of Alinhac.}, \ \Phi_1= T_p \eta  \text{ and }  \Phi_2=T_q U,\]
where we recall,
\[
\begin{cases}
B=(\partial_y \phi)_{|y=\eta}=\frac{\nabla \eta \cdot \nabla \psi +G(\eta) \psi}{1+|\nabla \eta|^2},\\
V=(\nabla_x \phi)_{|y=\eta}=\nabla \psi -B\nabla \eta.
\end{cases}\]
Then $\Phi_1, \Phi_2 \in C^0([0,T];H^s(\r^d))$ and 
\begin{equation}
\begin{cases}
\partial_t \Phi_1 +T_V \cdot \nabla \Phi_1-T_\gamma \Phi_2=f_1,\\
\partial_t \Phi_2+ T_V \cdot \nabla \Phi_2+ T_\gamma \Phi_1=f_2,
\end{cases}
\end{equation}
with $f_1,f_2 \in L^\infty(0,T;H^s(\r^d)),$ and $f_1,f_2$ have $C^1$ dependence on $(U,\theta)$ verifying:
\[
\norm{(f_1,f_2)}_{L^\infty(0,T;H^s(\r^d))} \leq C(\norm{(\eta,\psi)}_{L^\infty(0,T;H^{s+\frac{1}{2}}\times H^s(\r^d))}).
\]
\end{corollary}
\subsection{Proof of Theorem \ref{geometric proof quasi WW_introduction_Quasi-linearity of the water Wave system_theorem with surface tension}}
Corollary \ref{geometric proof quasi WW_quasi WW ST_cor paralin sym system} shows that the para-linearization and symmetrization of the equations \eqref{geometric proof quasi WW_introduction_ subsection the equations_WW eq} are of the form of the equations treated in Theorem \ref{geometric proof quasi WW_a tech gen_main theorem}. The goal of the proof is thus to mainly show that the previous change of unknowns preserves the quasi-linear structure of the equations. This we will be proved but with a slightly different change of unknowns that will satisfy the same type of equations.\\
\subsubsection{Reducing the problem around 0}
{\ } \\
Fix $T>0$, $r>0$ as in the proof of Theorem \ref{geometric proof quasi WW_a tech gen_main theorem}, given the local nature of the result we see that we can work on balls with radius $r$ small. Henceforth we will be working on $ \b (0,r) \subset C^0([0,T];H^{s+\frac{1}{2}}(\r^d)\times H^{s}(\r^d) )$ and without loss of generality we suppose that $H_t$ is always verified on $[0,T]$ on that set.
\subsubsection{New change of unknowns}
\begin{lemma}\label{geometric proof quasi WW_quasi WW ST_subsec New chnage of unknowns_lem1}%%%%%%%%%%%%ref%%%%%%%%%
Under the hypothesis of Proposition \ref{geometric proof quasi WW_quasi WW ST CP}, fix $\eps>0$ and introduce the unknowns
\[U=\psi-T_B \eta, \  \tilde{\Phi}_1=[ T_p+\eps (I-T_1) ]\eta  \text{ and }  \tilde{\Phi}_2=[T_q +\eps (I-T_1)] U.\]
Then $ \tilde{\Phi}_1,  \tilde{\Phi}_2\in C^0([0,T];H^s(\r^d))$ and 
\begin{equation} \label{geometric proof quasi WW_quasi WW ST_subsec New chnage of unknowns_lem1_eq1}
\begin{cases}
\partial_t  \tilde{\Phi}_1 +T_V \cdot \nabla  \tilde{\Phi}_1-T_\gamma  \tilde{\Phi}_2=\tilde{f}_1,\\
\partial_t  \tilde{\Phi}_2+ T_V \cdot \nabla  \tilde{\Phi}_2+ T_\gamma  \tilde{\Phi}_1=\tilde{f}_2,
\end{cases}
\end{equation}
with $\tilde{f}_1,\tilde{f}_2 \in L^\infty(0,T;H^s(\r^d)),$ and $\tilde{f}_1,\tilde{f}_2 $ have $C^1$ dependence on $(U,\theta)$ verifying:
\[
\norm{(\tilde{f}_1,\tilde{f}_2)}_{L^\infty(0,T;H^s(\r^d))} \leq C(\norm{(\eta,\psi)}_{L^\infty(0,T;H^{s+\frac{1}{2}}\times H^s(\r^d))}).
\]
\end{lemma}
\begin{proof}
 The Lemma simply follows from the fact that $I-T_1$ is a regularizing operator.
\end{proof}
\subsubsection{The new change of unknowns locally preserves the structure of the equations:} 
To apply Theorem \ref{geometric proof quasi WW_a tech gen_main theorem} we simply note that $DV(0,0)(h,k)=\nabla h$.
Thus proof of Theorem \ref{geometric proof quasi WW_introduction_Quasi-linearity of the water Wave system_theorem with surface tension} in the threshold $s>2+\frac{d}{2}$ will then follow  from Theorem \ref{geometric proof quasi WW_a tech gen_main theorem} combined with Lemma \ref{geometric proof quasi WW_quasi WW ST_subsec New chnage of unknowns_lem1} and the following Lemma.
\begin{lemma}\label{geometric proof quasi WW_quasi WW ST_subsec New change of unknowns local struc_lem1}%%%%%%%%%%%%ref%%%%%%%%%
Let $d\geq 1$ and $s> 2+\frac{d}{2}$. There exists $r,\eps>0$ such that:
\begin{align*}
\tilde{\Phi}: \b (0,r) &\rightarrow C^0([0,T];H^s(\r^d))\\
(\eta, \psi) &\mapsto (\tilde{\Phi}_1,  \tilde{\Phi}_2)
\end{align*}
is a $C^\infty$ diffeomorphism upon it's image and $\tilde{\Phi}(0)$=0.
\end{lemma}
\begin{proof}
\[
\tilde{\Phi}(\eta,\psi)=\underbrace{\begin{pmatrix}
  T_p +\eps(I-T_1) & 0 \\
  0 & T_q +\eps(I-T_1) 
 \end{pmatrix}}_{(1)}
 \underbrace{ \begin{pmatrix}
  I  & 0 \\
 -T_B& I
 \end{pmatrix}}_{(2)}
  \begin{pmatrix}
  \eta \\
 \psi
 \end{pmatrix} 
 \]
$(2)$ being clearly a diffeomorphism we will concentrate on $(1)$.\\
First we see that for $r$ small enough $T_q +\eps(I-T_1) $ is a perturbation of the $T_1 +\eps(I-T_1)$, indeed by symbolic calculus rules:
\begin{align*}
\norm{T_q +\eps(I-T_1)-T_1 -\eps(I-T_1)}_{\mathcal{L}(H^s)}&=\norm{T_q-T_1}_{\mathcal{L}(H^s)}\\
								&\leq M_0^0(q-1)\\
								&\leq C(\norm{\eta}_{W^{1,\infty}})\norm{\eta}_{W^{1,\infty}}\\
								&\leq C(\norm{\eta}_{H^s})\norm{\eta}_{H^s}
\end{align*}
which gives the desired result.\\
 Now we turn to $T_p +\eps(I-T_I) $. First notice that for $\eps>0$:
 \[T_{\abs{\xi}^{\frac{1}{2}}} +\eps(I-T_I):C^0([0,T];H^{s+\frac{1}{2}}(\r^d) ) \rightarrow C^0([0,T];H^s(\r^d))\] is a $C^\infty$ diffeomorphism. And now we see that $T_p +\eps(I-T_I) $ is a perturbation of $T_{\abs{\xi}^{\frac{1}{2}}} +\eps(I-T_I) $ indeed by symbolic calculus rules:
 \begin{align*}
\norm{T_p -T_{\abs{\xi}^{\frac{1}{2}}} }_{\mathcal{L}(H^{s+\frac{1}{2}},H^s)}	&\leq C(\norm{\eta}_{W^{1,\infty}})\norm{\eta}_{W^{1,\infty}}\\
								&\leq C(\norm{\eta}_{H^s})\norm{\eta}_{H^s}
\end{align*}
\end{proof}
Now to conclude the proof of Theorem \ref{geometric proof quasi WW_introduction_Quasi-linearity of the water Wave system_theorem with surface tension}, we want to apply Corollary \ref{geometric proof quasi WW_a tech gen_cor theorem sys} but as remarked previously we  find ourselves in the limit case $\beta=\alpha=\frac{3}{2}$ and $k=2$ which is not apriori covered by the Corollary. The key observation is that we have:
\[
V(\eta,\psi)=\nabla \psi -B\nabla \eta,\ \gamma=\gamma(\eta),
\]
and that $V(0,\psi)$ and $\abs{\xi}^{\frac{3}{2}}=\gamma(0)$ do verify the hypothesis of Corollary \ref{geometric proof quasi WW_a tech gen_cor theorem sys} which is sufficient in order to apply the Corollary by Remark \ref{geometric proof quasi WW_a tech gen_subsec proof point 2 key lem_rem water waves remark cancellation}. Thus by Lemma \ref{geometric proof quasi WW_quasi WW ST_subsec New change of unknowns local struc_lem1}, the equations \eqref{geometric proof quasi WW_quasi WW ST_subsec New chnage of unknowns_lem1_eq1} verify the hypothesis of Corollary \ref{geometric proof quasi WW_a tech gen_cor theorem sys} in the threshold $s>2+\frac{d}{2}$ with the choice $\eta_0=0$ thus we have two sequences:
\[
\begin{cases}
\exists ( \tilde{\Phi}_1^0, \tilde{\Phi}_2^0)\in C^0([0,T];H^s(\r^d)) \ \ \text{solution of \eqref{geometric proof quasi WW_quasi WW ST_subsec New chnage of unknowns_lem1_eq1}},\\
\exists ( \tilde{\Phi}_1^1, \tilde{\Phi}_2^1)\in C^0([0,T];H^s(\r^d)) \ \ \text{solution of \eqref{geometric proof quasi WW_quasi WW ST_subsec New chnage of unknowns_lem1_eq1}},
\end{cases}
\]
such that $\exists c>0$,
\[
\begin{cases}
\norm{( \tilde{\Phi}_1^0, \tilde{\Phi}_2^0)(0,\cdot)-( \tilde{\Phi}_1^1, \tilde{\Phi}_2^1)(0,\cdot)}_{H^s}=\norm{( 0, \tilde{\Phi}_2^0)(0,\cdot)-( 0, \tilde{\Phi}_2^1)(0,\cdot)}_{H^s}\rightarrow 0,\\
\norm{( \tilde{\Phi}_1^0, \tilde{\Phi}_2^0)-( \tilde{\Phi}_1^1, \tilde{\Phi}_2^1)}_{L^\infty([0,T], H^s)}>c.
\end{cases}
\]

Now putting $(\eta^0,\psi^0)=\tilde{\Phi}^{-1}( \tilde{\Phi}_1^0, \tilde{\Phi}_2^0)$ and $(\eta^1,\psi^1)=\tilde{\Phi}^{-1}( \tilde{\Phi}_1^1, \tilde{\Phi}_2^1)$ we get from Lemmas \ref{geometric proof quasi WW_quasi WW ST_subsec New chnage of unknowns_lem1} and \ref{geometric proof quasi WW_quasi WW ST_subsec New change of unknowns local struc_lem1}:
\[
\begin{cases}
(\eta^0,\psi^0)\in C^0([0,T];H^{s+\frac{1}{2}}(\r^d)\times H^{s}(\r^d) ) \ \ \text{is a solution of \eqref{geometric proof quasi WW_introduction_ subsection the equations_WW eq}},\\
(\eta^1,\psi^1)\in C^0([0,T];H^{s+\frac{1}{2}}(\r^d)\times H^{s}(\r^d) ) \ \ \text{is a solution of \eqref{geometric proof quasi WW_introduction_ subsection the equations_WW eq}},
\end{cases}
\]

such that 
\[
\begin{cases}
\norm{(\eta^0,\psi^0)(0,\cdot)-(\eta^1,\psi^1)(0,\cdot)}_{H^{s+\frac{1}{2}}\times H^s}=\norm{(0,\psi^0)(0,\cdot)-(0,\psi^1)(0,\cdot)}_{H^{s+\frac{1}{2}}\times H^s}\rightarrow 0,\\
\norm{(\eta^0,\psi^0)-(\eta^1,\psi^1)}_{L^\infty ([0,T],H^{s+\frac{1}{2}}\times H^s)}>c.
\end{cases}
\]
thus giving us the desired result. As the change of unknowns is a diffeomorphism (thus is Lipschitz) we get analogously the result on the control in weaker norms.
\section{Quasi-Linearity of the Gravity Water Waves}%%%eq6,th,lem,rem,cor
In this section we always have $\kappa=0$. The proof will follow as in the previous section but with some extra care, taking into account the lower regularity framework.
\subsection{Prerequisites from the Cauchy problem}
We start by recalling the apriori estimates given by Proposition $4.1$ of \cite{Alazard14}, we keep the notations of Theorem \ref{geometric proof quasi WW_introduction_Quasi-linearity of the water Wave system_theorem without surface tension}.
\begin{proposition} \label{geometric proof quasi WW_quasi WW without ST_subsec prereq CP}(From \cite{Alazard14}) Let $d\geq 1$ be the dimension and consider a real number $s> 1+\frac{d}{2}$. Then there exists a non decreasing function C such that, for all $T \in ]0,1]$ and all solution $(\eta,\psi)$ of \eqref{geometric proof quasi WW_introduction_ subsection the equations_WW eq} such that:
\[
\begin{cases}
(\eta,\psi) \in C^0([0,T];H^{s+\frac{1}{2}}(\r^d) \times H^{s+\frac{1}{2}}(\r^d)),\\  
\text{$H_t$ is verified for $t\in [0,T]$,}\\
 \exists c_0>0, \forall t \in [0,T], a(t,x)\geq c_0,
 \end{cases}
 \]
 we have \footnote{Recall B and V are defined by \eqref{geometric proof quasi WW_introduction_Quasi-linearity of the water Wave system_equations on the traces of the velocity}.} 
 \begin{align*}
 &\norm{(\eta,\psi,V,B)}_{L^\infty([0,T];H^{s+\frac{1}{2}}\times H^{s+\frac{1}{2}}\times H^{s}\times H^{s})}\\
 &\leq C(\norm{(\eta_0,\psi_0,V_0,B_0)}_{H^{s+\frac{1}{2}}\times H^{s+\frac{1}{2}}\times H^{s}\times H^{s}})\\
 &+TC(\norm{(\eta,\psi,V,B)}_{L^\infty(0,T;H^{s+\frac{1}{2}}\times H^{s+\frac{1}{2}}\times H^{s}\times H^{s})}).
\end{align*}
\end{proposition}
The proof will rely on the para-linearised and symmetrized version of \eqref{geometric proof quasi WW_introduction_ subsection the equations_WW eq} given by Proposition $4.8$ and $4.10$ of \cite{Alazard14}. Given the low regularity threshold, $\eta$ and thus $\Omega_t$ are in $W^{\frac{3}{2},\infty}(\r^d)$ for the gravity water waves by contrast to $W^{\frac{5}{2},\infty}(\r^d)$ frame work for the case with surface tension, the para-linearisation of \eqref{geometric proof quasi WW_introduction_ subsection the equations_WW eq} is done with the variables V and B. This will only add a technical level to our proof of quasi-linearity. 
\begin{proposition}\label{geometric proof quasi WW_quasi WW without ST_subsec prereq paralin sym eq}(From \cite{Alazard14})%%%%%%%%%%%%ref%%%%%%%%%
Under the hypothesis of Proposition \ref{geometric proof quasi WW_quasi WW without ST_subsec prereq CP}, suppose moreover that $ \norm{(V_0,B_0)}_{ H^{s}\times H^{s}}<+\infty$ thus by Proposition \eqref{geometric proof quasi WW_quasi WW without ST_subsec prereq CP}this regularity is propagated on $[0,T]$. Now introduce the unknowns
\[
\begin{cases}
\zeta=\nabla \eta, 
\\ U=V+T_\zeta B, 
\end{cases}
\text{ where, } 
\begin{cases}
B=(\partial_y \phi)_{|y=\eta}=\frac{\nabla \eta \cdot \nabla \psi +G(\eta) \psi}{1+|\nabla \eta|^2},\\
V=(\nabla_x \phi)_{|y=\eta}=\nabla \psi -B\nabla \eta.
\end{cases}\]
Now define the symbols:
\[
\begin{cases}
\lambda=\sqrt{(1+|\nabla \eta|^2)\abs{\xi}^2-(\nabla \eta \cdot \xi)^2},\\
\gamma=\sqrt{a \lambda},\\
q=\sqrt{\frac{a}{\lambda}}.
\end{cases}
\]
Set $\theta=T_q \zeta.$
Then $\theta, U \in C^0([0,T];H^s(\r^d) )$ and 
\begin{equation}
\begin{cases}
\partial_t U +T_V \cdot \nabla U+T_\gamma \theta=f_1,\\
\partial_t \theta+ T_V \cdot \nabla \theta- T_\gamma U=f_2,
\end{cases}
\end{equation}
with $f_1,f_2 \in L^\infty(0,T;H^s(\r^d) ),$ and $f_1,f_2$ have  $C^1$ dependence on $(U,\theta)$ verifying:
\[
\norm{(f_1,f_2)}_{L^\infty(0,T;H^s )} \leq C (\norm{(\eta,\psi,V,B)}_{L^\infty(0,T;H^{s+\frac{1}{2}}\times H^{s+\frac{1}{2}}\times H^{s}\times H^{s})})
\]
\end{proposition}
\subsection{Proof of Theorem \ref{geometric proof quasi WW_introduction_Quasi-linearity of the water Wave system_theorem without surface tension}}
As in the proof of Theorem \ref{geometric proof quasi WW_introduction_Quasi-linearity of the water Wave system_theorem with surface tension}, Proposition \eqref{geometric proof quasi WW_quasi WW without ST_subsec prereq paralin sym eq} shows that the para-linearisation and symmetrisation of the Equations \eqref{geometric proof quasi WW_introduction_ subsection the equations_WW eq} are of the form of the equations treated in Theorem \ref{geometric proof quasi WW_a tech gen_main theorem}. Thus again, the goal of the proof is thus to mainly show that the previous change of unknowns preserves the quasi-linear structure of the equations. This we will be proved but with a slightly different change of unknowns that will satisfy the same type of equations but where we take into account the low frequencies. For concision we will omit the $(\r^d) $ when writing the functional spaces.
\subsubsection{Reducing the problem around 0}
{\ } \\
Fix $T>0$, $r>0$ as in the proof of Theorem \ref{geometric proof quasi WW_a tech gen_main theorem} and \ref{geometric proof quasi WW_introduction_Quasi-linearity of the water Wave system_theorem with surface tension}, given the local nature of the result we see that first we can work on balls centered at $0$ with radius $r$ small. Put
\begin{align*}
I_{s,T}&=\set{(\eta, \psi) \in C^0([0,T];H^{s+\frac{1}{2}} \times H^{s+\frac{1}{2}}  ), (V, B) \in C^0([0,T];H^{s} \times H^{s}), \exists c>0, a\geq c  },\\
I_{s,0}&=\set{(\eta_0, \psi_0) \in H^{s+\frac{1}{2}} \times H^{s+\frac{1}{2}}  , (V_0, B_0) \in H^{s} \times H^{s}, \exists c>0, a\geq c },
\end{align*}
henceforth we will be working on $ \b (0,r) \subset I_{s,T}$ and without loss of generality we suppose that $H_t$ is always verified on $[0,T]$, on that set.
\subsubsection{New change of unknowns}
\begin{lemma}\label{geometric proof quasi WW_quasi WW without ST_subsec new change of unkniwns_lem1}%%%%%%%%%%%%ref%%%%%%%%%
Consider $\eps>0$ and $\omega \in C^\infty_0(\r^d)$ such that $\omega=1$ on $\b(0,1)$ and $\omega=0$ on $\r^d \setminus \b(0,2)$. Under the hypothesis of Proposition \eqref{geometric proof quasi WW_quasi WW without ST_subsec prereq paralin sym eq}, introduce the unknowns
\[
\begin{cases}
\tilde{\zeta}=(1-\omega(D))\nabla \eta, \\
 \tilde{U}=(1-\omega(D))(V+T_\zeta B), \\
 aux_1=\omega(D) \psi,\\
aux_2=\omega(D) \eta,
\end{cases}
 \text{ where, } 
\begin{cases}
B=(\partial_y \phi)_{|y=\eta}=\frac{\nabla \eta \cdot \nabla \psi +G(\eta) \psi}{1+|\nabla \eta|^2},\\
V=(\nabla_x \phi)_{|y=\eta}=\nabla\psi-B\nabla \eta.
\end{cases}\]
and set $\tilde{\theta}=T_q \tilde{\zeta}+\eps(I-T_1),$ where $q$ is defined in Proposition \eqref{geometric proof quasi WW_quasi WW without ST_subsec prereq paralin sym eq}.\\
Then $\tilde{\theta}, U,aux_1,aux_2 \in C^0([0,T];H^s )$ and 
\begin{equation}
\begin{cases}
\partial_t  \tilde{U} +T_V \cdot \nabla  \tilde{U}+T_\gamma \tilde{\theta}=f'_1,\\
\partial_t \tilde{\theta}+ T_V \cdot \nabla \tilde{\theta}- T_\gamma  \tilde{U}=f'_2,
\end{cases}
\end{equation}
with $f'_1,f'_2 \in L^\infty(0,T;H^s ),$ and $f'_1,f'_2$ have $C^1$ dependence on $(U,\theta)$ verifying:
\[
\norm{(f'_1,f'_2)}_{L^\infty(0,T;H^s )} \leq C (\norm{(\eta,\psi,V,B)}_{L^\infty(0,T;H^{s+\frac{1}{2}}\times H^{s+\frac{1}{2}}\times H^{s}\times H^{s})})
\]
\end{lemma}
\begin{proof}
Again the lemma simply follows from the fact that $I-T_1$ and $\omega(D)$ are regularizing operators.
\end{proof}
\subsubsection{Decomposing the change of variable:} 
Set 
\begin{align*}
\Phi:I_{s,T} \times H^{s} &\rightarrow C^0([0,T];H^s  )&\Phi:I_{s,0} \times H^{s} &\rightarrow H^s \\
(\eta,\psi) & \mapsto ( \tilde{U},\tilde{\theta},aux_1,aux_2)&
(\eta,\psi) & \mapsto ( \tilde{U},\tilde{\theta},aux_1,aux_2)
\end{align*}
The goal is to prove that $\Phi$ is locally invertible and then the proof will follow from Theorem \ref{geometric proof quasi WW_a tech gen_main theorem}. \\
We write $\Phi=\Phi_1 \circ \Phi_2$ with
\begin{align*}
\Phi_1:I_{s,T}&\rightarrow C^0([0,T];H^s \times H^{s-\frac{1}{2}} \times H^s \times H^s )\\
(\eta,\psi) & \mapsto ( \tilde{U}, \tilde{\zeta},aux_1,aux_2)
\end{align*}
and,
\begin{align*}
\Phi_2:C^0([0,T];H^s \times H^{s-\frac{1}{2}} \times H^s \times H^s ) &\rightarrow C^0([0,T];H^s)\\
( \tilde{U}, \tilde{\zeta},aux_1,aux_2) & \mapsto ( \tilde{U}, \tilde{\theta},aux_1,aux_2)
\end{align*}
We define $\Phi_1$ and $\Phi_2$ analogously when $\Phi$ is defined on $I_{s,0}$.
\begin{lemma}\label{geometric proof quasi WW_quasi WW without ST_subsec new change of unkniwns_lem2}%%%%%%%%%%%%ref%%%%%%%%%
There exists $r,r_1,\eps>0$ such that:
\begin{align*}
\Phi_1: \b (0,r) \cap I_{s,T}  &\rightarrow C^0([0,T];H^s \times H^{s-\frac{1}{2}} \times H^s \times H^s )
\end{align*}
is a $C^\infty$ diffeomorphism upon it's image.
\begin{align*}
\Phi_2: \b (0,r_1) \cap C^0([0,T];H^s \times H^{s-\frac{1}{2}} \times H^s \times H^s) &\rightarrow C^0([0,T];H^s)
\end{align*}
is a $C^\infty$ diffeomorphism upon it's image.\\
Analogous result hold when $\Phi$ is defined on $I_{s,0}$.
\end{lemma}
The proof of Theorem \ref{geometric proof quasi WW_introduction_Quasi-linearity of the water Wave system_theorem without surface tension} follows as in the previous section from Corollary \ref{geometric proof quasi WW_a tech gen_cor theorem sys} and the previous Lemma combined with the fact that $\Phi_1(0)=0$ thus we have
\[\b(0,r_1) \cap C^0([0,T];H^s \times H^{s-\frac{1}{2}} \times H^s \times H^s )  \subset \Phi_1 \bigg( \b (0,r) \cap C^0([0,T];H^{s+\frac{1}{2}}\times H^{s+\frac{1}{2}} ) \bigg).\]
Also $\Phi_2(0)=0$ thus there exists $r_2$:
\[\b(0,r_2) \cap C^0([0,T];H^s ) \subset \Phi_2 \bigg(\b (0,r_1) \cap C^0([0,T];H^s \times H^{s-\frac{1}{2}} \times H^s \times H^s\bigg).\]
We now turn to the proof of the lemma.
\begin{proof}
As all of the estimates used are punctual in time thus the proof is the same for $I_{s,T}$ and $I_{s,0}$ and we only write the one for $I_{s,T}$.
We start by $\Phi_1$, first the part $\eta \mapsto (\tilde{\zeta},aux_2)$ is invertible with inverse
\[\fr[\Phi^{-1}_1(\tilde{\zeta},aux_2)](\xi)=\frac{1}{d}\sum_j(1-\omega(\xi))\frac{\fr[\partial_i{\tilde{\zeta}}](\xi)}{i\xi_j}+\omega(\xi)\fr[aux_2].\] 
By the same argument $\psi\mapsto( (1-\omega(D))\nabla \psi,\omega(D)\psi)$ is invertible and we see that $(\tilde{U},aux_1)$ is a perturbation of that map indeed:
\begin{align*}
\norm{( (1-\omega(D))\nabla \psi,\omega(D)\psi)-(\tilde{U},aux_1)}_{\mathcal{L}(H^{s+\frac{1}{2}},H^s)}
&\leq C(\norm{B}_{W^{\frac{1}{2},\infty}})\norm{\eta}_{H^{s+\frac{1}{2}}}\\
&\leq C(\norm{(\eta,\psi)}_{H^{s+\frac{1}{2}}})\norm{\eta}_{H^{s+\frac{1}{2}}}
\end{align*}
thus for $r$ small enough we get the desired result.\\
 Now we turn to $\Phi_2$. This operator is the identity on $\tilde{U},aux_1,aux_2$ thus we only have to work on $\tilde{\theta}$. Put $a_0$ as the Taylor coefficient associated to the solution of the problem (0,0).  Now notice that for $\eps>0$:
 \[T_{\sqrt{a_0}\abs{\xi}^{-\frac{1}{2}}} +\eps(I-T_I):C^0([0,T];H^{s-\frac{1}{2}}  ) \rightarrow C^0([0,T];H^s )\] is a $C^\infty$ diffeomorphism. And now we see that $T_q +\eps(I-T_1) $ is a perturbation of $T_{\sqrt{a_0}\abs{\xi}^{-\frac{1}{2}}}+\eps(I-T_1) $ indeed by symbolic calculus rules:
 \begin{align*}
\norm{T_q -T_{\sqrt{a_0}\abs{\xi}^{-\frac{1}{2}}} }_{\mathcal{L}(H^{s-\frac{1}{2}},H^s)}
								&\leq C(\norm{\eta}_{H^s})\norm{\eta}_{H^s},
\end{align*}
which gives the result by taking $r$ small.
\end{proof}

\appendix
\section{Pseudodifferential and Paradifferential operators}\label{paracomposition_section Notions of microlocal analysis}
In this paragraph we review classic notations and results about pseudodifferential and paradifferential calculus that we need in this paper.
We follow the presentations in \cite{Hormander71} , \cite{Taylor07}, and \cite{Metivier08} 
 which give an accessible presentation. 
\subsection{Notations and functional analysis}\label{paracomposition_section Notations and functional analysis}
We present the definitions of the functional spaces that will be used.\\
We will use the usual definitions and standard notations for the regular functions $C^k$, $C^k_0$ for those with compact support, the distribution space $\d'$,$\er'$ for those with compact support, $\d'^k$,$\er'^k$ for distributions of order k, Lebesgue spaces ($L^p$), Sobolev spaces ($H^s,W^{p,q}$) and the Schwartz class $\sr$ and it's dual $\sr'$. All of those spaces are equipped with their standard topologies.

We also recall the \textit{Landau notation} the expression $O_{\norm{ \ }}(X)$ is used 
to denote any quantity bounded in $\norm{ \ }$ by $C X$, thus $Y=O_{\norm{ \ }}(X)$ is equivalent to $\norm{Y}\leq C X$.

For the definition of the periodic symbol classes we will need the following definitions and notations.
\begin{notation}
We will use $\d$ to denote $\t$ or $\r$ and $\hat{\d}$ to denote their duals that is $\z$ in the case of $\t$ and $\r$ in the case of $\r$. For concision an integral on on $\z^d$ i.e $\displaystyle \int_{\z^d}$ should be understood as $\displaystyle \sum_\z^d$. A function $a$ is said to be in $ C^\infty(\t^d \times \z^d)$ if for every $\xi \in \z$ $a(\cdot,\xi) \in C^\infty(\t^d)$. For $\xi\in \z^d$ and $i \in \set{1,\cdots,d}$, $\partial_{\xi_i}$ should be understood as the partial forward difference operator, i.e
\[\partial_{\xi_i} a(\xi_1,\cdots,\xi_i,\cdots,\xi_d)=a(\xi_1,\cdots,\xi_i+1,\cdots,\xi_d)-a(\xi_1,\cdots,\xi_i,\cdots,\xi_d),\ \xi \in \z^d.\]
We recall the following simple identities for the Fourier transform on the Torus:
\[
\begin{cases}
\fr_{\t^d}(\partial_x^\alpha f)(\xi)=\xi^\alpha \fr_{\t^d}(f)(\xi), \xi \in \z^d,\\
\fr_{\t^d}((e^{-2i\pi x}-1)^\alpha f)(\xi)=\xi^\alpha \fr_{\t^d}(f)(\xi), \xi \in \z^d, \ x\in \t^d.
\end{cases}
\]
\end{notation}
\begin{definition}[Littlewood-Paley decomposition]
Pick $P_0\in C^\infty_0(\r^d)$ so that $P_0(\xi)=1$ for $\abs{\xi}<1$ and 0 for $ \abs{\xi}>2 $. We define a dyadic decomposition of unity by:
\[ \text{ for} k \geq 1, P_{\leq k}(\xi)=\Phi_0(2^{-k}\xi), \ P_k(\xi)=P_{\leq k}(\xi)-P_{\leq k-1}(\xi). \]
 Thus,\[ P_{\leq k}(\xi)=\sum_{0\leq j \leq k}P_j(\xi) \text{ and } 1=\sum_{j=0}^\infty P_j(\xi). \]
 Introduce the operator acting on $\mathscr S '(\r^d)$: 
 \[P_{\leq k}u=\fr^{-1}(P_{\leq k}(\xi)u) \text{ and } u_k=\fr^{-1}(P_k(\xi)u).\]
 Thus,
 \[u=\sum_k u_k.\]
 Finally put $\set{k\geq 1, C_k=\supp  P_k}$ the set of rings associated to this decomposition.
\end{definition}

\begin{remark}
An interesting property of the Littlewood-Paley decomposition is that even if the decomposed function is merely a distribution the terms of the decomposition are regular, indeed they all have compact spectrum and thus are entire functions. On classical functions spaces this regularization effect can be "measured" by the following inequalities due to Bernstein.
\end{remark}

\begin{proposition}[Bernstein's inequalities]\label{paracomposition_Notations and functional analysis_bernstein1}
Suppose that $a\in L^p(\r^d)$ has its spectrum contained in the ball $\set{\abs{\xi}\leq \lambda}$. Then $a\in C^\infty$ and for all $\alpha \in  \n^d$ and $q \geq p$, there is $C_{\alpha,p,q}$ (independent if $\lambda$) such that 
\[\norm{\partial^{\alpha}_x a}_{L^q} \leq C_{\alpha,p,q} \lambda^{\abs{\alpha}+\frac{d}{p}-\frac{d}{q}}\norm{a}_{L^p}.\]
In particular,
\[\norm{\partial^{\alpha}_x a}_{L^q} \leq C_{\alpha} \lambda^{\abs{\alpha}}\norm{a}_{L^p}, \text{ and for $p=2$, $p=\infty$}\]
\[\norm{a}_{L^\infty}\leq C \lambda^{\frac{d}{2}} \norm{a}_{L^2}.\]
\end{proposition}

\begin{proposition}\label{paracomposition_Notations and functional analysis_bernstein2}
For all $\mu >0$, there is a constant $C$ such that for all $\lambda>0$ and for all $\alpha \in W^{\mu,\infty}$ with spectrum contained in $\set{\abs{\xi}\geq \lambda}$. one has the following estimate: 
\[\norm{a}_{L^\infty}\leq C \lambda^{-\mu} \norm{a}_{W^{\mu,\infty}}.\]
\end{proposition}

\begin{definition-proposition}[Sobolev spaces on $\r^d$]
It is also a classical result that for $s\in \r$ :
\[H^s(\r^d)=\set{u\in\sr'(\r^d),\abs{u}_s= \bigg(\sum_q 2^{2qs} {\norm{u_q}_{L^2}}^2 \bigg)^{\frac{1}{2}}<\infty}\]
 with the right hand side equipped with its canonical topology giving it a Hilbert space structure and $\abs{\ }_s$ is equivalent to the usual norm on $\norm{\ }_{H^s}$ .\\
\end{definition-proposition}
\begin{proposition} \label{paracomposition_Notations and functional analysis_proposition Sobolev spaces on balls}%%%%%%%
Let $\b$ be a ball with center 0. There exists a constant C such that for all $s>0$ and for all $(u_q)_{\in \n}\in \sr'(\r^d)$ verifying:
\[\forall q ,\supp \hat{u_q} \subset  2^q \b \text{ and } (2^{qs}\norm{u_q}_{L^2})_{q\in \n} \text{ is in} \ L^2(\n) \]
\[\text{then}, u=\sum_q u_q \in H^s(\r^d) \text{ and } \abs{u}_s \leq \frac{C}{1-2^{-s}} \bigg(\sum_q 2^{2qs} {\norm{u_q}_{L^2}}^2 \bigg)^{\frac{1}{2}}. \]
\end{proposition}
\begin{remark}
The previous definition and properties of the Littlewood-Paley decomposition and Sobolev spaces carries out naturally to $\t^d$.
\end{remark}
Here we recall the usual Kato-Ponce \cite{Kato88} commutator estimates:
\begin{proposition}\label{paracomposition_Notations and functional analysis_KatoPonce commutator estimate}
Consider $s>0$ and $f,g \in H^s$ then 
\[
\norm{[\D^s,f]g}_{L^2}\leq C(\norm{f}_{W^{1,\infty}}\norm{g}_{H^{s-1}}+\norm{f}_{H^s}\norm{g}_{L^\infty}).
\]
\end{proposition}
\subsection{Pseudodifferential operators}
We introduce here the basic definitions and symbolic calculus results. We first introduce the classes of regular symbols.
\begin{definition}\label{paracomposition_Notions of microlocal analysis_Pseudodifferential Calculus_def symbol}
Given $m \in \r,0\leq \rho \leq1$ and $0\leq \sigma \leq1$ we denote the symbol class $S^m_{\rho,\sigma}(\d^d \times \hat{\d}^d)$ the set of all $a\in C^\infty(\d^d \times \hat{\d}^d)$ such that for all multi-orders $\alpha,\beta$ we have the estimate:
\[\abs{\partial^{\alpha}_x\partial^{\beta}_\xi a(x,\xi)}\leq C_{\alpha,\beta}(1+\abs{\xi})^{m-\rho \beta+\sigma \alpha}.\]
$S^m_{\rho,\sigma}(\d^d \times \hat{\d}^d)$ is a Fr\'echet space with the topology defined by the family of semi-norms:
\[M^m_{\alpha,\beta}(a)=\sup_{i\leq \alpha,j\leq \beta}\sup_{\d^d\times\hat{\d}^d}\abs{\partial^{i}_x\partial^{j}_\xi a(x,\xi)(1+\abs{\xi})^{\rho j-m-\sigma i}}.\]
Set \[S^{m}(\d^d \times \hat{\d}^d)=S^m_{1,0}(\d^d \times \hat{\d}^d),\]
\[ S^{-\infty}(\d^d \times \hat{\d}^d)=\bigcap_{m\in \r}S^m(\d^d \times \hat{\d}^d) \text{ and } S^{+\infty}(\d^d \times \hat{\d}^d)=\bigcup_{m\in \r}S^m(\d^d \times \hat{\d}^d)  \]
equipped with their canonically induced topology. 
\end{definition}

For $u \in \sr(\d^d)$ we have 
\begin{equation*} 
\begin{split}
\op(a)u(x) &=(2\pi)^{-d}\int_{\hat{\d}^d}e^{ix.\xi}a(x,\xi)\hat{u}(\xi)d\xi \\
 & = (2\pi)^{-d}\int_{\hat{\d}^d}e^{ix.\xi}a(x,\xi)\int_{\d^d}e^{-iy.\xi}u(y)dy d\xi \\
 & = \int_{\d^d}\bigg((2\pi)^{-n}\int_{\hat{\d}^d}e^{i(x-y).\xi}a(x,\xi) d\xi \bigg)u(y)dy\\
\end{split}
\end{equation*}

Thus giving us the following Proposition.

\begin{proposition}
For $a \in S^m(\d^d \times \hat{\d}^d)$, $\op(a)$ has a kernel K defined by 
\begin{equation} \label{paracomposition_Notions of microlocal analysis_Pseudodifferential Calculus_Kernel equation of a pseudo operator} %%%%%
K(x,y)=(2\pi)^{-d}\int_{\hat{\d}^d}e^{i(x-y).\xi}a(x,\xi) d\xi=(2\pi)^{-n}\fr_\xi a (x,y-x).
\end{equation}

Which can be inverted to give:
\begin{align}%%%%%
a(x,\xi)&=\fr_{y\rightarrow \xi}K(x,x-y)=\int_{\d^d}e^{-iy.\xi}K(x,x-y)dy \nonumber \\
&=(-1)^de^{-ix.\xi}\int_{\d^d}e^{iy.\xi}K(x,y)dy  \label{paracomposition_Notions of microlocal analysis_Pseudodifferential Calculus_Kernel equation of a pseudo operator inverse} 
\end{align}
\end{proposition}

\begin{definition}
Let $m \in \r$, an operator T is said to be of order m if, and only if, for all $\mu \in \r$, it is bounded from $H^\mu(\r^d)$ to $H^{\mu-m}(\r^d)$. 
\end{definition}

\begin{theorem}
If $a \in S^m(\d^d \times \hat{\d}^d)$, then $a(x,D)$ is an operator of order m. Moreover we have the norm estimate:
\[\norm{a(x,D)}_{H^\mu \rightarrow H^{\mu-m}}\leq C M^m_{\mu,m+d/2+1}(a)\]
\end{theorem}
We will now present the main results in symbolic calculus associated to pseudodifferential operators.
\begin{theorem} \label{paracomposition_Notions of microlocal analysis_Pseudodifferential Calculus_theorem symbolic calclus} %%%%%
Let $m,m' \in \r$, $a \in S^m(\d^d \times \hat{\d}^d)$and $b \in S^{m'}(\d^d \times \hat{\d}^d)$. 
\begin{itemize}
\item Composition: Then $\op(a)\circ \op(b)$ is a pseudodifferential operator of order $m+m'$ with symbol $a \#b $ defined by:
\[a \#b(x,\xi)=(2\pi)^{-d}\int_{\d^d \times \hat{\d}^d}e^{i(x-y).(\xi-\eta)} a(x,\eta)b(y,\xi)dyd\eta\]
Moreover,
\[\op(a)\circ \op(b)(x,\xi)-\op(\sum_{\abs{\alpha}<k}\frac{1}{i^{\abs{\alpha}}\alpha!}(\partial^\alpha_\xi a(x,\xi))(\partial^\alpha_x b(x,\xi)))  \text{ is of order $m+m'-k$} \]
for all $k\in \n$.
\item  Adjoint: The adjoint operator of $\op(a)$, $\op(a)^\top$ is a pseudodifferential operator of order m with  symbol $a^\top$ defined by:
\[a^\top(x,\xi)=(2\pi)^{-d}\int_{\d^d\times \hat{\d}^d}e^{-iy.\xi} \bar{a}(x-y,\xi-\eta)dyd\eta\]
Moreover,
\[\op(a^\top)(x,\xi)-\op(\sum_{\abs{\alpha}<k}\frac{1}{i^{\abs{\alpha}}\alpha!}(\partial^\alpha_\xi \partial^\alpha_x  \bar{a}(x,\xi)))  \text{ is of order $m-k$} \]
for all $k\in \n$.
\end{itemize}

\end{theorem}
\begin{definition}
Let $(a_j)\in S^{m_j}(\d^d \times \hat{\d}^d)$ be a series of symbols with $(m_j) \in \r$ decreasing to $-\infty$. We say that $a\in S^{m_0}(\d^d \times \hat{\d}^d)$ is the asymptotic sum of $(a_j)$ if
\[\forall k \in \n, a-\sum_{j=0}^k a_j \in S^{m_{k+1}}(\d^d \times \hat{\d}^d). \]
We denote $a \sim \sum a_j $
\end{definition}

\begin{remark}
We can now write simply:
\[a \#b \sim \sum_{\abs{\alpha}}\frac{1}{i^{\abs{\alpha}}\alpha!}(\partial^\alpha_\xi a(x,\xi))(\partial^\alpha_x b(x,\xi)) \]
and
\[a^\top \sim \sum_{\abs{\alpha}}\frac{1}{i^{\abs{\alpha}}\alpha!}(\partial^\alpha_\xi \partial^\alpha_x  \bar{a}(x,\xi)) .\]
\end{remark}  

Now we present the classic results of change of variables in pseudodifferential operators.
\begin{theorem} \label{paracomposition_section Pull-back of pseudo and para- differential operators_theorem change of variable pseudo} %%%%%
Let $\chi:\d^d \rightarrow \d^d$ be a $C^\infty$ diffeomorphism with $D\chi \in C^\infty_b$ and $A=a(x,D) \in S^m_{loc}(\Omega' \times \r^d)$ a properly supported pseudodifferential operator with kernel K.\\
Then the operator $A^*$  defined by $K^*$ i.e:
\[\forall u \in C^\infty_0(\Omega), A^*u =\int_{\d^d}K(\chi(x),\chi(y))u(y)|detD\chi(y)|dy \] 
is a properly supported pseudodifferential operator with symbol
\[a^*(x,\xi)=(-1)^de^{-ix.\xi}\int_{\d^d\times \hat{\d}^d} a(\chi(x),\eta) e^{i(\chi(x)-\chi(y).\eta+iy.\xi}|det D\chi(y)|dyd\eta \in  S^m(\d^d \times \hat{\d}^d),\]
and verifies
\[(\op(a)u)\circ \chi=\op(a^*)(u\circ \chi).\]
An expansion of $a^*$ is given by:
\begin{equation}\label{paracomposition_section Pull-back of pseudo and para- differential operators_theorem change of variable pseudo eq1}%%%%%%%
a^*(x,\xi) \sim \sum_{\alpha}\frac{1}{\alpha!}\partial^\alpha a(\chi(x),D\chi^{-1}(\chi(x))^\top\xi)P_{\alpha}(\chi(x),\xi),
\end{equation}
where,
\[ P_{\alpha}(x',\xi)=D^\alpha_{y'}(e^{i(\chi^{-1}(y')-\chi^{-1}(x')-D \chi^{-1}(x')(y'-x')).\xi})_{|y'=x'} \]
and $P_{\alpha}$ is polynomial in $\xi$ of degree $\leq \frac{\abs{\alpha}}{2}$, with $P_{0}=1, P_{1}=0$.
\end{theorem}

\subsection{Paradifferential operators}
We start by the definition of symbols with limited spatial regularity. Let $\w\subset \sr'(\d^d)$ be a Banach space.
\begin{definition}
Given $m \in \r$, $\Gamma^m_\w(\d^d)$ denotes the space of locally bounded functions $a(x,\xi)$ on $\d^d\times (\hat{\d}^d \setminus 0)$, which are $C^\infty$ with respect to $\xi$ for $\xi \neq 0$ and such that, for all $\alpha \in \n^d$ and for all $\xi \neq 0$, the function $x \mapsto \partial^\alpha_\xi a(x,\xi)$ belongs to $\w$ and there exists a constant $C_\alpha$ such that,
\[\forall \abs{\xi}>\frac{1}{2}, \norm{\partial^\alpha_\xi a(.,\xi)}_{\w}\leq C_\alpha (1+\abs{\xi})^{m-\abs{\alpha}} \]
\end{definition}
Given a symbol $a$, define the paradifferential operator $T_a$ by
\[\widehat{T_a u}(\xi)=(2\pi)^{-d}\int_{\hat{\d}^d}\theta(\xi-\eta,\eta)\hat{a}(\xi-\eta,\eta)\psi(\eta)\hat{u}(\eta)d\eta,\]
where $\hat{a}(\eta,\xi)=\int e^{-ix.\eta}a(x,\xi)dx$ is the Fourier transform of $a$ with respect to the first variable; $\theta$ and $\psi$ are two fixed $C^\infty$ functions such that:
\[ \psi(\eta)=0 \text{ for } \abs{\eta} \leq 1, \ \psi(\eta)=1 \text{ for } \abs{\eta}\geq 2, \] 
and $\theta(\xi,\eta)$ is homogeneous of degree 0 and satisfies for $0<\eps_1<\eps_2$ small enough,
\[\theta(\xi,\eta)=1 \text{ if } \abs{\xi}\leq \eps_1 \abs{\eta}, \ \theta(\xi,\eta)=0 \text{ if } \abs{\xi}\geq \eps_2 \abs{\eta}.\]

For quantitative estimates we introduce as in \cite{Metivier08}:
\begin{definition}
For $m\in \r$, $\rho \geq 0$ and $a \in \Gamma^m_\w(\d^d)$, we set
\[M^m_\w(a)=\sup_{\abs{\alpha}\leq \frac{d}{2}+1+c} \sup_{\abs{\xi}\geq\frac{1}{2}}\norm{(1+\abs{\xi})^{m-\abs{\alpha}}\partial^\alpha_\xi a(.,\xi)}_{\w}, \text{ with } c>0.\]
We will essentially work with $\w=W^{\rho,\infty}$ and write $M^m_{W^{\rho,\infty}}(a)=M^m_\rho(a)$ with $c=\rho$.
\end{definition}
The main features of symbolic calculus for paradifferential operators are given by the following Theorems.
\begin{theorem}
Let $m \in \r$. if $a\in \Gamma^m_0(\d^d)$, then $T_a$ is of order m. Moreover, for all $\mu \in \r$ there exists a constant K such that
\[\norm{T_a}_{H^\mu \rightarrow H^{\mu-m}}\leq K M^m_0(a).\]
\end{theorem}

\begin{theorem} \label{paracomposition_Notions of microlocal analysis_Paradifferential Calculus_symbolic calculus para} %%%%%
Let $m,m' \in \r$, and $\rho>0$, $a \in \Gamma^m_\rho(\d^d)$and $b \in \Gamma^{m'}_\rho(\d^d)$. 
\begin{itemize}
\item Composition: Then $T_a T_b$ is a paradifferential operator of order $m+m'$ and $T_a T_b- T_{a\#b}$ is of order $m+m'-\rho$ where $a \#b $ is defined by:
\[a \#b=\sum_{\abs{\alpha}<\rho }\frac{1}{i^{\abs{\alpha}}\alpha!} \partial^\alpha_\xi a \partial^\alpha_x b \]
Moreover, for all $\mu \in \r$ there exists a constant K such that
\[ \norm{T_aT_b-T_{a\#b}}_{H^\mu \rightarrow H^{\mu-m-m'+\rho}} \leq K M^m_\rho (a) M^{m'}_\rho(b). \]

\item  Adjoint: The adjoint operator of $T_a$, $T_a^\top$ is a paradifferential operator of order m with  symbol $a^\top$ defined by:
\[a^\top=\sum_{\abs{\alpha}<\rho} \frac{1}{i^{\abs{\alpha}}\alpha!}\partial^\alpha_\xi \partial^\alpha_x \bar{a} \]
Moreover, for all $\mu \in \r$ there exists a constant K such that
\[ \norm{T_a^\top-T_{a^\top}}_{H^\mu \rightarrow H^{\mu-m+\rho}} \leq K M^m_\rho (a). \]
\end{itemize}
\end{theorem}

If $a=a(x)$ is a function of $x$ only, the paradifferential operator $T_a$ is called a paraproduct. 
It follows from Theorem \ref{paracomposition_Notions of microlocal analysis_Paradifferential Calculus_symbolic calculus para} and the Sobolev embeddings that:
\begin{itemize}
\item If $a \in H^\alpha(\d^d)$ and $b \in H^\beta(\d^d)$ with $\alpha,\beta>\frac{d}{2}$, then
\[T_aT_b-T_{ab} \text{ is of order } -\bigg( min\set{\alpha,\beta}-\frac{d}{2} \bigg).\]
\item If $a \in H^\alpha(\d^d)$ with $\alpha>\frac{d}{2}$, then
\[T_a^\top-T_{a^\top} \text{ is of order } -\bigg(\alpha-\frac{d}{2} \bigg).\]
\item If $a \in W^{r,\infty}(\d^d)$, $r\in \n$ then:
\[\norm{au-T_au}_{H^r} \leq C \norm{a}_{W^{r,\infty}} \norm{u}_{L^2}.\]
\end{itemize}

An important feature of paraproducts is that they are well defined for function $a=a(x)$ which are not $L^\infty$ but merely in some Sobolev spaces $H^r$ with $r<\frac{d}{2}$.
\begin{proposition}
Let $m>0$. If $a\in H^{\frac{d}{2}-m}(\d^d)$ and $u \in H^\mu(\d^d)$ then $T_au \in  H^{\mu-m}(\d^d)$. Moreover,
\[ \norm{T_a u}_{H^{\mu -m}}\leq K \norm{a}_{H^{\frac{d}{2} -m}}\norm{u}_{H^{\mu}} \]
\end{proposition}

A main feature of paraproducts is the existence of paralinearisation Theorems which allow us to replace nonlinear expressions by paradifferential expressions, at the price of error terms which are smoother than the main terms.

\begin{theorem} \label{paracomposition_Notions of microlocal analysis_Paradifferential Calculus_paralinearisation para product} %%%%%
Let $\alpha, \beta \in \r $ be such that $\alpha,\beta> \frac{d}{2}$, then
\begin{itemize}
\item Bony's Linearization Theorem\footnote{In our recent work \cite{Ayman18} we give a generalization to this Theorem.} For all $C^\infty$ function F, if $a \in H^\alpha (\d^d)$ then
\[ F(a)- F(0)-T_{DF(a)}a \in H^{2\alpha-\frac{d}{2}} (\d^d). \]
\item If $a\in H^\alpha(\d^d)$ and $b\in H^\beta(\d^d)$, then $ab-T_ab-T_ba \in H^{\alpha+ \beta-\frac{d}{2}} (\d^d)$. Moreover there exists a positive constant K independent of a and b such that:
\[\norm{ab-T_ab-T_ba}_{H^{\alpha+ \beta-\frac{d}{2}} }\leq K  \norm{a}_{H^\alpha} \norm{b}_{H^\beta}  .\]
\end{itemize}
\end{theorem}

\subsection{Paracomposition}
We recall the main properties of the paracomposition operator first introduced by S. Alinhac in \cite{Alinhac86} to treat low regularity change of variables. Here we present the results we reviewed and generalized in some cases in \cite{Ayman18}.
\begin{theorem} \label{paracomposition_section Paracomposition_subsec paracomp on the euclidean space_theorem defintion of paracomposition}%%%%%%%%
 Let $\chi:\d^d \rightarrow \d^d$ be a $W^{1+r,\infty}_{loc}$ diffeomorphism with $D\chi \in W^{r,\infty}$, $r>0, r\notin \n$ and take $s \in \r$ then the following maps are continuous:  
   \begin{align*}
   H^s(\d^d) &\rightarrow H^s(\d^d)\\
  u &\mapsto \chi^* u=\sum_{k\geq 0}  \sum_{\substack{l\geq 0 \\ k-N \leq l \leq k+N}}P_l(D)u_k\circ \chi,
\end{align*} 
where $N \in \n$ is chosen such that $2^{N}>sup_{k,\d^d} \abs{\Phi_k D\chi}^{-1}$ and $2^{N}>sup_{k,\d^d} \abs{\Phi_k D\chi}$.

Taking $\tilde{\chi}:\d^d \rightarrow \d^d$ a $C^{1+\tilde{r}}$ diffeomorphism with $D\chi \in W^{\tilde{r},\infty}$ map with $\tilde{r}>0$, then the previous operation has the natural fonctorial property:
\[\forall u \in H^s(\d^d) , \chi^* \tilde{\chi}^* u= ({\chi \circ \tilde{\chi}})^* u +Ru,\]   
 \[\text{with, }  R:H^{s}(\r^d) \rightarrow H^{s+min(r,\tilde{r})}(\r^d) \text{ continous}.\]
  \end{theorem}
We now give the key paralinearization theorem taking into account the paracomposition operator. 
\begin{theorem}  \label{paracomposition_section Paracomposition_subsec paracomp on the euclidean space_theorem paralinearisation of composition}%%%%%%%%
 Let $u$ be a $W^{1,\infty}(\d^d)$ map and $\chi:\d^d \rightarrow \d^d$ be a $W^{1+r,\infty}_{loc}$ diffeomorphism with $D\chi \in W^{r,\infty}$, $r>0, r\notin \n$. Then:
 \[u \circ \chi(x)=\chi^* u(x)+ T_{Du\circ \chi}\chi(x)+ R_0(x)+R_1(x)+R_2(x)\]
 where the paracomposition given in the previous Theorem verifies the estimates:
 \[\forall s \in \r, \norm{\chi^* u(x)}_{H^s}\leq C(\norm{D\chi}_{\infty})\norm{u(x)}_{H^s},\]
 \[u'\circ \chi \in  \Gamma^0_{W^{0,\infty}(\d^d)}(\d^d) \text{ for $u$  Lipchitz,}   \]
and the remainders verify the estimates: 

\[  \norm{R_0}_{H^{1+r +min(1+\rho,s-\frac{d}{2})}} \leq C\norm{D\chi}_{r}\norm{u}_{H^{1+s}} \]
	\[ \norm{R_1}_{H^{1+r+s}} \leq C(\norm{D\chi}_{\infty}) \norm{D\chi}_{r}\norm{u}_{H^{1+s}}. \]
	\[ \norm{R_2}_{H^{1+r+s}} \leq C(\norm{D\chi}_{\infty},\norm{D\chi^{-1}}_{\infty}) \norm{D\chi}_{r}\norm{u}_{H^{1+s}}. \]
	
	Finally the commutation between a paradifferential operator $a \in \Gamma^m_{\beta}(\d^d)$ and a paracomposition operator $\chi^*$ is given by the following
\[
	 \chi^* T_a u =T_{a^*} \chi^* u+T_{{q}^*} \chi^* u  \text{ with } q \in \Gamma^{m-\beta}_{0}(\d^d),
\]
where $a^*$ has the local expansion:
\begin{equation}\label{paracomposition_section Pull-back of pseudo and para- differential operators_theorem change of variable para eq1}%%%%%%%%
a^*(x,\xi) \sim \sum_{\substack{ \alpha \\ \abs{\alpha}\leq \lfloor min(r,\rho) \rfloor}}\frac{1}{\alpha!}\partial^\alpha a(\chi(x),D\chi^{-1}(\chi(x))^\top\xi)P_{\alpha}(\chi(x),\xi)\in \Gamma^m_{\min(r,\beta)}(\d^d),
\end{equation}
where,
\[ P_{\alpha}(x',\xi)=D^\alpha_{y'}(e^{i(\chi^{-1}(y')-\chi^{-1}(x')-D \chi^{-1}(x')(y'-x')).\xi})_{|y'=x'} \]
and $P_{\alpha}$ is polynomial in $\xi$ of degree $\leq \frac{\abs{\alpha}}{2}$, with $P_{0}=1, P_{1}=0$.
\end{theorem} 
\begin{remark} \label{paracomposition_section Paracomposition_subsec paracomp on the euclidean space_rem simplest example}%%%%%%
The simplest example for the paracomposition operator is when $\chi(x)=Ax$ is a linear operator and in that case we see that if $N$ is chosen sufficiently large in the definition:
\[u(Ax) = (Ax)^*u,\text{ and } T_{u'(Ax)}Ax = 0.\]
\end{remark}

\section{Energy estimates and well-posedness of some pulled back hyperbolic equations}\label{geometric proof quasi WW_Appendix energy est of pulled back eq}
\begin{theorem} \label{geometric proof quasi WW_Appendix energy est of pulled back eq_theorem on para eq}%%%%%%%%%ref%%%%%%%%%%
Let $T>0$, $\chi \in W^{1,\infty}([0,T], W^{1,\infty}_{loc}(\d^d))$ with $D_x\chi \in L^{,\infty}([0,T],L^{\infty}(\d^d))$ and consider $(a_t)_{t\in \r}$ a family of symbols in $\Gamma^\beta_1(\d^d)$ with $\beta \in \r$, such that
 $t \mapsto a_t$ is continuous and bounded from $\r$ to $\Gamma^\beta_1(\d^d)$
  and such that $ Re(a_t)=\frac{a_t+a_t^\top}{2}$ is bounded in $\Gamma^0_1(\d^d)$.
  Suppose moreover that $\chi(t,\cdot)$ is a diffeomorphism between open sets of $\d^d$ and that we have the bounds:
  \begin{equation}\label{geometric proof quasi WW_Appendix energy est of pulled back eq_theorem on para eq_eq 1}
  \exists C>0, \forall t \leq T,\forall x, \ C^{-1}\leq \abs{ D_x \chi(t,x)}\leq C.\\
 \end{equation}
Put $(\cdot)^*  $ is the change of variables by $\chi$ as presented in Theorem \ref{paracomposition_section Paracomposition_subsec paracomp on the euclidean space_theorem paralinearisation of composition}.\\
   Then for all initial data $u_0 \in H^s(\d^d)$ and $f \in  C^0([0,T];H^s(\d^d))$ the Cauchy problem:
 \begin{equation} \label{geometric proof quasi WW_Appendix energy est of pulled back eq_theorem on para eq_eq 2} %%%%%%%%%ref%%%%%%%%%%
\begin{cases}
 \partial_t u+T_{a^*}u=f\\
 \forall x \in \d^d, u(0,x)=u_0(x)
\end{cases}
\end{equation}
has a unique solution $u \in C^0([0,T];H^s(\d^d))\cap C^1([0,T];H^{s-\beta}(\d^d))$ which verifies the estimates:
  \begin{align}\label{geometric proof quasi WW_Appendix energy est of pulled back eq_theorem on para eq_eq 3}
\norm{u(t)}_{H^s}&\leq e^{C(\norm{D_x\chi}_{L^\infty L^\infty},\norm{D_x\chi^{-1}}_{L^\infty L^\infty},M_0^0(Re(a)))t}\norm{u_0}_{H^s}\\
&+2\int_0^te^{C(\norm{D_x\chi}_{L^\infty L^\infty},\norm{D_x\chi^{-1}}_{L^\infty L^\infty},M_0^0(Re(a)))(t-t')}\norm{f(t')}_{H^s}dt'. \nonumber
\end{align}
%where C depends on the finite symbol semi-norm $M_{1}^0(Re(a_t))$.
Again fixing the initial data at 0 is an arbitrary choice.
More precisely, $\forall 0 \leq t_0\leq T$ and all data \\ $u_0 \in H^s(\d^d)$  the Cauchy problem:
 \begin{equation} \label{geometric proof quasi WW_Appendix energy est of pulled back eq_theorem on para eq_eq 4} %%%%%%%%%ref%%%%%%%%%%
\begin{cases}
 \partial_t u+T_{a^*}u=f\\
 \forall x \in \d^d, u(t_0,x)=u_0(x)
\end{cases}
\end{equation}
has a unique solution $u \in C^0([0,T];H^s(\d^d))\cap C^1([0,T];H^{s-\beta}(\d^d))$ which verifies the estimate:
\begin{align*}
\norm{u(t)}_{H^s}&\leq e^{C(\norm{D_x\chi}_{L^\infty L^\infty},\norm{D_x\chi^{-1}}_{L^\infty L^\infty},M_0^0(Re(a)))\abs{t-t_0}}\norm{u_0}_{H^s}\\
&+2\abs{\int_{t_0}^te^{C(\norm{D_x\chi}_{L^\infty L^\infty},\norm{D_x\chi^{-1}}_{L^\infty L^\infty},M_0^0(Re(a)))(t-t')}\norm{f(t')}_{H^s}dt'}.
\end{align*}
\end{theorem} 
\begin{proof} 
The existence of a solution follows from standard compacity arguments after regularization given the priory estimates \eqref{geometric proof quasi WW_Appendix energy est of pulled back eq_theorem on para eq_eq 3}. 
Also, the equation being linear those estimates give the unicity immediately. Thus we will only show the desired priory estimates. \\
Put $\Gamma_s= \D^s$, we will compute $\frac{d}{dt}({ \Gamma_s }^* u,{ \Gamma_s }^* u)_{L^2(\d^d,\abs{D_x\chi(t,x)}dx)}$ in two different ways.
\begin{itemize}
\item \textbf{Method 1.}
 First notice that by Theorem \ref{paracomposition_section Paracomposition_subsec paracomp on the euclidean space_theorem paralinearisation of composition}
\[
{ \Gamma_s }^*(x,\xi) \sim ([D\chi^{-1}(t,\chi(t,x))]^t\xi)^s+R
\]
Where R is of order $s-1$.\\ 
Thus using the lower and upper bound on $\abs{D\chi(t,x)}$ combined with upper bound on $\frac{d}{dt}\abs{D\chi(t,x)}$ we have:
\begin{align*}
&C(\norm{D_x\chi^{-1}}_{L^\infty L^\infty})\frac{d}{dt}[( \Gamma_s u, \Gamma_s u)_{L^2}]-C(\norm{D_x\chi}_{L^\infty L^\infty})\norm{ \Gamma_s u}^2_{L^2}
\\
&\leq\frac{d}{dt}({ \Gamma_s }^* u,{ \Gamma_s }^* u)_{L^2(\abs{D_x\chi(t,x)}dx)}.
\end{align*}
\item \textbf{Method 2.} Now we use the PDE,
 \begin{align*}
&\frac{d}{dt}({ \Gamma_s }^* u,{ \Gamma_s }^* u)_{L^2(\abs{D_x\chi(t,x)}dx)}\\
										&=2Re((\partial_t { \Gamma_s }^* u, { \Gamma_s }^* u))_{L^2(\abs{D\chi(t,x)}dx)})
										+({ \Gamma_s }^* u,{ \Gamma_s }^* u)_{L^2(\frac{d}{dt}\abs{D\chi(t,x)}dx)}\\
										&= -2 Re(( { \Gamma_s }^*T_{a^*} u, { \Gamma_s }^* u)_{L^2(\abs{D\chi(t,x)}dx)})
										+2Re(({ \Gamma_s }^*  f,{ \Gamma_s }^* u))_{L^2(\abs{D\chi(t,x)}dx)}\\
										&+2Re(([\partial_t { \Gamma_s }^*] u, { \Gamma_s }^* u)_{L^2(\abs{D\chi(t,x)}dx)})
										+({ \Gamma_s }^* u,{ \Gamma_s }^* u)_{L^2(\frac{d}{dt}\abs{D\chi(t,x)}dx)}									
\intertext{by change of variables,}
&\frac{d}{dt}({ \Gamma_s }^* u,{ \Gamma_s }^* u)_{L^2(\abs{D_x\chi(t,x)}dx)}\\
&= -2 Re(( { \Gamma_s }^*T_{a^*} u\circ \chi^{-1}, { \Gamma_s }^* u\circ \chi^{-1})_{L^2})
										+2Re(({ \Gamma_s }^*  f,{ \Gamma_s }^* u))_{L^2(\abs{D\chi(t,x)}dx)}\\
										&+2Re(([\partial_t { \Gamma_s }^*] u, { \Gamma_s }^* u)_{L^2(\abs{D\chi(t,x)}dx)})
										+({ \Gamma_s }^* u,{ \Gamma_s }^* u)_{L^2(\frac{d}{dt}\abs{D\chi(t,x)}dx)}.
\end{align*}											
Now notice that,
\[\RE\bigg(\int T_{D[{ \Gamma_s }^*T_{a^*} u\overline{{ \Gamma_s }^*u}]\circ \chi^{-1}}\chi^{-1}dx\bigg)=\int T_{D[{ \Gamma_s }T_{\RE(a)} u\overline{{ \Gamma_s }u}]\circ \chi^{-1}}\chi^{-1}dx+R,\]							
where $R$ verifies by Theorem \ref{paracomposition_section Paracomposition_subsec paracomp on the euclidean space_theorem paralinearisation of composition}:
\begin{equation}\label{geometric proof quasi WW_Appendix energy est of pulled back eq_theorem on para eq_proof eq 1}
\abs{R}\leq C(\norm{D\chi^{-1}}_{L^\infty L^\infty})
														\norm{ \Gamma_s u}^2_{L^2}.
\end{equation}
Thus by Theorem \ref{paracomposition_section Paracomposition_subsec paracomp on the euclidean space_theorem paralinearisation of composition}:
\begin{align}\label{geometric proof quasi WW_Appendix energy est of pulled back eq_theorem on para eq_proof eq 6}
&\frac{d}{dt}({ \Gamma_s }^* u,{ \Gamma_s }^* u)_{L^2(\abs{D_x\chi(t,x)}dx)}\\
										&= -2 ({ \Gamma_s }T_{Re(a)} [(\chi^{-1})^*u], { \Gamma_s }[ (\chi^{-1})^*u])_{L^2}
										+\RE\bigg(\int T_{D[{ \Gamma_s }^*T_{a^*} u\overline{{ \Gamma_s }^*u}]\circ \chi^{-1}}\chi^{-1}dx\bigg)+R\nonumber\\
										&+2Re(({ \Gamma_s }^*  f,{ \Gamma_s }^* u)_{L^2(\abs{D\chi(t,x)}dx)})
										+2Re(([\partial_t { \Gamma_s }^*] u, { \Gamma_s }^* u)_{L^2(\abs{D\chi(t,x)}dx)})
										\nonumber\\
										&+({ \Gamma_s }^* u,{ \Gamma_s }^* u)_{L^2(\frac{d}{dt}\abs{D\chi(t,x)}dx)},\nonumber\\
										&= -2 ({ \Gamma_s }T_{Re(a)} [(\chi^{-1})^*u], { \Gamma_s }[ (\chi^{-1})^*u])_{L^2}
										+\int T_{D[{ \Gamma_s }T_{\RE(a)} u\overline{{ \Gamma_s }u}]\circ \chi^{-1}}\chi^{-1}dx+R\nonumber\\
										&+2Re(({ \Gamma_s }^*  f,{ \Gamma_s }^* u)_{L^2(\abs{D\chi(t,x)}dx)})
										+2Re(([\partial_t { \Gamma_s }^*] u, { \Gamma_s }^* u)_{L^2(\abs{D\chi(t,x)}dx)})
										\nonumber\\
										&+({ \Gamma_s }^* u,{ \Gamma_s }^* u)_{L^2(\frac{d}{dt}\abs{D\chi(t,x)}dx)},\nonumber
\end{align}
Now we have
\begin{equation}\label{geometric proof quasi WW_Appendix energy est of pulled back eq_theorem on para eq_proof eq 2}
\abs{\int T_{D[{ \Gamma_s }T_{\RE(a)} u\overline{{ \Gamma_s }u}]\circ \chi^{-1}}\chi^{-1}dx}
\leq C(\norm{D\chi^{-1}}_{L^{\infty}})
														\norm{ \Gamma_s u}^2_{L^2}.
\end{equation}														
By the upper bound on $\abs{D\chi^{-1}(t,x)}$:
\begin{align}\label{geometric proof quasi WW_Appendix energy est of pulled back eq_theorem on para eq_proof eq 3}
&({ \Gamma_s }T_{Re(a)} [(\chi^{-1})^*u], { \Gamma_s }[ (\chi^{-1})^*u])_{L^2(\d^d)}\nonumber\\
%&=( { \Gamma_s }^*T_{Re(a)^*} u, { \Gamma_s }^* u)_{L^2(\abs{D\chi(t,x)}dx)}\\
														&\leq  M_0^0(Re(a))C(\norm{D_x\chi^{-1}}_{L^\infty L^\infty})
														\norm{ \Gamma_s u}^2_{L^2}.
\end{align}
Now by the upper bound on $\frac{d}{dt}\abs{D\chi(t,x)}$ and $\frac{d}{dt}\abs{D\chi^{-1}(t,x)}$  we have:
\begin{equation*}
({ \Gamma_s }^* u,{ \Gamma_s }^* u)_{L^2(\frac{d}{dt}\abs{D\chi(t,x)}dx)}\leq C(\norm{D_x\chi}_{L^\infty L^\infty})\norm{{ \Gamma_s }^*u}^2_{L^2}
\end{equation*}
Now using the upper bound on $\abs{D\chi(t,x)}$:
\begin{equation}\label{geometric proof quasi WW_Appendix energy est of pulled back eq_theorem on para eq_proof eq 7}
({ \Gamma_s }^* u,{ \Gamma_s }^* u)_{L^2(\frac{d}{dt}\abs{D\chi(t,x)}dx)}\leq C(\norm{D_x\chi}_{L^\infty L^\infty},\norm{D_x\chi^{-1}}_{L^\infty L^\infty})\norm{ \Gamma_s u}^2_{L^2}.
\end{equation}
Analogously we get:
\begin{align}
({ \Gamma_s }^*  f,{ \Gamma_s }^* u)_{L^2(\abs{D\chi(t,x)}dx)}
&\leq C(\norm{D_x\chi}_{L^\infty L^\infty(\d^d)},\norm{D_x\chi^{-1}}_{L^\infty L^\infty})\norm{ \Gamma_s u}_{L^2}\norm{ \Gamma_s f}_{L^2},\label{geometric proof quasi WW_Appendix energy est of pulled back eq_theorem on para eq_proof eq 4}\\
([\partial_t { \Gamma_s }^*] u, { \Gamma_s }^* u)_{L^2(\abs{D\chi(t,x)}dx)}&\leq C(\norm{D_x\chi}_{L^\infty L^\infty},\norm{D_x\chi^{-1}}_{L^\infty L^\infty})\norm{ \Gamma_s u}^2_{L^2}.\label{geometric proof quasi WW_Appendix energy est of pulled back eq_theorem on para eq_proof eq 5}
\end{align}
Thus finally we get by combining \eqref{geometric proof quasi WW_Appendix energy est of pulled back eq_theorem on para eq_proof eq 6}, \eqref{geometric proof quasi WW_Appendix energy est of pulled back eq_theorem on para eq_proof eq 1}, \eqref{geometric proof quasi WW_Appendix energy est of pulled back eq_theorem on para eq_proof eq 2}, \eqref{geometric proof quasi WW_Appendix energy est of pulled back eq_theorem on para eq_proof eq 3}, \eqref{geometric proof quasi WW_Appendix energy est of pulled back eq_theorem on para eq_proof eq 7}, \eqref{geometric proof quasi WW_Appendix energy est of pulled back eq_theorem on para eq_proof eq 4}and \eqref{geometric proof quasi WW_Appendix energy est of pulled back eq_theorem on para eq_proof eq 5}:
\begin{equation}
({ \Gamma_s }T_{Re(a)} [(\chi^{-1})^*u], { \Gamma_s }[ (\chi^{-1})^*u])_{L^2(\d^d)}\leq C(\norm{D_x\chi}_{L^\infty L^\infty},\norm{D_x\chi^{-1}}_{L^\infty L^\infty})\norm{ \Gamma_s u}^2_{L^2}
\end{equation}
\end{itemize}

To conclude we combine the computations from both methods and get:
\begin{align*}
\frac{d}{dt}[( \Gamma_s u, \Gamma_s u)_{L^2}]&\leq C(\norm{D_x\chi}_{L^\infty L^\infty},\norm{D_x\chi^{-1}}_{L^\infty L^\infty},M_0^0(Re(a)))\norm{ \Gamma_s u}^2_{L^2}\\
								&+C(\norm{D_x\chi}_{L^\infty L^\infty},\norm{D_x\chi^{-1}}_{L^\infty L^\infty})
								\norm{ \Gamma_s u}_{L^2}\norm{ \Gamma_s f}_{L^2}.
\end{align*}								
The result then follows from the Gronwall Lemma.
\end{proof}
We see that the proof depends essentially on symbolic calculus rules and those still clearly hold in the case of pseudodifferential operators as presented in Appendix \ref{paracomposition_section Notions of microlocal analysis}. 
\begin{theorem} \label{geometric proof quasi WW_Appendix energy est of pulled back eq_theorem on pseudo eq} %%%%%%%%%ref%%%%%%%%%%
Let $T>0$, $\chi \in W^{1,\infty}([0,T], C^{\infty}(\d^d))$ such that $D_x \chi \in C^{\infty}_b(\d^d)$ and consider $(a_t)_{t\in \r}$ a family of symbols in $S^\beta (\d^d)$ with $\beta \in \r$, such that
 $t \mapsto a_t$ is continuous and bounded from $\r$ to $S^\beta(\d^d)$
  and such that $ Re(a_t)=\frac{a_t+a_t^\top}{2}$ is bounded in $S^0(\d^d)$.
  Suppose moreover that $\chi(t,\cdot)$ is a diffeomorphism between open sets of $\d^d$ and that we have the bounds:
  \begin{equation}\label{geometric proof quasi WW_Appendix energy est of pulled back eq_theorem on pseudo eq_eq1}
  \exists C>0, \forall t \leq T,\forall x, C^{-1}\leq \abs{ D_x \chi(t,x)}\leq C.\\
 \end{equation}
Put $(\cdot)^*  $ is the change of variables by $\chi$ as presented in Theorem \ref{paracomposition_section Pull-back of pseudo and para- differential operators_theorem change of variable pseudo}.\\
   Then for all initial data $u_0 \in H^s(\d^d)$ and $f \in  C^0([0,T];H^s(\d^d))$ the Cauchy problem:
 \begin{equation} \label{geometric proof quasi WW_Appendix energy est of pulled back eq_theorem on pseudo eq_eq2} %%%%%%%%%ref%%%%%%%%%%
\begin{cases}
 \partial_t u+\op({a^*})u=f\\
 \forall x \in \d^d, u(0,x)=u_0(x)
\end{cases}
\end{equation}
has a unique solution $u \in C^0([0,T];H^s(\d^d))\cap C^1([0,T];H^{s-\beta}(\d^d))$ which verifies the estimates:
  \begin{align}\label{geometric proof quasi WW_Appendix energy est of pulled back eq_theorem on pseudo eq_eq2}
\norm{u(t)}_{H^s}&\leq e^{C(\norm{D_x\chi}_{L^\infty L^\infty},\norm{D_x\chi^{-1}}_{L^\infty L^\infty})t}\norm{u_0}_{H^s}\\
&+2\int_0^te^{C(\norm{D_x\chi}_{L^\infty L^\infty},\norm{D_x\chi^{-1}}_{L^\infty L^\infty})(t-t')}\norm{f(t')}_{H^s}dt', \nonumber
\end{align}
where C depends also on a finite symbol semi-norm of $Re(a_t)$.
Again fixing the initial data at 0 is an arbitrary choice.
More precisely, $\forall 0 \leq t_0\leq T$ and all data \\ $u_0 \in H^s(\d^d)$  the Cauchy problem:
 \begin{equation} \label{geometric proof quasi WW_Appendix energy est of pulled back eq_theorem on pseudo eq_eq3} %%%%%%%%%ref%%%%%%%%%%
\begin{cases}
 \partial_t u+\op({a^*})u=f\\
 \forall x \in \d^d, u(t_0,x)=u_0(x)
\end{cases}
\end{equation}
has a unique solution $u \in C^0([0,T];H^s(\d^d))\cap C^1([0,T];H^{s-\beta}(\d^d))$ which verifies the estimate:
\begin{align*}
\norm{u(t)}_{H^s}&\leq e^{C(\norm{D_x\chi}_{L^\infty L^\infty},\norm{D_x\chi^{-1}}_{L^\infty L^\infty})\abs{t-t_0}}\norm{u_0}_{H^s}\\
&+2\abs{\int_{t_0}^te^{C(\norm{D_x\chi}_{L^\infty L^\infty},\norm{D_x\chi^{-1}}_{L^\infty L^\infty})(t-t')}\norm{f(t')}_{H^s}dt'}.
\end{align*}
\end{theorem} 

We finally show a general regularizing effect due to integration in time.
\begin{theorem} \label{geometric proof quasi WW_Appendix energy est of pulled back eq_theorem on regularization of elliptic evolution pde} %%%%%%%%%ref%%%%%%%%%%
Consider$(a_t)_{t\in \r}$ a family of symbols in $S^\beta(\d^d)$ with $\beta \in \r$, such that
 $t \mapsto a_t$ is continuous and bounded from $\r$ to $S^\beta(\d^d)$
  and such that $ \RE(a_t)=\frac{a_t+a_t^\top}{2}$ is bounded in $S^0(\d^d)$, and take $T>0$.
   Then for all initial data $u_0 \in H^s(\d^d)$, and $f \in  C^0([0,T];H^s(\d^d))$ the Cauchy problem:
 \begin{equation} 
\begin{cases}
 \partial_t u+op(a)u=f\\
 \forall x \in \d^d, u(0,x)=u_0(x)
\end{cases}
\end{equation}
has a unique solution $u \in C^0([0,T];H^s(\d^d))\cap C^1([0,T];H^{s-\beta}(\d^d))$ which verifies the estimates:
\[\norm{u(t)}_{H^s(\d^d)}\leq e^{Ct}\norm{u_0}_{H^s(\d^d)}+2\int_0^te^{C(t-t')}\norm{f(t')}_{H^s(\d^d)}dt',\]
where C depends on a finite symbol semi-norm $M_{1}^0(\RE(a_t))$.\\
Suppose moreover that $a$ is elliptic that is:
\[ \forall (x,\xi)\in \r^{2d}, \abs{a(x,\xi)}\geq C\langle \xi \rangle^\beta. \]
Then $\forall t \in[0,T]$:
\[\norm{\int_0^t u(s,\cdot)ds}_{H^s}\leq C(\norm{u_0}_{H^{s-1}}+\norm{u_0}_{H^{s-\beta}}+\norm{f}_{L^\infty([0,T],H^{s-1})}+\norm{f}_{L^\infty([0,T],H^{s-\beta})}).\]
\end{theorem}
\begin{proof}
We start by writing:
\begin{align*}
\partial_t u+\op(a)u&=f\\
\intertext{we then apply $\op(a^{-1})$:}
\op(a^{-1})\partial_t u+u&=\op(a^{-1})f+Ru\\
\intertext{with R $\in S^{-1}(\d^d)$,}
\partial_t \op(a^{-1}) u+u&=\op(a^{-1})f+Ru+\op(\partial_ta^{-1})u=\op(a^{-1})f+Ru+\op(\frac{\partial_ta}{a^2})u.
\end{align*}
the proof then follows by integration in time and the usual elliptic estimates.
\end{proof}

\end{document}